\documentclass[reqno,11pt]{amsart}
\usepackage[utf8]{inputenc}
\usepackage{amssymb,soul}
\usepackage[normalem]{ulem}

\usepackage{bbm}
\usepackage{amsthm}
\usepackage{hyperref}
\usepackage{amsmath}
\usepackage[dvips]{epsfig}
\usepackage{color,xcolor}
\usepackage{url}
\usepackage{caption}
\usepackage{subcaption}
\usepackage{epstopdf}
\usepackage{mathrsfs}

\usepackage{mathtools}
\mathtoolsset{showonlyrefs}

\makeatletter\@addtoreset {equation}{section}\makeatother

\usepackage[margin=0.95in]{geometry}

\theoremstyle{plain}
\newtheorem{definition}{Definition}[section]

\newtheorem{theorem}[definition]{Theorem}

\newtheorem{lemma}[definition]{Lemma}

\newtheorem{remark}[definition]{Remark}
\newtheorem{proposition}[definition]{Proposition}
\DeclareMathOperator{\rest}{{\rm Rest}}
\newcommand{\lr}[1]{\langle #1 \rangle}
\newcommand{\distF}{\widetilde{\mathcal{F}}}
\newcommand{\normm}[2]{\left\|#1\right\|_{L^2_{#2}}}
\newcommand{\R}{\mathbb R}

\newcommand{\setsign}{\{+,-\}}
\newcommand{\Suml}[4]{\sum_{\ell_1,\ell_2,\ell_3,\ell_4\in \setsign}}
\newcommand{\wt}[1]{\widetilde{#1}}
\newcommand{\threef}[9]{\wt{#1}^{(#2)}(#3) \wt{#4}^{(#5)}(#6) \wt{#7}^{(#8)}(#9)}
\newcommand{\onef}[3]{\wt{#1}^{(#2)}(#3)}
\newcommand{\threefct}[9]{{#1}^{(#2)}(#3) {#4}^{(#5)}(#6) {#7}^{(#8)}(#9)}
\newcommand{\onefct}[3]{{#1}^{(#2)}(#3)}

\newcommand{\firstpsi}[7]{\overline{{#1}(#2,#3)} {#4}^{(#5)}(#6,#7)}
\newcommand{\scdpsi}[8]{{#1}^{(#2)}(#3,#4) {#5}^{(#6)}(#7,#8)}
\newcommand{\al}{\alpha}

\renewcommand{\Re}{\mathfrak{Re}}
\renewcommand{\Im}{\mathfrak{Im}}
\newcommand{\la}{\langle}
\newcommand{\ra}{\rangle}
\newcommand{\abs}[1]{\left|#1\right|}
\newcommand{\norm}[1]{\left\|#1\right\|}
\newcommand{\Rho}{\rho_{\Gamma_0 }}

\definecolor{bluegreen}{rgb}{0.0, 0.3, 0.9}

\newcommand{\gc}[1]{\color{teal}{#1} \color{black}}
\newcommand{\fp}[1]{\color{bluegreen}{#1} \color{black}}

\newcommand{\cfp}[1]{\color{bluegreen}  {\tt [FP: #1]} \color{black}  }
\newcommand{\cgd}[1]{\color{blue}  {\tt [GD: #1]} \color{black}  }

\title[Quartic KG with an internal mode]{Global dynamics for the 1D quartic Klein-Gordon equation 
\\ with an internal mode}

\author{Gong Chen}

\author{Gael Y. Diebou}

\author{Adilbek Kairzhan}

\author{Fabio Pusateri}

\address[Chen]{School of Mathematics, Georgia Institute of Technology. 686 Cherry Street,
Atlanta, GA}
\email{gc@math.gatech.edu}
\address[Diebou]{Department of Mathematics, University of Toronto, 40 St. George Street, Toronto, Canada}
\email{gaely.diebou@utoronto.ca}

\address[Kairzhan]{School of Computing and Artificial Intelligence, Nazarbayev University, Astana, Kazakhstan}
\email{akairzhan@nu.edu.kz}
\address[Pusateri]{Department of Mathematics, University of Toronto, 40 St. George Street, Toronto, Canada}
\email{fabiop@math.toronto.edu}

\begin{document}

\begin{abstract}

\normalsize 
We study a one dimensional nonlinear Klein–Gordon equation with a potential,
a localized quadratic nonlinearity, and a non-localized quartic nonlinearity.
We assume that the linearized operator 
has a single discrete eigenvalue below the continuous spectrum,
corresponding to an ``internal mode''.
This eigenvalue generates localized, time-periodic solutions for the linear equation.
Assuming the natural Fermi Golden Rule,
we give a global description of the dynamics of the amplitude of the internal mode
through the radiation damping mechanism, and prove scattering for the radiation.
Our approach is based on the distorted Fourier transform
and a collection of refined 
dispersive decay and smoothing estimates. 

The corresponding cubic problem is closely related to the question of 
global-in-space dispersive asymptotics for perturbations of kink solutions in the classical $\phi^4$
model, which remains a challenging open question. 
From the perspective of nonlinear estimates, a quartic (non-localized) interaction is essentially sharp for our analysis, 
and closing the estimates with such a low degree nonlinear term requires 
a careful and delicate treatment. 

As a direct application of our analysis, we prove a result on the asymptotic stability of kinks 
for classes of scalar field models that can be seen as perturbations of the $\phi^4$ model. 
\end{abstract}
\maketitle
\tableofcontents 
\section{Introduction}

\medskip
\subsection{The model}
We consider the following nonlinear Klein-Gordon equation with a potential
\begin{equation}
\label{nlkg-quartic}
\partial_t^2 u - \partial_x^2 u + u + V(x) u = \beta(x) u^2 + u^4, 
\end{equation}
where $(t,x) \in \mathbb{R} \times \mathbb{R}$ are time and space variables, 
$u(t,x) \in \mathbb{R}$ is the unknown solution, $\beta: \mathbb{R} \to \mathbb{R}$ 
is a real-valued function that is sufficiently smooth and decaying as $|x| \to \infty$
together with its derivatives, 
and $V: \mathbb{R} \to \mathbb{R}$ is a real-valued  external potential 
sufficiently smooth and decaying. 
The addition of a localized cubic term $\beta_2(x)u^3$, or quintic and 
higher order nonlinear terms $u^p (p \geq 5)$ does not bring further complications in the analysis; 
hence, we restrict our discussion to the model \eqref{nlkg-quartic} with a localized quadratic 
term and a quartic term $u^4$ for concreteness. 
The corresponding Cauchy problem is subject  
to the initial conditions $u(0)=u_0$ and $\partial_tu(0)=u_1$. 

The Hamiltonian, or the energy, associated with the model \eqref{nlkg-quartic} is 
\begin{equation}
\label{hamiltonian}
\mathcal{H} (u, \partial_t u) := \frac{1}{2} \int_{\mathbb{R}} \left[ (\partial_t u)^2 + (\partial_x u)^2 + u^2 + V(x) u^2 \right] dx - \int_{\mathbb{R}} \left( \frac{1}{3} \beta u^3 + \frac{1}{5} u^5 \right) dx.
\end{equation}

We assume that $V$ is a fast decaying potential and that the spectrum of the linearized operator
\begin{equation}
\label{operator-L-defn}
H+1:= -\partial_x^2 + 1 + V(x)
\end{equation} 
consists of purely absolutely continuous spectrum $[1,\infty)$ 
and a single positive discrete eigenvalue $\lambda^2$, such that 
\begin{equation}\label{eq:lambdaintro}
\left(-\partial_x^2 +1 +V(x)\right) \varphi (x) = \lambda^2 \varphi(x), \quad \|\varphi\|_{L^2_x} = 1
\end{equation}
where $\varphi(x)$ is a spatially localized eigenfunction that decays exponentially as $|x|\to \infty$. 
In addition, we assume that no resonance occurs at the edge of the continuous spectrum, which means that there is no nonzero bounded $\zeta$ such that  $H \zeta =0$ (equivalently, $(H+1)\zeta=\zeta$).
This condition is equivalent to the assumption that $V$ is a \textit{generic} potential.

The presence of a positive eigenvalue $\lambda^2\in (0,1)$ of $H+1=L^2$ in \eqref{operator-L-defn} 
gives rise to a family of spatially localized and time-periodic solutions 
\begin{equation}
\label{time-periodic-linear-solutions}
\phi_{A, \theta} (t, x) = A \cos(\lambda t + \theta) \varphi(x), \quad A, \theta \in \mathbb{R}
\end{equation}
to the linear equation associated with \eqref{nlkg-quartic}, namely
\begin{equation}
\label{linear-nlkg}
\partial_t^2 \phi - \partial_x^2 \phi + \phi + V(x) \phi = 0.
\end{equation}
Such solutions $\phi_{A, \theta}$ are known as \textit{bound states}. 
Upon adding nonlinear terms to \eqref{linear-nlkg}, 
the question of the persistence of small amplitude bound states arises.
In this paper, we answer the question for the nonlinear model \eqref{nlkg-quartic}. 
In particular, we prove that for 
initial data 
that are sufficiently small in a suitable weighted Sobolev space, 
solutions to \eqref{nlkg-quartic} decay over time and, 
consequently, bound states \eqref{time-periodic-linear-solutions} do not persist under the nonlinear flow.

\medskip
\subsection{Motivation and difficulties}\label{ssecmotiv}
Nonlinear equations featuring an external potential result from perturbing nonlinear systems near special solutions.  
Our work is mainly inspired by the question of asymptotic stability of the kink solution $K(x) = \tanh (x/\sqrt{2})$ for the one-dimensional cubic wave equation 
\begin{equation}
\label{cubic-wave-pde}
\partial_t^2 \phi - \partial_x^2 \phi = \phi - \phi^3,    
\end{equation}
also known as the $\phi^4$-model. Considering a perturbed kink $\phi(t,x) = K(x) + u(t,x)$, the problem reduces to understanding global-in-time dynamics of solutions to 
\begin{equation}
\label{perturbed-cubic-wave-pde}
\partial_t^2 u + (-\partial_x^2 + 2 - 3 {\rm sech}^2(x/\sqrt{2})) u = -3K(x) u^2 - u^3
\end{equation}
with an initial data of small size $\varepsilon_0$. The spectrum of its linear operator $\mathcal{L} = -\partial_x^2 + 2 - 3 {\rm sech}^2(x/\sqrt{2})$ consists of discrete eigenvalues $\{0, 3/2\}$ and a continuous spectrum $[2, \infty)$. The presence of the positive eigenvalue produces solutions as in \eqref{time-periodic-linear-solutions}.

To rule out the persistence of non-dispersing solutions in the nonlinear flow, one needs to perform a careful analysis of the interactions 
between the internal mode and the orthogonal radiation field,
and rely on a non-degeneracy coupling condition known as the Fermi Golden Rule (FGR).
The analysis can be quite complicated due to slow ($\sim t^{-1/2}$) point-wise decay rate of linear solutions to the $1$d (perturbed) Klein-Gordon equation, and the singular forcing from the coupling to the discrete spectrum.
The main difficulty in \eqref{perturbed-cubic-wave-pde} is that the radiation, which is strongly affected by the internal mode,
interacts with itself through the cubic nonlinearity,
and such coupling does not generate efficient decay.  
In our model \eqref{nlkg-quartic}, the internal mode affects 
the radiation field in the same way, but the radiation's
self-interaction is quartic;
while this interaction is also problematic, we show that it is sufficiently perturbative to propagate decay and control the whole nonlinear solution.

In the seminal work of  Kowalczyk, Martel and Mu\~noz in \cite{KMM17},  the authors proved asymptotic stability 
of the kink in \eqref{cubic-wave-pde} under odd perturbations in the energy space 
with the error term measured on a compact set. 
Given a localized perturbation, the pointwise global behavior of the error term is open. This question was addressed
by Delort and Masmoudi in \cite{DM22}, who provided explicit decay rates 
of solutions for times of order $O (\varepsilon_0^{-4+})$. 
Both of the papers cited above consider \textit{odd} perturbations of the kink to 
avoid the zero eigenvalue and the resonance in the spectrum of $\mathcal{L}$, which present 
substantial additional challenges. 
See, for example, the recent work of the first author and L\"uhrmann in \cite{CL24}
addressing the zero eigenvalue in the context of the sine-Gordon equation.


Although the construction of global dispersive solutions  to 
\eqref{perturbed-cubic-wave-pde} seems to be out of reach at the moment, 
the quartic model \eqref{nlkg-quartic} considered here
serves as a valuable intermediate step to illustrate the difficulties encountered when dealing with
an internal mode and a low pure power nonlinearity.
We remark that our main result in Theorem \ref{thm:main-theorem} improves  on the previous findings of  Komech-Kopylova in \cite{KK11}, 
where the authors establish global decay for solutions to  
equations of the type \eqref{perturbed-cubic-wave-pde} 
where the nonlinearities are localized in space up to order $12$.
Furthermore, our approach, and in particular the `good-bad' decomposition,
combined with the normal form arguments in Germain-Pusateri \cite{GP22-mathpi}, 
is likely to yield a result analogous to that of \cite{DM22}.
Finally, in Section \ref{sec:application}, we also show applications of our result to perturbations 
of the $\phi^4$ model where the error term satisfies an equation 
with localized quadratic and cubic terms but a non-localized quartic term.


\subsection{Related literature}
Internal modes are known to appear when investigating 
questions of stability of special solutions to many dispersive PDEs. 
Restricting our attention to $1d$ models, in \cite{PKA98} Pelinovsky-Kivshar-Afanasjev provide a class of Schr\"odinger-type models whose linearization around their solitary wave solutions possesses an internal mode. 
For the nonlinear Schr\"odinger equation with a nonlinearity vanishing to order at least $9$,
the full asymptotic stability of the solitary wave was proven by Buslaev and Sulem in \cite{BS03} under a certain symmetry condition. 
A more recent paper by Martel \cite{M24}
treats a cubic-quintic Schr\"odinger 
equation, while Rialland \cite{R25} considers perturbations of the cubic NLS in the finite energy setting. 
We refer to the review papers \cite{KMM17-review} and \cite{G25} for more details about the latest developments 
on the asymptotic stability of nonlinear waves in dispersive PDE models.

Internal modes in three spatial dimensions 
have been considered by several authors. 
The first rigorous result was obtained for the cubic Klein-Gordon
$\partial_t^2u + (-\Delta + V(x) +m^2) u = u^3$ 
in the seminal paper by Soffer and Weinstein \cite{SW99}. 
It was shown that the internal mode component decays under the assumption of a Fermi Golden Rule (FGR), 
that is a nonlinear resonance condition introduced by Sigal in \cite{S93}. 
In fact, the FGR plays a crucial role in the damping mechanism of the internal mode component;
the internal mode turns out to have a very slow time-decay rate (meta-stability) 
which is essentially due to a resonant nonlinear interaction between the discrete modes and the radiation term. 
The presence of multiple internal modes in the cubic Klein-Gordon equation 
was studied by Bambusi and Cuccagna in \cite{BC11}.
These results were further extended and complemented in recent important works by
Lei-Liu-Yang \cite{LLY22,LLY23}. 

A quadratic Klein-Gordon equation with an internal mode in 2d arises in the study of the stability of the one-vortex for the Abelian Yang-Mills-Higgs model, and has been treated very recently by L\"{u}hrmann- Palacios-Pusateri-Schlag-Shahshahani \cite{LPPSS1}-\cite{LPPSS2}-\cite{LPPSS3}.

We also mention that the asymptotic stability of $1$d kinks in the $3d$ version of the $\phi^4$-model 
was proven by Cuccagna in \cite{C08} for a class of compactly supported perturbations.
The results of Soffer and Weinstein in \cite{SW99} were further extended to 
the quadratic Klein-Gordon equation by L\'eger and Pusateri in \cite{LP21, LP22}.   
The problem of meta-stability for a class of nonlinear Schr\"odinger equations was studied by
Tsai-Yau in \cite{TY02}; see also the survey by Cuccagna-Maeda \cite{CM20} 
for a more thorough discussion of recent advances in the NLS case.

\medskip 
\subsection{Informal statement of results}\label{sec:simplified_MR}
Here we state a simplified version of the main result of this paper on the long-time dynamics of small solutions to \eqref{nlkg-quartic}-\eqref{eq:lambdaintro}.
We first briefly introduce some notations, state the main spectral assumptions, and refer to Subsection \ref{ssec:notations} for a more complete set of notations.
Denote $L:=\sqrt{H+1}$,  $P_c$ the projection onto the continuous spectral subspace of $L^2(\R)$ relative to $H$ and $(\cdot,\cdot)$ the scalar product on $L^2(\R)$. 

\subsubsection*{Assumptions}
We assume that the real-valued potential $V$ and the coefficient $\beta$
belong to the class of Schwartz functions (for simplicity, as these assumptions can easily be relaxed
to finite smoothness and algebraic decay of the functions and their derivatives).
In addition, we impose the following conditions:

\setlength{\leftmargini}{2em}

\smallskip
\begin{itemize}
\item (Genericity of $V$) The potential $V$ is generic, that is, $0$ is neither an eigenvalue nor a resonance of $H$.

\smallskip
\item  (Internal mode-continuous spectrum interaction)    
  $\lambda\in (1/2,1)$ and $2\lambda \in \sigma_{c}(L)$ (or, equivalently 
    $4\lambda^2 \in \sigma_{c}(H+1)$).

\smallskip
\item  (The Fermi Golden Rule) The following nondegeneracy condition holds:
\begin{equation}\label{Fermi-Gordon}
\Gamma := \frac{\pi}{2\lambda} \left( \theta, \delta(L-2\lambda) \theta \right) > 0, \, \quad \theta := P_c(\beta\varphi^2).
\end{equation}
\end{itemize}
We are now in a position to state a simplified version of our main result.

\begin{theorem}\label{thm:main-theorem}
Assume that the above spectral assumptions are satisfied,
and that the initial data
$(u(t=0), \partial_t u(t=0))$ is sufficiently small in a weighted Sobolev space
of sufficiently high regularity $N$ and satisfies 
the ``resonance dominated'' condition from \eqref{mainthmdata}.
Then, there exists a unique solution $u\in C(\R,H^N_x)\cap C^1(\R,H^{N-1}_x)$
of \eqref{nlkg-quartic}
such that 
 \begin{equation}\label{eq:decompu}
u(t,x) = a(t) \varphi (x) + v(t,x), \quad (v(t), \varphi)=0
\end{equation}
with
\begin{equation}
|a(t)| \approx \varepsilon_0 (1+\varepsilon_0^2 t)^{-1/2},\quad t>0
\qquad \mbox{and} \qquad 
\label{v-est-mainthm}
\|v(t)\|_{L^\infty_x}\lesssim \frac {\log(2+\varepsilon_0^2 t)}{\lr{t}^{1/2}}, 
\quad t\geq\varepsilon_0^{-2},
\end{equation}
where $\varepsilon_0$ is as in \eqref{mainthmdata}.
Moreover, the radiation (or field) component $v$ scatters to a solution of 
the associated linear equation in $H^{N-1}_x$.
\end{theorem}

\smallskip
Here we opted for a simple statement of the results and postponed 
the more complete formulation to Section \ref{sec:eqs-mr} and Theorem \ref{thm:mainthmdetail}.
In particular, we will also derive more precise information about the (local) decay of solutions,
provide precise asymptotics for the amplitude of the internal mode
and prove lower bounds on a weighted norm of the radiation profile.


\medskip
\subsection{Main challenges and strategies of the proof}
The study of the long-time dynamics of solutions to dispersive equations in one spatial dimension, especially with 
low-power polynomial nonlinearities and in the presence of internal modes, 
presents well-known significant difficulties. 
Compared to the existing results on internal mode analysis, 
studying the internal mode equation in our setting is significantly more involved because of the slow time decay
and the low power, non-localized, quartic nonlinearity.
In this subsection, we highlight the difficulties encountered in our analysis and discuss 
some of the ideas used to address them.

\subsubsection*{Slow decay of the amplitude of the internal mode} 
Our starting point is the decomposition of $u$ as 
\begin{equation}
u(t)=a(t)\varphi+v(t),\quad (v(t),\varphi)=0,
\end{equation}
where $a(t)$ is the amplitude of the internal mode or the discrete component, 
while $v(t) = P_c u$ is the continuous component (or radiation field). 
These components satisfy a strongly coupled nonlinear quartic ODE$-$PDE system
\begin{equation}
\label{ode-pde-system-initial}
\begin{cases}
\begin{aligned}
& \ddot a + \lambda^2 a = \left( \beta (a \varphi + v)^2, \varphi \right) + \left( (a \varphi + v)^4, \varphi \right) \\
& \partial_t^2 v + L^2 v = P_c \left[ \beta (a\varphi + v)^2+  (a\varphi + v)^4 \right].
\end{aligned}
\end{cases}
\end{equation}
Controlling the dynamics of this system for large times is a nontrivial task.
Starting with the ODE, we introduce the slowly varying amplitude 
\begin{align}\label{slow-variable}
\begin{split}
A(t) & := \frac{e^{-i\lambda t}}{2i\lambda}\big( \dot{a}(t) + i \lambda a(t) \big),
\quad a(t) = A(t) e^{i\lambda t} + \overline A(t) e^{-i\lambda t} \quad
  \big( \dot{A}(t) e^{i\lambda t} + \dot{\overline A}(t) e^{-i\lambda t} = 0 \big).   
\end{split}
\end{align}
The dynamics of $A$ is governed by the equation
\begin{equation}
\label{evolution-A-cs}
\begin{aligned}
\dot A(t) & = \frac{1}{2i\lambda} e^{-i\lambda t} \left[ \left( \beta (a\varphi + v)^2, \varphi \right) 
  +  \left( (a\varphi + v)^4, \varphi \right) \right]
\\
& = \frac{1}{2i\lambda} e^{-i\lambda t} a^2(t) \int_\R \beta\varphi^3 dx 
+ \frac{1}{i\lambda} e^{-i\lambda t} a(t) \int_\R \beta \varphi^2v(t) dx 
+\cdots
\end{aligned}
\end{equation}
where the remaining terms are not explicitly written, as they do not contribute to the main discussion here. 
The second term on the right-hand side of \eqref{evolution-A-cs} 
is the one that produces the resonant damping contributions,
as we will explain below. 

To analyze the solution of the PDE,
we consider $q$ solving $\partial_t^2q+L^2q=a^2(t)P_c(\beta\varphi^2)$ with $q(0)=0=q_t(0)$; 
this isolates the quadratic forcing generated by the discrete component.
Using \eqref{slow-variable} and then integration by parts in $s$ with Sokhotsky’s formula (see \eqref{Sokhotsky}), we express
\begin{align}
q(t,x) & = 
 L^{-1} \Im \int_0^te^{i(t-s)L}a^2(s)ds\, P_c(\beta \varphi^2)
 \\
 & = L^{-1} \frac{1}{2i}e^{itL} \int_0^te^{-is(L-2\lambda)} A^2(s)ds\, P_c(\beta \varphi^2)  +\cdots
 \\
 & = \frac{1}{2}e^{2i\lambda t} A(t)^2\left[L^{-1}P.V.\frac{1}{L-2\lambda}P_c(\beta \varphi^2)-i\pi L^{-1}\delta(L-2\lambda)P_c(\beta \varphi^2)\right] +\cdots
\end{align}
in the sense of distributions, with complex conjugates and 
  more regular and faster decaying terms left out. 
After substituting $q$ for $v$ in the ODE \eqref{evolution-A-cs}, we obtain the
equation 
\begin{equation}
 \dot{A}(t)=\frac{1} {2i\lambda} |A(t)|^2A(t)\int_\R \left(L^{-1}P.V.\frac{1}{L-2\lambda}P_c(\beta \varphi^2)
 -\frac{i\pi}{2\lambda} \delta(L-2\lambda)P_c(\beta \varphi^2)\right)P_c(\beta \varphi^2)   dx+\cdots.
\end{equation}
Letting $y=|A|^2$ leads to an ODE of the form $\dot{y}= -(\Gamma/\lambda) y^2+R$
with $\Gamma > 0$ under the Fermi Golden Rule \eqref{Fermi-Gordon}.
 The remainder $R$ consists of many nonlinear terms in $a$ and $v$ (projected along $\varphi$), 
 some of which decay quite slowly; therefore, a standard perturbative approach does not suffice
 and we need to rely, among other things, on refined smoothing estimates
 (see for example Lemma \ref{lem:local-estimate-time-cut-off})
 and on multilinear (distorted) Fourier analysis to control the contributions to
 $v$ appearing on the right-hand side of \eqref{evolution-A-cs}.
 


\medskip
\subsubsection*{Bootstrap setup} A major step in the analysis of the equation governing the behavior of the continuous component $v$ is the choice of suitable norms that can be propagated globally in time. 
In the presence of an external potential, the distorted Fourier transform appears to be a natural tool to use. 
To obtain dispersive estimates, we choose to control  $\partial_k(\wt{\mathcal {F}}f)(t,k)$ in $L^2_k$,
where $f$ is the interaction variable associated with $v$: 
\begin{equation}\label{relation-fwv}
f := e^{-itL}w, \qquad w:=(\partial_t+iL)v, \qquad v=L^{-1}\Im\, w,
\end{equation} 
and $\wt{\mathcal{F}}$ denotes the distorted Fourier transform adapted to the Schr\"odinger operator $H$ 
(see Subsection \ref{ssec:notations} for the definition). 
Using Duhamel's formula for $f$ and taking the distorted Fourier transform, we write (with $\wt{f}:=\wt{\mathcal{F}}f$)
\begin{align}\label{evolution-f-in}
\begin{aligned}
\wt{f}(t,k) &= \wt{f}(0,k) + \int_0^t e^{-is\lr{k}} \wt{\mathcal{F}}\big[P_c \big( \beta (a(s) \varphi + v(s))^2 + (a(s) \varphi + v(s))^4\big) \big](k) ds 
  \\
&=\int_0^te^{-is(\lr{k}-2\lambda)}A^2(s)ds\, \widetilde{\beta \varphi^2}(k) + F(t,k)+\cdots
\end{aligned}
\end{align}
where we are neglecting many contributions for the sake of explanation,  and have defined
\begin{align}
\begin{aligned}\label{F-mu}
 F(t,k) &:=\frac{1}{2^4}\int_0^t
  \iiiint_{\R^4}e^{is(-\lr{k}+\sum_{j=1}^4\lr{k_j})}
  \wt{f}(s,k_1)\wt{f}(s,k_2)\wt{f}(s,k_3)\wt{f}(s,k_4)
  \\
  & \qquad \qquad \qquad \qquad \qquad \qquad \times \frac{\mu(k,k_1,k_2,k_3,k_4)}{\lr{k_1}\lr{k_2}\lr{k_3}\lr{k_4}}dk_1dk_2dk_3dk_4ds, 
  \\
& \mu(k,k_1,k_2,k_3,k_4)= \int_{\R}\overline{\psi(x,k)}\psi(x,k_1)\psi(x,k_2)\psi(x,k_3)\psi(x,k_4)dx
 \end{aligned}
\end{align}
where $\psi(x,k)$ is the distorted planar wave that diagonalizes the Schr\"odinger operator $H$, 
see \eqref{matK} for details.

The main difficulty we encounter here is due to the feedback of the internal mode $a(t)$ 
into the PDE satisfied by $v$ (see \eqref{ode-pde-system-initial}). 
This generates a singular (resonant) contribution in the (distorted) frequency space, 
that is, the first integral term in \eqref{evolution-f-in}. 
Following \cite{LP21}, we localize this term around the resonant
frequency so that the interaction variable $f$ can be split 
as follows:\footnote{Technically, in the definition of $g$ we do not use 
the profile $A$ but a renormalized version, which we call $B$; see the 
exact definition in \eqref{bad-g-defn}. Here, for the sake of explanation 
we gloss over this distinction.}
\begin{equation}\label{profile-decomposition}
\wt{f} = \wt{h} - \wt{g} ,
  \quad \wt{g}(t,k) = - \chi(|\lr{k}-2\lambda|)\int_0^te^{-is(\lr{k}-2\lambda)}A^2(s)ds
  \,\wt{\mathcal{F}}P_c(\beta \varphi^2)
\end{equation}
where $\chi$ is a bump function with small support near the origin. 
We think of $h$ as a `good' component and of $g$ as a `bad' (singular or Fermi) component.
Indeed, one can see that 
\begin{align}\label{introdkwtg}
{\| \partial_k\wt{g}(t,\cdot) \|}_{L^2_k} \lesssim \sqrt{t}, \,  \qquad t\gg1,
\end{align}
by a direct application of the smoothing estimate \eqref{smooth-est}, 
and that this upper bound is sharp (see Lemma \ref{app:growth-f}).
Thus, propagating the bound for $\partial_k(\wt{\mathcal {F}}f)(t,k)$ over time is a nontrivial task in light of the quartic term $F(t,k)$ in \eqref{F-mu}.

After removing the singular term, we look at the component $\wt{h}$ 
and aim to prove that it satisfies an acceptable estimate with slow growth of the form
\begin{align}\label{introdkwth}
{\big\| \partial_k\wt{h}(t) \big\|}_{L^2_k} \lesssim \varepsilon_0^{\alpha}t^{\delta_1}, \qquad t\gg 1,
\end{align}
for some small constants $\alpha, \delta_1 > 0$.  This plays a crucial role in 
the control of the additional natural quantities that we bootstrap, that is:
\begin{enumerate}
    \item the linear decay type norm: $\displaystyle \sup_{0\leq t\leq T}t^{1/2-\delta}\|v(t)\|_{L^\infty_x}$,
    \item the local decay norm: $\displaystyle \sup_{0\leq t\leq T}t^{1-\delta_2}\|\lr{x}^{-2}v(t)\|_{L^\infty_x}$,
    \item the energy norm: $\displaystyle \sup_{0\leq t\leq T}\|v(t)\|_{H^N_x}$,
\end{enumerate}
for some small constants $0<\delta<\delta_2$. 
Note that we do not propagate the sharp $t^{-1/2}$ decay rate for the nonlinear solution.
This may be possible, but would require several additional arguments.

For the first two norms, the terms in the equation of $v$ (see \eqref{eq-v}) 
can be controlled by relying on the linear dispersive decay estimates except for the quartic contribution which requires a more careful and technical analysis.
For the Sobolev norm, the slow decay in one dimension necessitates
a combination of smoothing estimates and weighted Gagliardo–Nirenberg-type inequalities. 

\subsubsection*{The quartic term} 
Our approach to the most delicate estimates for the quartic nonlinearity is based on multilinear analysis in
frequency space.
The main quantity to estimate is 
the $L^2$-norm of $\partial_k F(t,k)$ for $F$ as in \eqref{F-mu}.
After applying $\partial_k$ to \eqref{F-mu}, we try to integrate by parts 
to exploit the oscillations of the phase.
We then need to pay attention to the regularity of the integrand and in particular 
to the regularity of the nonlinear `spectral distribution' $\mu$. 
The decomposition of eigenfunctions into singular and regular
components (see \eqref{psi-singular-part}-\eqref{psi-regular-part}) 
induces an analogous splitting of the nonlinear distribution $\mu:=\mu_S+\mu_R$. 
Let us denote by $F_S$ the singular part of $F$ corresponding to $\mu_S$, 
and concentrate solely on this in the following idealized discussion.

Differentiating in $k$, and using commutation identities or, equivalently, integration by parts, we have
\begin{align}\label{sing-terms}
\begin{aligned}
\partial_k F_S(t,k) \approx \frac{2^{-4}}{\lr{k}}
  \int_0^t\iiiint_{\R^4}e^{is(-\lr{k}+\sum_{j=1}^4\lr{k_j})}
  \lr{k_1}
  \partial_{k_1}\wt{f}(s,k_1)\wt{f}(s,k_2)\wt{f}(s,k_3)
  \\ \times \wt{f}(s,k_4)\frac{T(q)}{\lr{k_1}\lr{k_2}\lr{k_3}\lr{k_4}}dk_{1234}ds+\cdots
\end{aligned}
\end{align}
where $T$ is a $\delta$ distribution and $q=k-k_1-k_2-k_3-k_4$;
note that we are suppressing many other terms, including 
the contributions from other sign combinations,
terms with a $\mathrm{p.v.}$ distribution instead of a $\delta$, 
and we are disregarding the presence of the scattering coefficients.
Switching back to physical space by means of the (flat) inverse Fourier transform,
we schematically obtain, up to similar or lower order terms, 
\begin{equation}\label{fW}
  {\big\| \partial_k F_S(t,\cdot) \big\|}_{L^2_k} \lesssim 
  \int_0^t {\big\|\wt{f}(s)\big\|}_{H^1_k} 
 {\| v(s) \|}^3_{L^\infty_x} \, ds.
\end{equation}
Using our a priori assumptions (see the statement of Theorem \ref{thm:mainthmdetail})
we can then estimate
\begin{equation}\label{fW'}
{\big\| \partial_k F_S(t,\cdot) \big\|}_{L^2_k} \lesssim 
  \int_0^t \lr{s}^{1/2} \big(\varepsilon_0^\alpha \lr{s}^{-1/2+\delta} \big)^3 \, ds
  \lesssim \varepsilon_0^{3\alpha} \lr{t}^{3\delta};
\end{equation}
the actual proof requires relaxing these bounds even further, but \eqref{fW'} 
already suffices for this explanation.
From \eqref{fW'} one can sees that while this suffices to recover the a priori assumption \eqref{introdkwth},
it is not enough on its own to recover the $L^\infty_x$ decay assumption
(through a standard linear estimate). This is one of the main sources of difficulties in this problem.

Further technical analysis is then needed, and we are eventually 
able to obtain an acceptable bound 
 also using 
 an argument based on dyadic decomposition in time and interpolation. 
 As a consequence we can recover the main bootstrap assumption \eqref{introdkwth} 
 and obtain pointwise decay bounds for the flow of the good part, 
 $e^{itL}h$, 
 as well as improved local decay through standard linear estimates. 
 Similar decay estimates can then be obtained for $e^{itL}f$
 eventually closing the full bootstrap argument for the radiation field $v$.

\medskip
\subsection{Structure of the paper}
The paper is organized as follows.
In Section \ref{sec:eqs-mr} we set up the equations for the nonlinear analysis,
explain the damping mechanism, introduce all necessary notation,
and give a precise statement of our main results in Theorem \ref{thm:mainthmdetail}.
In Section \ref{sec:Linear} we give the main supporting results concerning the distorted Fourier transform,
and then use this in Section \ref{sec:linear-decay} to prove a variety of 
linear dispersive and smoothing estimates.
In Section \ref{sec:ODE} we analyze the ODE and derive 
upper and lower bounds for the amplitude of the internal mode.
Sections \ref{sec:PDE-part0}-\ref{sec:PDE-part2} are dedicated to analyzing the PDE for the radiation
and closing the bootstrap assumptions. Section \ref{sec:PDE-part2} is 
the main one dedicated to the quartic terms.
In Section \ref{app:asy} we prove the asymptotic statements on the internal mode
and the radiation, as well as the lower bound on the weighted norm of the radiation's profile.
Finally, in Section \ref{sec:application} we show how our results can be applied to
obtain asymptotic stability on the line for odd perturbations of kink solutions
in models close to the $\phi^4$ equation.

\medskip
\subsection{Conventions and notations}\label{ssec:notations}
We denote by $C$ a generic positive constant 
whose value may vary from line to line and will be specified when needed. 
For non-negative numbers $N,M$, we write $N\lesssim M$ to mean $N\leq CM$ and $N\approx M$ if $N\lesssim M\lesssim N$. 
In addition, for any $R \in \mathbb{C}$, we write $R=O(N)$ if $|R|\leq CN$.
The implicit constant $C$ in the above notation
may depend on the fixed equation and bootstrap parameters,
but not on time or the small data size, unless otherwise stated.
Given a real number $N$, $N+$ (resp. $N-$) denotes a number $M>N$ (resp. $M<N$) that can be arbitrarily close to $N$.

Throughout the manuscript, we use the following bracket notation
\[\lr{x}:=(1+x^2)^{1/2}.\]
We denote by $\Re$ and $\Im$ the real and imaginary parts of a complex number.
If $I\subset \R$ is an interval, we denote its characteristic function by $\mathbbm{1}_I$. 

For the Fourier transform (resp.  the inverse Fourier transform) 
of a Schwartz function $f$, we use the notation $\widehat{f}$ and $\mathcal{F}f$ 
(resp.  $\check{f}$ and  $\mathcal{F}^{-1}f$) interchangeably with
\begin{equation*}
 \begin{aligned}
\mathcal{F}f(k):=  \frac{1}{\sqrt{2\pi}}\int_\R e^{-ixk} f(x)dx   
\\
\mathcal{F}^{-1}f(x):= \frac{1}{\sqrt{2\pi}}\int_\R e^{ixk} f(k)dk.    
 \end{aligned}   
\end{equation*}

Standard notation is used for cutoffs. Let $\phi:\R \rightarrow [0,1]$ be a smooth bump function supported in $[-8/5,8/5]$ and $\phi=1$ on $[-5/4,5/4]$. For $\ell\in \mathbb Z$, define  $\phi_\ell(k)=\phi(2^{-\ell}k)-\phi(2^{-\ell+1}k)$, $k\in \R$. The family $(\phi_{\ell})_{\ell\in \mathbb Z}$ forms a partition of unity, $\sum_{\ell\in \mathbb Z}\phi_\ell(k)=1$ for $k\neq 0$. For a given interval $I\subset \R$, denote
\begin{equation*}
\phi_{I}(k):=\sum_{\ell\in \mathbb Z\cap I}\phi_{\ell}(k),\quad \phi_{\leq r}(k):=\phi_{(-\infty,r]}(k).    
\end{equation*}

We denote by $P_\ell$, $\ell\in \mathbb Z$, 
the Littlewood-Paley projection adapted to the distorted Fourier transform and defined by 
\begin{equation}
   \widetilde{P_\ell f}(k) = \phi_{\ell}(k)\wt{f}(k), \quad    
   \widetilde{P_I f}(k) = \phi_{I}(k)\wt{f}(k).
\end{equation}
We denote by $T$ a positive time and use the interval $[0,T]$ to state  our bootstrap estimates. When decomposing time integrals  (see, for example, \eqref{reductionestQm}) for any $t\in [0,T]$, we decompose the indicator function $\mathbbm{1}_{[0,t]}$ considering the family of functions $(\tau_n)_{n=0}^{L+1}$ with $|L-\log_2(t+2)|<2$ and $\tau_n :\R \rightarrow [0,1]$ and $L\in \mathbb N$ such that
\begin{equation}\label{timeloc}
 \begin{aligned}
 \sum_{n=0}^{L+1}   \tau_n(s)=\mathbbm{1}_{[0,t]},\quad \mbox{supp}\, \tau_0\subset [0,2],\quad \mbox{supp}\,\tau_{L+1}\subset [t/4,t]\\
 \mbox{and}\quad \mbox{supp}\,\tau_n\subset [2^{n-1},2^{n+1}],\quad |\tau'_n(s)|\lesssim 2^{-n} \quad \mbox{for}\quad n=1,...,L.
 \end{aligned}
\end{equation}

\medskip
\subsection{Acknowledgments}
G.C. was partially supported by NSF Grants DMS-2350301 and CAREER-DMS-2540992, by the Simons Foundation MP-TSM-00002258, and by the AMS Stefan Bergman Fellowship. A. K. was supported by the Nazarbayev University Faculty Development Competitive Research Grants Program 040225FD4702.
F. P. and G.Y.D. were supported in part by a start-up grant from the University of Toronto, 
and NSERC grants RGPIN-2018-06487 and RGPIN-2025-06419.



\bigskip
\section{Setup and main results}\label{sec:eqs-mr}
In this section, we first derive equations for the discrete and continuous components of the solutions
and isolate the resonant interactions responsible for radiation damping.
We then state the complete main result of this paper in Theorem \ref{thm:mainthmdetail},
including all the details on the asymptotic behavior which 
were omitted in our earlier simplified Theorem \ref{thm:main-theorem}.

\subsection{Derivation of the equations}
Recall the amplitude $A$ defined in \eqref{slow-variable}:
\begin{equation}
\label{a-in-A-defn}
a(t) = A(t) e^{i\lambda t} + \overline A(t) e^{-i\lambda t} \quad \text{with} \quad \dot A(t) e^{i\lambda t}
+ \dot{\overline A}(t) e^{-i\lambda t} = 0.
\end{equation}
Substituting the decomposition \eqref{eq:decompu} into \eqref{nlkg-quartic}, 
we obtain the equation for the internal mode
\begin{equation}
\label{ode}
\begin{aligned}
& \ddot a + \lambda^2 a = \left( \beta (a \varphi + v)^2, \varphi \right) + \left( (a \varphi + v)^4, \varphi \right).
\end{aligned}
\end{equation}
From \eqref{a-in-A-defn}, the evolution of the profile $A(t)$ is governed by the equation 
\begin{equation}
\label{evolution-A-initial}
\begin{aligned}
\dot A(t) = &~ \frac{1}{2i\lambda} e^{-i\lambda t} \left[ \left( \beta (a\varphi + v)^2, \varphi \right) +  \left( (a\varphi + v)^4, \varphi \right) \right].
\end{aligned}
\end{equation}
Before proceeding we also record (cf. \eqref{relation-fwv}) the equation for the continuous part $v$ of the solution:
\begin{align}\label{eq-v}
\begin{split}
v(t) & = L^{-1} \Im\big[ e^{itL}\big(\partial_tv(0)+iLv(0)\big) \big]
  \\
  & + \int_0^t L^{-1}\Im \Big[ e^{i(t-s)L} P_c \left[ \beta (a(s) \varphi + v(s))^2 
  + (a(s) \varphi + v(s))^4 \right] \Big] ds.
\end{split}
\end{align}

To capture the leading order dynamics in \eqref{evolution-A-initial}, 
we need to isolate the terms responsible for the leading order behavior of the radiation field
$v$, or equivalently, of $w =\left( \partial_t + iL \right) v$. 
This is done in the following lemma, which also organizes some of the nonlinear terms 
according to their features and the way we are going to treat them in the proofs to come.

\begin{lemma}\label{lemevolw}
Let 
$w:=\left( \partial_t + iL \right) v$, where $v$ is the solution of \eqref{ode-pde-system-initial}.
Define $\theta := P_c \left( \beta \varphi^2 \right)$. We have the following decomposition
\begin{align}\label{lemevolw1}
\begin{split}
w(t) = e^{itL}w(0) + i e^{2i\lambda t} A^2(t) {\rm P.V.} \left( \frac{1}{L-2\lambda} \right) \theta + e^{2i\lambda t} A^2(t) \pi \delta(L-2\lambda) \theta
\\
  + e^{itL} R_w(t) + \rest(t)
\end{split}
\end{align}
where the first term on the right-hand side of \eqref{lemevolw1} is the linear evolution,
the second and third terms are the leading order `singular' contributions
and the other terms are:

\setlength{\leftmargini}{2em}
\begin{itemize}

\item Remainder terms from quadratic forcing: 
\begin{align}\label{Rwextra}
R_w(t) := \lim\limits_{\eta \to 0^+} & \Big[ - 
  \frac{i A^2(0)}{L-2\lambda + i\eta} \theta 
  - 2i 
  \int_0^t A(s) \dot{A} (s) \frac{e^{-is (L-2\lambda)} e^{\eta s} }{L-2\lambda + i\eta} \theta~ds \Big]
\\
\label{Rw}
& + 2 \int_0^t |A(s)|^2e^{-isL}  \theta~ds + \int_0^t \overline{A(s)}^2 e^{-is(L+2\lambda)} \theta ds
 \qquad \big( \theta = P_c (\beta \varphi^2 ) \big); 
\end{align}
(The resolvent limits above are understood in the sense of distributions, for each fixed $t$.)

\item Other quadratic terms from discrete-continuous interactions,
the continuous-continuous interactions, and all the quartic interactions:
\begin{align}\label{Rest}
\begin{split}
\rest(t) =  2 \int_0^t e^{i(t-s)L} a(s) 
  P_c  \left( \beta \varphi v(s) \right) ds + \int_0^t e^{i(t-s)L} P_c  
  \left( \beta v^2(s) \right) ds \\ + \int_0^t e^{i(t-s)L} P_c  \left( (a(s) \varphi + v(s))^4 \right) ds.  
\end{split}
\end{align}

\end{itemize}
 
\end{lemma}

\begin{proof}
Recall that  $f:=e^{-itL}w$ so that \eqref{ode-pde-system-initial} gives
\begin{equation}
\label{f-evolution}
\partial_t f = e^{-itL}P_c \left[ \beta (a\varphi + v)^2+  (a\varphi + v)^4 \right],
\end{equation}
and Duhamel's formula reads 
\begin{equation}\label{evolution-f}
f(t) = f(0) + \int_0^t e^{-isL} P_c \left[ \beta (a(s) \varphi + v(s))^2 + (a(s) \varphi + v(s))^4 \right] ds. 
\end{equation}
In terms of $w$, we can write
\begin{equation}
\label{duhamel-for-w}
\begin{aligned}
w(t) & = e^{itL}w(0) +  \int_0^t e^{i(t-s)L} P_c \left[ \beta (a(s) \varphi + v(s))^2 + (a(s) \varphi + v(s))^4 \right] ds\\
& :=  e^{itL}w(0) + \int_0^t e^{i(t-s)L} a^2(s) P_c  \left( \beta \varphi^2 \right) ds + \rest. 
\end{aligned}
\end{equation}
with $\rest$ as in \eqref{Rest}.
We investigate the second term on the right-hand side of \eqref{duhamel-for-w} in more detail. 
To this end, with $\theta= P_c \left( \beta \varphi^2 \right)$ and using \eqref{a-in-A-defn}, we write
\begin{equation}
\label{second-term-w-duhamel}
\begin{aligned}
\int_0^t e^{i(t-s)L} a^2(s)  \theta ds 
& := e^{itL} \int_0^t e^{-is(L-2\lambda)} A^2(s) \theta~ds + e^{itL} R_w',
\end{aligned}
\end{equation}
where $R_w'$ is the sum of the last two integrals appearing in \eqref{Rw}.
Next, regularizing the first term on the right-hand side of \eqref{second-term-w-duhamel}
and integrating by parts in $s$ yields 
\begin{equation}
\label{first-term-integral-Asquared}
\begin{aligned}
& e^{itL} \int_0^t e^{-is(L-2\lambda)} A^2(s) \theta~ds 
\\ & =
  e^{itL} \lim\limits_{\eta \to 0^+} \int_0^t e^{-is(L-2\lambda + i\eta)} A^2(s) \theta~ds 
\\
& = \lim\limits_{\eta \to 0^+} \Bigg[ \frac{i e^{2i \lambda t} e^{\eta t}}{L-2\lambda + i\eta} A^2(t) \theta 
  - i e^{itL} \frac{A^2(0)}{L-2\lambda + i\eta} \theta 
  - 2ie^{itL} \int_0^t \frac{e^{-is (L-2\lambda)} e^{\eta s} }{L-2\lambda + i\eta} A(s) \dot{A} (s) \theta~ds \Bigg].
\end{aligned}
\end{equation}
Using the distributional identity 
\begin{equation}\label{Sokhotsky}
\frac{1}{x \pm i 0^+} = \lim\limits_{\eta \to 0^+} \frac{1}{x \pm i\eta} = {\rm P.V.} \frac{1}{x} \mp i\pi \delta,
\end{equation}
we obtain
\begin{equation*}
\lim\limits_{\eta \to 0^+} \frac{i e^{2i \lambda t} e^{\eta t}}{L-2\lambda + i\eta} A^2(t) \theta 
= i e^{2i\lambda t} A^2(t) {\rm P.V.} \left( \frac{1}{L-2\lambda} \right) \theta + e^{2i\lambda t} A^2(t) \pi \delta(L-2\lambda) \theta. 
\end{equation*}
This is the expression appearing on the right-hand side of \eqref{lemevolw1};
the other two terms on the right-hand side of \eqref{first-term-integral-Asquared}
give the remaining terms in $e^{itL}R_w$ according to the definition in \eqref{Rwextra}.
This concludes the proof of the lemma.
\end{proof}

Returning to the evolution equation for $A$ in \eqref{evolution-A-initial}, 
we can now rewrite and organize the nonlinear terms as follows:

\begin{lemma}
\label{lemevolA}
Let $A$ be a solution of \eqref{evolution-A-initial}. Then we can write 
\begin{equation}\label{evolution-A}
\begin{aligned}
\dot{A}(t) = P(t) + \frac{1}{2i \lambda} (\Lambda - i \Gamma) |A|^2(t) A(t) + R(t),
\end{aligned}
\end{equation}
where:

\setlength{\leftmargini}{2em}

\begin{itemize}
\item $P$ collects the quadratic self-interactions:
\begin{equation}
\label{P-defn}
P(t) := \frac{1}{2i\lambda} e^{-i\lambda t} (A e^{i\lambda t} 
  + \overline{A} e^{-i \lambda t})^2 \int_\R \beta 
  \varphi^3 dx;
\end{equation}

\item The real constants $\Lambda$ and $\Gamma$ are defined by
\begin{equation}\label{Fermi-Gordonnum}
\Lambda := \left( \theta, \frac{1}{L} {\rm P.V.} \frac{1}{L-2\lambda} \theta \right), 
  \qquad \Gamma := \frac{\pi}{2\lambda} \left( \theta, \delta(L-2\lambda) \theta \right),
  \qquad \theta := P_c \left( \beta \varphi^2 \right);
\end{equation}
\item The remainder term can be split as
\begin{align}\label{RsplitDF}
R = R_D + R_F,
\end{align}
with
\begin{equation}    
\label{R_D-defn}
R_D(t) := \frac{1}{2i\lambda} \left( e^{2i\lambda t} (\Lambda - i \Gamma) A^3
  +(\Lambda + i \Gamma)
  \left( e^{-4i\lambda t} \overline{A}^3 + e^{-2i\lambda t} |A|^2 \overline{A}  \right)\right)
\end{equation}
giving the `non-resonant' cubic self-interactions, and
\begin{align}\label{R_F-defn}
\begin{split}
R_F(t) & := \frac{1}{2i \lambda} e^{-i\lambda t} 
  \Bigg[ \int_\R 
    v^2 \, \beta \varphi \, dx
\\ & + 2a(t) \int_\R \beta \varphi^2 \frac{1}{L} \Im \left( e^{itL} w(0) + e^{itL} R_w + \rest \right) dx 
+ \int (a\varphi + v)^4 \varphi dx \Bigg]
\end{split}
\end{align}
giving all the other terms.
\end{itemize}

\smallskip
\noindent
For later purposes we also single out the feedback of the quartic terms in $v$ from \eqref{R_F-defn} by letting
\begin{align}\label{RFsplitEQ}
R_E := R_F - R_Q,
\end{align}
with (cf. the last term in \eqref{Rest})
\begin{align}\label{R_Q-defn}
\begin{split}
R_Q(t) := \frac{1}{i \lambda} e^{-i\lambda t} a(t) \int_\R \beta \varphi^2 \frac{1}{L} 
  \Im \left( \int_0^t e^{i(t-s)L} P_c v(s)^4 ds \right) dx. 
\end{split}
\end{align}

\end{lemma}

\begin{proof}
Expanding the right-hand side of \eqref{evolution-A-initial}, and recalling that $v = L^{-1} \Im w$,
$\theta = P_c \left( \beta \varphi^2 \right)$, we write
\begin{align}
\dot A(t) & = \frac{1}{2i\lambda} e^{-i\lambda t} 
  \left[ a^2(t) \left(\beta, \varphi^3 \right) 
  + \left( v^2, \beta \varphi \right) + 2 a(t) \left(v, \beta \varphi^2 \right) 
  + \left( (a(t)\varphi + v)^4, \varphi \right) \right].
\end{align}
Then, from the definition \eqref{P-defn},
using the formula for $w$ in \eqref{lemevolw1}, and comparing with \eqref{R_F-defn}, we have
\begin{align}\label{lemevolApr2}
\begin{split}
\dot A(t) & = P(t) + R_F(t) 
  \\
  & + \frac{ e^{-i\lambda t} }{i\lambda}a(t) 
  \left( L^{-1}
  \Im\Big(i e^{2i\lambda t} A^2(t) {\rm P.V.} \left( \frac{1}{L-2\lambda} \right) \theta + e^{2i\lambda t} A^2(t) \pi \delta(L-2\lambda) \theta\Big)
  , \beta \varphi^2 \right) .
\end{split}
\end{align}
To conclude the proof of the lemma it suffices to
check that the second line in \eqref{lemevolApr2} matches the cubic terms in $A$
from \eqref{evolution-A} and \eqref{R_D-defn};
this is a direct calculation which is left to the reader.
\end{proof}

\smallskip
Note that the Fermi Golden Rule \eqref{Fermi-Gordon} 
gives the nonlinear damping term $-(\Gamma/2\lambda) |A|^2A$ in \eqref{evolution-A}.

\medskip
\subsection{Main result}
Recall our assumptions on the regularity and genericity of the external potential 
$V$ from Section \ref{sec:simplified_MR}. Recall also that the solution $u$ of \eqref{nlkg-quartic} is constructed as the sum of the discrete and continuous components:
\begin{equation}\label{u-decomposed}
u(t)=a(t)\varphi +v(t),\quad\quad (v(t),\varphi)=0.    
\end{equation}
For definiteness we fix the following parameters in the bootstrap argument:
\begin{equation}\label{choice-of-csts}
 \al=10^{-11},
\quad \delta = 10^{-10},\quad \delta_2= 10^{-9}\quad \mbox{and}\quad \delta_1=10^{-2},
  \quad N=10^{5} + 1.
\end{equation}
These parameters are not optimized. 

\begin{theorem}\label{thm:mainthmdetail}
Under the spectral assumptions of Subsection \ref{sec:simplified_MR},
there exists a sufficiently small number 
$\overline{\varepsilon_0}>0$ such that,
for all $0 < \varepsilon_0 \leq \overline{\varepsilon_0}$ and all real-valued initial 
data $(u_0,u_1) = (u(t=0),u_t(t=0))$ with 
\begin{align}\label{mainthmdata}
\begin{split}
& |(\varphi,u_0)|+|(\varphi,u_1)| = \varepsilon_0, 
\\
&\| L \,P_c u_0\|_{H^{N-1}_x}+\|P_c u_1\|_{H^{N-1}_x}
  + \|\lr{x} L \,P_c u_0\|_{H^2_x} + \|\lr{x} P_c u_1\|_{H^2_x} \leq \varepsilon_0,
\end{split}
\end{align}
the nonlinear  Klein-Gordon equation \eqref{nlkg-quartic} admits 
a unique solution $u \in C(\R,H^{N}) \cap C^1(\R,H^{N-1})$ 
satisfying the decomposition \eqref{u-decomposed} and all the following properties:

\setlength{\leftmargini}{1.5em}
\begin{itemize}

\smallskip
\item Decay and Sobolev bounds for the internal mode and radiation:
\begin{align}
&  |A(t)| = \frac{1}{2\lambda} \sqrt{\dot{a}^2(t) + \lambda^2 a^2(t)} 
  \approx \varepsilon_0(1+\varepsilon^2_0|t|)^{-1/2},
\end{align}
and
\begin{align}
& \langle t \rangle^{1/2-\delta}
  {\|v(t)\|}_{L^\infty_x} 
  + \lr{t}^{1-\delta_2} {\|\lr{x}^{-2}v(t)\|}_{L^\infty_x} \lesssim \varepsilon_0^\alpha,
  \qquad {\|v(t)\|}_{H^N_x} \lesssim \varepsilon_0,
\end{align}
for all $t \in \R$, with
\begin{align}
{\|v(t)\|}_{L^\infty_x}  \lesssim \frac{\log(2+\varepsilon_0^2 \lr{t} )}{\lr{t}^{1/2}}, 
  \qquad t \gtrsim \varepsilon_0^{-2}.
\end{align}

\smallskip
\item Weighted estimates for the profile $f=h - g$
(see the definitions in \eqref{bad-g-defn}): 
\begin{equation}\label{growth-f}
\|\partial_k\wt{g}(t)\|_{L^2_k}\gtrsim\sqrt{t},\quad\quad t \gtrsim \varepsilon_0^{-2}.    
\end{equation}    
and
\begin{equation}
\sup_{ t \in \R} \, \lr{t}^{-\delta_1}\|\partial_k\wt{h}(t)\|_{L^2_k}\lesssim \varepsilon_0^\alpha. 
\end{equation}

\smallskip
\item Asymptotics for the internal mode:
there exists a real phase $\psi_\infty$ depending on the initial data,
and constants $c_2, \beta_1 \in \R$ depending only on the equation 
(see Theorem \ref{thm:asympA} for the exact expressions) such that
\begin{equation}\label{internal-mode-asymp}
A(t) = \exp\Big[ i \Big(\psi_\infty 
  - (c_2 - \beta_1) \tfrac{\lambda}{\Gamma}\log \big(1+\tfrac{\Gamma}{\lambda} \rho_0 t\big) \Big) \Big]
  \sqrt{\frac{\lambda}{\Gamma \, t}}
  + O\big( t^{-1/2-\gamma} \big), \qquad   t \gg \varepsilon_0^{-2}
\end{equation}
for some small $\gamma >0$, where $\rho_0 = |A(0)|^2$. 
(The remainder is understood as $t\rightarrow \infty$ for fixed data.)
 
\smallskip
\item Scattering to the linear KG flow with potential: there exists $\Phi_+\in P_c H^{N-1}_x$ such that
\begin{align}\label{scattering-main}
{\big\| 
  (\partial_t + iL) v - e^{itL}\Phi_+ \big\|}_{H^{N-1}_x} 
    \lesssim \varepsilon_0 (1+\varepsilon_0^2 t)^{-\sigma}, \quad t>0
\end{align}
for some fixed $\sigma >0$.
Analogous asymptotics as $t \rightarrow -\infty$ can be obtained by time-reversal symmetry.  

\end{itemize}

\end{theorem}




\bigskip
\section{Linear theory}\label{sec:Linear}
\subsection{Decomposition of generalized eigenfunctions}
Consider the Jost functions $f_+(x, k)$ and $f_-(x, k)$ defined as solutions to the differential equation
\begin{equation}
\label{Jost-eq}
H f_\pm (x, k) = k^2 f_\pm (x, k)
\end{equation}
with the asymptotic boundary conditions
\begin{equation}
\label{bc}
\lim_{x \to \infty} |e^{-ikx} f_+(x, k) - 1| = 0 \quad \text{and} \quad \lim_{x \to -\infty} |e^{ikx} f_-(x, k) - 1| = 0.
\end{equation}
Also, define \begin{equation}\label{m+-}m_\pm (x, k) := e^{\mp ikx} f_\pm (x, k),\end{equation}
which satisfies the equation 
\begin{equation}
 \partial_x^2m_{\pm}\pm2ik\partial_xm_{\pm}-Vm_{\pm}=0, 
\end{equation}
subject to the boundary conditions  $m_{\pm}\rightarrow 1$  as $x\rightarrow \pm\infty$. 
The decay/growth properties of $m_{\pm}$ and its derivatives in $x$ and $k$ are given below. 
See, for example,  Deift-Trubowitz \cite{DT} and Weder \cite{Wed} for the proofs of the lemmas below.

\begin{lemma}
\label{lem:oscillatory-Jost}
For every $s = 0,1,2$ and every non-negative integer $N$, we have

\begin{equation}
\label{mixed-der-bds}
\begin{aligned}
&\left|\partial_x^N \partial_k^s (m_\pm (x, k)-1)\right| \lesssim \lr{x}^{-12} \lr{k}^{-1+N}, \qquad    \pm x \geq -1, \\
&\left|\partial_x^N \partial_k^s (m_\pm (x, k)-1)\right| \lesssim \lr{x}^{s+1} \lr{k}^{-1+N}, \qquad \pm x \leq 1.
\end{aligned}
\end{equation}
\end{lemma}

Denote by $T(k)$ and $R_\pm(k)$ the  transmission and reflection attached to the potential $V$. They are defined through the formulas
\begin{equation}
\label{jost-TR}
\begin{aligned}
T(k) f_+(x,k) & = R_-(k) f_-(x,k) + f_-(x, -k),
\\
T(k) f_-(x,k) & = R_+(k) f_+(x,k) + f_+(x, -k),
\end{aligned}
\end{equation}
and by the boundary conditions \eqref{bc}, we obtain 
\begin{equation}
\begin{aligned}
T(k) f_+(x,k) & \sim R_-(k) e^{-ikx} + e^{ikx} \quad \text{as } x\to -\infty
\\
T(k) f_-(x,k) & \sim R_+(k) e^{ikx} + e^{-ikx} \quad \text{as } x\to \infty.
\end{aligned}
\end{equation}
The coefficients $T(k)$ and $R_\pm(k)$ can be expressed explicitly through ($k\neq0$)
\begin{equation}
\label{explitcit-T-R-coeffs}
\begin{aligned}
\frac{1}{T(k)} & = 1 - \frac{1}{2ik} \int V(x) m_\pm(x,k) dx, 
\\
\frac{R_\pm(k)}{T(k)} & = \frac{1}{2ik} \int e^{\mp 2ikx} V(x) m_\mp(x,k) dx,
\end{aligned}
\end{equation}
and satisfy the following relations at every $k$:
\begin{equation*}
\begin{aligned}
& T(-k) = \overline{T(k)}, \quad R_\pm (-k) = \overline{R_\pm(k)}, 
\\
& |T(k)|^2 + |R_\pm(k)|^2 = 1, \quad T(k) \overline{R_-(k)} + \overline{T(k)} R_+(k) = 0.
\end{aligned}
\end{equation*}

\begin{lemma}
\label{lem:T-R-prop}
Let $V$ be a generic potential. 
The coefficients $T(k)$ and $R_\pm(k)$ satisfy
\begin{itemize}
\item As $k\to 0$, 
\begin{equation}
\label{T-R-near-zero}
\begin{aligned}
T(k) & = \alpha k + O (k^2) \quad \text{for some } \alpha \neq 0,
\\
R_\pm(k)+1 & = \alpha_\pm k + o(k),
\end{aligned}
\end{equation}
so that, in particular $T(0) = 0$ and $R_\pm (0) = -1$.

\item For every $k\in \mathbb{R}$ we have
\begin{equation}
\label{uniform-estimate-T-R}
\begin{aligned}
|T(k)| + |R_\pm(k)| &\lesssim 1,\\
|\partial_k T(k)| + |\partial_k R_\pm(k)|& \lesssim \lr{k}^{-1}.
\end{aligned}
\end{equation}
\item For every $k \in \R$ we have 
\begin{equation}
\label{ptwise-bd-T}
\left| k^{-1}T(k) \right| + \lr{k} \left| \partial_k \left(k^{-1}T(k)\right) \right| \lesssim \lr{k}^{-1},
\end{equation}
and 
\begin{equation}
\label{bound-on-R-over-k}
\left| k^{-1}(R_\pm(k)+1) \right| + \left| \partial_k \left(k^{-1}(1+R_\pm(k))\right) \right| 
  \lesssim \lr{k}^{-1}.
\end{equation}
\end{itemize}
\end{lemma}

\subsection{Distorted Fourier transform}
We recall some basic properties of the distorted Fourier transform
with respect to the perturbed Schr\"odinger operator.
First, recall that the standard Fourier transform is defined, for $f\in L^2$, as
\[
\mathcal{F}\left[f\right](k):=\hat{f}(k)=\frac{1}{\sqrt{2\pi}}\int e^{-ik x}f(x)\,dx
\]
with its inverse as
\[
\mathcal{F}^{-1}\left[f\right](x):=\frac{1}{\sqrt{2\pi}}\int e^{ikx}f(k)\,dk.
\]

\begin{definition}
Given the Jost functions $f_{\pm}$ from \eqref{Jost-eq}, we set
\begin{align}\label{matK}
\psi(x,k):=\frac{1}{\sqrt{2\pi}}
\begin{cases}
T(k) f_+(x,k) & k\geq0
\\
T(-k) f_-(x,-k) & k<0
\end{cases},
\end{align}
and define the ``distorted Fourier transform'' for $f\in \mathcal{S}$ by
\begin{align}\label{tildeF}
\wt{\mathcal F}\left[f\right](k) = \widetilde{f}(k) := \int \overline{\psi(x,k)} f(x)\,dx.
\end{align}
\end{definition}

By the definition of the generalized eigenfunctions and $m_{\pm}$, we have  
\begin{equation}
\label{eigenfunctions}
\begin{aligned}
& \sqrt{2\pi} \psi(x, k) = T(k) f_+(x, k) = T(k) m_+(x, k) e^{ikx} \quad \text{for} \quad k\geq 0\\[2pt] 
& \sqrt{2\pi} \psi(x, k) = T(-k) f_-(x, -k) = T(-k) m_-(x, -k) e^{ikx} \quad \text{for} \quad k<0.
\end{aligned}
\end{equation}
Let $\varphi$ be a smooth, nonnegative function that is equal to $1$ in a neighborhood of $0$, vanishes outside of $[-2,2]$, and satisfies $\int \varphi ~dx = 1$. 
Define the functions $\sigma_+$ and $\sigma_-$, by
\begin{equation}
\label{sigma-functions}
\sigma_+(x) = \int_{-\infty}^x \varphi(y) dy, \quad \text{ and } \quad \sigma_+(x) + \sigma_-(x) = 1.
\end{equation}
Using \eqref{jost-TR} we can write
\begin{equation*}
\begin{aligned}
&\sqrt{2\pi} \psi(x, k) = \sigma_+(x) T(k) f_+(x, k) + \sigma_-(x) \left[ f_-(x, -k) +R_-(k) f_-(x,k) \right] \quad\text{for}\quad k\geq0  \\[2pt] 
& \sqrt{2\pi} \psi(x, k) = \sigma_-(x) T(-k) f_-(x, -k) + \sigma_+(x) \left[ f_+(x, k) +R_+(-k) f_+(x,-k) \right]\quad\text{for} \quad k<0,
\end{aligned}
\end{equation*}
which in terms of $m_\pm (x,k)$, through \eqref{m+-}, read \begin{align}
\label{eigenfunction-exp}
&\sqrt{2\pi} \psi(x, k) =  \sigma_+(x) T(k) m_+(x, k) e^{ikx}  + \sigma_-(x) \left( m_-(x, -k) e^{ikx} + R_-(k) m_-(x,k) e^{-ikx} \right)\,\, \text{for}\,\, k\geq0,\\
&\sqrt{2\pi} \psi(x, k) = \sigma_-(x) T(-k) m_-(x, -k) e^{ikx} 
+ \sigma_+(x) \left( m_+(x, k)e^{ikx} +R_+(-k) m_+(x,-k) e^{-ikx} \right)\,\,  \text{for} \,\, k<0.
\end{align}

\begin{lemma}\label{lemma-distorted-fourier}
In the above setting, $\widetilde{\mathcal{F}}: P_c L^2_x \to L^2_k$ is unitary,
\begin{equation}
\label{isometry}
\|\distF h \|_{L^2_k} = \|P_c h \|_{L^2_x},\quad \text{for all  } h \in L^2.
\end{equation}
The inverse of the distorted Fourier transform is given by
$$
\distF^{-1} [g] (x) = \int \psi(x,k) g (k) dk.
$$
Setting $\widetilde{D} := \distF^{-1}k \distF$ on $P_cL^2$ or ($\widetilde{D}^2 := P_c H$),
we have 
$$
m(\widetilde D)P_c = \distF^{-1} m(k) \distF.
$$
Moreover, the distorted Fourier transform satisfies the following properties:
\begin{itemize}
\item $\distF \distF^{-1} = \text{\rm Id}_{L_k^2}$ and $\distF^{-1} \distF = P_c$;

\item If $h \in L^1 \cap L^2$, and since $V$ is generic,
then $\widetilde{h}$ is a continuous, bounded function with $\widetilde{h}(0) = 0$;

\item We have
\begin{equation}
\label{derivative-in-fourier-bound}
{\big\| \widetilde{h} \big\|}_{L^2_k}
 + {\big\|\partial_k \widetilde{h} \big\|}_{L^2_k} \approx \|\lr{x} P_c h\|_{L^2_x},
\end{equation}
and, for $s \geq 0$\begin{equation}
\label{k-s-h-tilde-bound}
{\big\| \lr{k}^s \widetilde{h} \big\|}_{L^2_k} \approx \|P_c{h}\|_{H_x^s} 
\end{equation}
\end{itemize}
\end{lemma}
The proofs of these properties 
can be found in Agmon \cite{A75}, Dunford-Schwartz \cite{DS88}, 
Germain-Pusateri-Rousset \cite{GPR18} and Yafaev \cite{Y10}.

\begin{definition}[The wave operator]\label{def:waveoper}
    Let $V$ be a generic potential and denote by $H_0=-\partial_x^2$ and $H=-\partial_x^2+V$, the flat and perturbed Schr\"{o}dinger operators. The wave operator is defined by the strong limit in $L^2$
    \begin{align}
     \mathcal{W}=\lim_{t\rightarrow  \infty}   e^{itH}\,e^{-itH_0}.
    \end{align}
\end{definition}

\begin{lemma}\label{lem:waveoper}
  The wave operator $\mathcal{W}: L^2 \rightarrow P_cL^2$ is unitary, 
  $\mathcal{W}\mathcal{W}^\ast = \mathrm{id}$ and $\mathcal{W}^\ast\mathcal{W}=P_c$, 
  and satisfies the following properties:
  \begin{itemize}
  \item $\mathcal{W}=\wt{\mathcal{F}}^{-1}\widehat{\mathcal{F}}$ and $\mathcal{W}^{-1}=\mathcal{W}^{\ast}=\widehat{\mathcal{F}}^{-1}\wt{\mathcal{F}}$,\\
  \item For all bounded Borel function $m$, 
\begin{equation}\label{intertwining-prop}
m(H)=\mathcal{W}m(H_0)\mathcal{W}^{\ast}.
\end{equation}
  \item $\mathcal{W}$ and $\mathcal{W}^\ast$ 
  extend to bounded operators on the Sobolev space $W^{s,p}$ for any $s\in \mathbb Z^+$ and $p\in (1,\infty)$.  
\end{itemize}
\end{lemma}
The proofs of the properties listed in this lemma are standard. See for instance Artbazar-Yajima \cite{GY}, 
Weder \cite{Wed} and  Germain-Pusateri \cite{GP22-mathpi}.

\subsubsection{Decompositions}
Given the above localization properties of $m_{\pm}-1$ and their derivatives, 
we further decompose $\psi$ into singular and regular parts. 
We denote by 
$\mathbbm{1}_+$ and $\mathbbm{1}_-$ the indicator functions $\mathbbm{1}_{[0, \infty)}$ 
and $\mathbbm{1}_{(-\infty, 0)}$, respectively.
We can then write
\begin{equation}\label{generalized-eigenf-decomp}
\sqrt{2\pi} \, \psi(x,k) = \psi_{S}(x,k) + \psi_{R}(x,k),
\end{equation}
where the singular part is defined as
\begin{equation}
\label{psi-singular-part}
\psi_S(x,k) := \sigma_+(x) \psi_{S, +}(x,k) + \sigma_-(x) \psi_{S, -}(x,k),
\end{equation}where
\begin{equation}
\label{singular-part-pieces}
\left\{ \begin{array}{l}
\psi_{S,+}(x,k) := \mathbbm{1}_+(k) T(k) e^{ikx} + \mathbbm{1}_-(k) \left[ e^{ikx} + R_+(-k) e^{-ikx} \right],\\[4pt]
\psi_{S,-}(x,k) := \mathbbm{1}_+(k) \left[ e^{ikx} + R_-(k) e^{-ikx} \right] + \mathbbm{1}_-(k) T(-k) e^{ikx}.
\end{array} \right.
\end{equation}
Note that $\psi_S$ 
is a linear combination of oscillatory terms given by exponential functions with bounded coefficients. 
The regular part is then given by
\begin{equation}
\label{psi-regular-part}
\psi_R(x,k) := \left\{ \begin{array}{l}
\begin{aligned}
 &\sigma_+(x) T(k) (m_+(x, k)-1) e^{ikx}\\[-2pt]
&+ \sigma_-(x) [ (m_-(x, -k)-1) e^{ikx} + R_-(k) (m_-(x,k)-1) e^{-ikx} ] \quad \mbox{for}\,\, k\geq0
\end{aligned}\\
\phantom{t}\\
\begin{aligned}
& \sigma_-(x) T(-k) (m_-(x, -k)-1) e^{ikx} \\[-2pt]
&+ \sigma_+(x) [ (m_+(x, k)-1) e^{ikx} +R_+(-k) (m_+(x,-k)-1) e^{-ikx} ] \,\,\mbox{for}\,\, k<0.
\end{aligned}
\end{array} 
\right.
\end{equation}
Note that $\psi_R$ is localized around $x=0$ due to \eqref{mixed-der-bds}.
It is also useful to further decompose the singular part in \eqref{singular-part-pieces} by rewriting
\begin{equation}
\label{decomposition-sing-part}
\left\{ \begin{array}{l}
\psi_{S,+}(x,k) := A_{+,+}(k) e^{ikx} + A_{+,-}(k) e^{-ikx},
\\[4pt]
\psi_{S,-}(x,k) := A_{-,+}(k) e^{ikx} + A_{-,-}(k) e^{-ikx},
\end{array} \right.
\end{equation}
where 
\begin{equation}
\label{A-functions}
\begin{aligned}
& A_{+,+}(k) := \mathbbm{1}_+(k) T(k) + \mathbbm{1}_-(k), \qquad A_{+,-}(k) := \mathbbm{1}_-(k) R_+(-k),
\\
& A_{-,+}(k) := \mathbbm{1}_+(k) + \mathbbm{1}_-(k) T(-k), \qquad A_{-,-}(k) := \mathbbm{1}_+(k) R_-(k).
\end{aligned}
\end{equation}

Finally, we introduce the functions 
\begin{align}\label{psi-sing}
\begin{aligned}
\psi_{+,+}(x,k):=A_{+,+}(k)e^{ixk}    ,\quad \psi_{+,-}(x,k):=A_{+,-}(k)e^{-ixk}
\\
\psi_{-,+}(x,k):=A_{-,+}(k)e^{ixk}    ,\quad \psi_{-,-}(x,k):=A_{-,-}(k)e^{-ixk},
\end{aligned}
\end{align}
so that
\begin{align}
\psi_{S,+}(x,k) = \psi_{+,+}(x,k) + \psi_{+,-}(x,k) \qquad
\psi_{S,-}(x,k) = \psi_{-,+}(x,k) + \psi_{-,-}(x,k).
\end{align}
Consistently, for any given Schwartz function $\phi$, we can write
\begin{align}\label{phi-decomp}
P_c\phi = \phi_S + \phi_R, \qquad  
   \phi_{E} = \frac{1}{\sqrt{2\pi}} \int \psi_{E}(x,k)\wt{\phi}(k)dk,\quad E \in \{S,R\},
\end{align}
and will adopt a similar notation  

\begin{align}\label{phi-decomp'}
\phi_{S,\ast}(x) = \frac{1}{\sqrt{2\pi}} 
\int \psi_{S,\ast}(x,k)\wt{\phi}(k)dk,\quad \ast \in \{+,-\},
\end{align}
so that 
\begin{align}\label{phi-decomp''}
\phi_{S,\ast} = \phi_{\ast,+} + \phi_{\ast,-},
 \qquad \phi_{S}(x) = \sigma_+(x)\big( \phi_{+,+} + \phi_{+,-} \big)
  + \sigma_-(x)\big( \phi_{-,+} + \phi_{-,-} \big).
\end{align}


\bigskip
\section{Linear estimates and local decay}\label{sec:linear-decay}
In this section, we provide linear and local decay estimates for the linear operator 
\begin{equation}
\label{g-profile-linear-term}
\lr{\wt{D}}^{-1} e^{\pm it \lr{\wt{D}}} P_c, 
  \qquad \lr{\wt{D}}P_c := LP_c := \sqrt{HP_c+1} = \distF^{-1} \lr{k} \distF,
\end{equation}
cf. the notation in Proposition \ref{lemma-distorted-fourier}.

\medskip
\subsection{Basic linear decay estimates}
\begin{lemma}
\label{lem:dispersive-estimates}
Under the assumption that $V$ is generic and decays sufficiently fast, for every $t\geq 0$ we have

\begin{equation}
\label{dispersive-estimates0}
\left\| L^{-1} e^{\pm it L} P_c f \right\|_{L_x^\infty} \lesssim \lr{t}^{-\frac{1}{2}} 
  \sum_{\ell\in \mathbb{Z}} (1+2^{\ell}) {\| P_\ell f \|}_{L^1}
\end{equation}
where we are using the notation $P_\ell f := \distF^{-1}(\phi_\ell \wt{f})$ for Littlewood-Paley
projectors (see \S \ref{ssec:notations})
and
\begin{equation}
\label{dispersive-estimates}
\left\|
L^{-1} e^{\pm it L} P_c f \right\|_{L_x^\infty} \lesssim \lr{t}^{-\frac{1}{2}} 
  \left( \|\partial_k \widetilde f\|_{L^2_k} + \|\lr{k}^2 \widetilde f\|_{L^2_k} \right).
\end{equation}
\end{lemma}

The proof of Lemma \ref{lem:dispersive-estimates} 
is standard: \eqref{dispersive-estimates0} is the $L^1$--$L^\infty$ linear decay estimate
obtained by applying the stationary phase lemma, and
\eqref{dispersive-estimates} follows from similar (in fact easier) arguments 
to those in \cite[Proposition 3.11]{GP22-mathpi}.

\begin{remark}\label{rmk:Linfty-flat-linear-decay}
Lemma \ref{lem:dispersive-estimates} can be repeated for $V=0$. 
In such a case, $\wt{D}$ is replaced by $D$ and the distorted Fourier 
transform becomes the standard Fourier transform. 
The estimate corresponding to \eqref{dispersive-estimates} is
\begin{equation}
\label{disp-est-stand-fourier}
\left\| 
\lr{D}^{-1} e^{\pm it \lr{D}} f \right\|_{L_x^\infty} \lesssim \lr{t}^{-1/2} \left( \|\partial_k \widehat f\|_{L^2_k} + \|\lr{k}^2 \widehat f\|_{L^2_k} \right).
\end{equation}

\end{remark}

\smallskip
Next, we state the local 
decay estimates which will be used later.
We start with the estimates for the standard Fourier transform under an additional assumption at $k=0$. 
\begin{lemma}\label{lem:local-l-infty-decay-standard-Fourier}
Suppose $\hat f \in H_k^1$ and $\hat f (0) = 0$, then for every $t \geq 0$ we have

\begin{equation}
\label{local-l-infty-decay-standard-Fourier}
\left\| \lr{x}^{-1} \lr{D}^{-1} e^{\pm it \lr{D}} f \right\|_{L^\infty_x} 
  \lesssim \lr{t}^{-3/4} \| \widehat{f} \|_{H^1_k}. 
\end{equation}
\end{lemma}
The proof of \eqref{local-l-infty-decay-standard-Fourier} is standard, and can be found in \cite{CP22} in the completely analogous case of the Schr\"odinger equation.

Estimates for the distorted Fourier transform have better decay rates under the genericity assumption on $V$.

\begin{lemma}
\label{lem:local-l-infty-decay-lemma}
Suppose $V$ is generic and decays sufficiently fast. Then, for every $t \geq 0$, we have
\begin{equation}
\label{local-l-infty-decay}
\left\| \lr{x}^{-2} L^{-1} e^{it L}P_c f \right\|_{L^\infty_x} 
  \lesssim \lr{t}^{-1} \| \widetilde{f} \|_{H^1_k}. 
\end{equation}
\end{lemma}
The proof of this result can be found in \cite{GPZ23}. 
We will also need stronger decay rates (under stronger assumptions) 
in our proofs in Section \ref{sec:ODE}, so we include the following: 

\begin{lemma}\label{lem:stationary-phase}
Given $X \in \mathbb{R}, x\in \mathbb{R}$ and $t \geq 0$. Consider the integral
\begin{equation}
\label{oscillatory-integral-generic}
\mathcal{I}_{\iota, \nu} (t, X, x) = \int_{\mathbb{R}_\iota} e^{it (\nu \lr{k} - kX)} a(x,k) f(k) dk, \qquad \iota, \nu \in \{+, -\}
\end{equation}
and let 
\begin{equation}\label{cond-kernel-a}
\sup\limits_{x\in \mathbb{R}, k \in \mathbb{R}_\iota} \left( |a(x,k)| + \lr{k} |\partial_k a(x,k)| \right) \lesssim \lr{k}^{-1}.
\end{equation}
Then, we have 
\begin{equation*}
\left| \mathcal{I}_{\iota, \nu} (t, X, x) \right| \lesssim \lr{t}^{-1/2} \left( \normm{\partial_k f}{k} + \normm{\lr{k}^2 f}{k} \right). 
\end{equation*}
\end{lemma}
The proof follows from the arguments in \cite{GP22-mathpi}. 
An analogous estimate for the Schr\"odinger flow is given in \cite{GPR18}.
Based on this result, we can obtain the local decay rate of $\lr{t}^{-3/2}$. 

\begin{lemma}\label{lem:stronger-local-decay}
It suffices to look at $t\geq 1$.
Under the assumption that $V$ is generic and decays sufficiently fast, we have 
\begin{equation}
\label{local-decay-faster}
\begin{aligned}
\left\| \lr{x}^{-3} L^{-1} e^{it L} P_c f \right\|_{L^\infty_x} & \lesssim \lr{t}^{-3/2} \left( 
\| \partial_k^2 \widetilde{f} \|_{L^2_k} + \| \lr{k}^2 \partial_k \widetilde{f} \|_{L^2_k} + \| \lr{k}^2 \widetilde{f} \|_{L^2_k} \right) \\
& \lesssim \lr{t}^{-3/2} \left(\| \lr{x}^2 f \|_{L^2_x} + 
  \| \lr{x} f \|_{H^2_x} \right).
\end{aligned}
\end{equation}
\end{lemma}

\begin{proof}
We first rewrite, using Lemma \ref{lemma-distorted-fourier} and \eqref{eigenfunctions},
\begin{equation*}
\begin{aligned}
L^{-1} e^{it L} P_c f (x) & = \frac{1}{\sqrt{2\pi}}\int_0^\infty \lr{k}^{-1} e^{it\lr{k}} 
T(k) m_+(x,k) e^{ikx} \widetilde{f}(k) dk \\
& \qquad \qquad + \frac{1}{\sqrt{2\pi}}\int_{-\infty}^0 \lr{k}^{-1} 
e^{it\lr{k}} T(-k) m_-(x,-k) e^{ikx} \widetilde{f}(k) dk 
\\
& = I_+ + I_-. 
\end{aligned}
\end{equation*}
Concentrating on $I_+$ only (since $I_-$ can be treated similarly), 
we integrate by parts using $\partial_k e^{it \lr{k}} = it k\lr{k}^{-1} e^{it\lr{k}}$ 
to get
\begin{equation*}
\begin{aligned}
i \sqrt{2\pi} I_+ & = \frac{1}{t} \int_0^\infty 
\partial_k \left( \frac{T(k)}{k} \right) m_+(x,k) e^{ikx} e^{it\lr{k}} \widetilde f (k) dk \\
& \qquad + \frac{1}{t} \int_0^\infty \frac{T(k)}{k} \partial_k m_+(x,k) e^{ikx} e^{it\lr{k}} \widetilde f (k) dk \\
& \qquad + \frac{1}{t} \int_0^\infty \frac{T(k)}{k} m_+(x,k) (ix) e^{ikx} e^{it\lr{k}} \widetilde f (k) dk \\
& \qquad + \frac{1}{t} \int_0^\infty \frac{T(k)}{k} m_+(x,k) e^{ikx} e^{it\lr{k}} \partial_k \widetilde f (k) dk 
  = I_{+,1} + I_{+,2} + I_{+,3} + I_{+,4},
\end{aligned}
\end{equation*}
where the boundary terms in the process canceled since $\widetilde{f}(0) = 0$ 
and we may assume that the integrand vanishes at infinity. 
Then, the bound \eqref{local-decay-faster} is obtained by applying Lemma 
\ref{lem:stationary-phase} to each of $I_{+,j}$ above.  

We note that the localization term $\lr{x}^{-3}$ in the statement
\eqref{local-decay-faster} is needed to compensate the growth of $\partial_k^2 m_+ (x,k)$ for 
large negative $x$ and the factor $x$ appearing in $I_{+, 3}$. 
\end{proof}

\smallskip
\begin{remark}\label{rmk:r-s-local-decay}We can analogously derive the following estimates for the regular and singular parts of 
the flow $L^{-1} e^{\pm it L} P_c$:
\begin{align}
\label{local-l-infty-decay-sing}
\left\| \lr{x}^{-2} \left( L^{-1} e^{\pm it L} P_c  h\right)_S\right\|_{L^\infty_x} 
  &\lesssim \lr{t}^{-1} \| \widetilde{h} \|_{H^1_k}, 
  \\    
\label{local-l-infty-decay-reg}
\left\| \lr{x}^{10} \left(L^{-1} e^{\pm it L} P_c h\right)_R \right\|_{L^\infty_x} 
 & \lesssim \lr{t}^{-1} \| \widetilde{h} \|_{H^1_k},\\  
  \label{improved-local-decay-sing/reg}
\left\| \lr{x}^{-2} (L^{-1} e^{\pm it L} P_c f)_{E} \right\|_{L^\infty_x}  &\lesssim \lr{t}^{-3/2} \left( 
\| \partial_k^2 \widetilde{f} \|_{L^2_k} + \| \lr{k}^2 \partial_k \widetilde{f} \|_{L^2_k} + \| \lr{k}^2 \widetilde{f} \|_{L^2_k} \right)
\end{align}
where $(\cdot)_E$ is defined via \eqref{phi-decomp} and $\psi_{E}$, 
$E\in \{S,R\}$, is given in \eqref{psi-singular-part}-\eqref{psi-regular-part}.  Moreover, we have 
\begin{equation}\label{local-l-infty-decay-sing'}
    \left\| \lr{x}^{-2} \left( L^{-1} e^{\pm it L} P_c  h\right)_{S,\ast}\right\|_{L^\infty_x} 
  \lesssim \lr{t}^{-1} \| \widetilde{h} \|_{H^1_k}, 
\end{equation}
where the notation $(\cdot)_{E,\ast}$ is defined through \eqref{phi-decomp'}.
\end{remark}

\smallskip
\begin{lemma}\label{lemlocal-decay-0-plus}
Under the assumption that $V$ is generic and decays sufficiently fast, we have
\begin{equation}
\label{local-decay-0-plus}
\left\| h(x) \frac{e^{it L}}{L-2\lambda + i0^+} P_c f \right\|_{L^1_x} \lesssim \lr{t}^{-3/2}  
  \| \lr{x}^5 h \|_{L^2_x} \| \lr{k}^3 \widetilde{f} \|_{H^3_k}. 
\end{equation}
\end{lemma}
\begin{proof}
The estimate can be derived similarly as in Lemma 2.7 of \cite{LP21}.    
\end{proof}

We will also use the following:

\begin{lemma}\label{lem:nonconv-bds}
For $\gamma=0,1$, we have
\begin{equation}
\bigg\|\int \partial^{\gamma}_k\overline{\psi_R(x,k)}f(x)dx\bigg\|_{L^2_k}  \lesssim \|\lr{\cdot}^{-11+\gamma}f\|_{L^2_x}.
\end{equation}    
\end{lemma}

\begin{proof}
The proof is straightforward using the localization in $x$ of $\psi_R$ in \eqref{psi-regular-part}.
See Section 2 of \cite{CP22} for the proof and also Remark 2.4 in \cite{KP23}.  
\end{proof}

We conclude this subsection with a useful linear decay estimate
which will be used in Section \ref{sec:PDE-part2}.

\begin{lemma}\label{lemlininterp}
Under the same assumption on $L=(-\partial_x^2 + V+1)^{1/2}$ above, for any $N\geq 4$ we have
\begin{align}\label{lininterpest}
{\big\| e^{itL} P_c f \big\|}_{L^\infty} \lesssim
  \lr{t}^{-1/2} \log(2+t) \, \Big( {\big\| f \big\|}_{H^4}^{1/2} 
  \, {\big\| \partial_k \wt{f} \big\|}_{L^2}^{1/2} 
  + {\big\| f \big\|}_{H^N} \Big).
\end{align}
\end{lemma}

\begin{proof}
For $|t|\leq 1$ the statement follows from Sobolev's embedding so we consider $|t| > 1$
in what follows.
We use the notation for Littlewood-Paley projections from \S\ref{ssec:notations}
and begin by bounding low and high frequencies:
\begin{align}\label{interpestpr0}
{\big\| e^{itL} P_\ell f \big\|}_{L^\infty} \lesssim 2^{\ell/2} {\big\| P_\ell f \big\|}_{L^2} 
  \lesssim \min \big(2^{\ell/2} {\| f \|}_{L^2}, 2^{(1/2-N)\ell } {\| f \|}_{H^N} \big).
\end{align}
Using \eqref{interpestpr0} and the linear estimate \eqref{dispersive-estimates0} we have
\begin{align}\label{interpestpr1}
\begin{split} 
{\big\| e^{itL} f \big\|}_{L^\infty} 
  & \lesssim \sum_{\lr{t}^{-2} \leq 2^\ell \leq \lr{t}^{2/N}}  
  {\big\| e^{itL} P_\ell f \big\|}_{L^\infty} 
  + \sum_{2^\ell \geq \lr{t}^{2/N}} {\big\| e^{itL} P_\ell f \big\|}_{L^\infty} 
  + \sum_{2^\ell \leq \lr{t}^{-2} } {\big\| e^{itL} P_\ell f \big\|}_{L^\infty}  
  \\
  & \lesssim \frac{1}{\lr{t}^{1/2}}\log(2+t) \cdot \sup_{\ell} (1+2^{2\ell}) 
  {\big\| P_\ell f \big\|}_{L^1} 
  + \frac{1}{\lr{t}} {\big\| f \big\|}_{H^N}.
\end{split}
\end{align}
The last term on the right-hand side of \eqref{interpestpr1} is acceptable for \eqref{lininterpest}
so we concentrate on the first and estimate the $L^1$-norm 
using the basic interpolation inequality
\begin{align}\label{interpestpr2}
{\big\| P_\ell f \big\|}_{L^1} \lesssim
   {\big\| P_\ell f \big\|}_{L^2}^{1/2} {\big\| x P_\ell f \big\|}_{L^2}^{1/2}
   \lesssim (1+2^{2\ell})^{-1}{\big\| f \big\|}_{H^4}^{1/2} 
   \, {\big\| x P_\ell f \big\|}_{L^2}^{1/2}.
\end{align}
Using \eqref{derivative-in-fourier-bound} we then observe that 
\begin{align}\label{interpestpr3}
\begin{split}
{\big\| x P_\ell f \big\|}_{L^2} & \lesssim {\| f \|}_{L^2} + {\big\| P_\ell x f \big\|}_{L^2} 
  + 2^{-\ell} {\big\| P_{[\ell-2,\ell+2]} f \big\|}_{L^2} 
  \\
  & \lesssim {\| f \|}_{L^2} + {\big\| \partial_k \wt{f} \big\|}_{L^2} 
  + 2^{-\ell} {\big\| \phi_{[\ell-2\ell+2]} \wt{f} \big\|}_{L^2} 
\end{split}
\end{align}
having used \eqref{derivative-in-fourier-bound}.
To obtain \eqref{lininterpest} it then suffices to show that
\begin{align}\label{interpestpr4}
2^{-\ell} {\big\| \phi_{[\ell-2,\ell+2]} \wt{f} \big\|}_{L^2}  
  \lesssim {\big\| \partial_k \wt{f} \big\|}_{L^2}.
\end{align}
Using $\wt{f}(0)=0$ we can bound
\begin{align}\label{interpestpr5}
 |\phi_{[\ell-2,\ell+2]}(k) \wt{f}(k)|
  \leq \phi_{[\ell-2,\ell+2]}(k) \int_0^k \big|\partial_{k'} \wt{f}(k') \big| \, dk'
  \leq  |k|^{1/2}\phi_{[\ell-2,\ell+2]}(k)  {\| \partial_{k} \wt{f} \|}_{L^2},
\end{align}
so that 
\begin{align}\label{interpestpr6}
 {\|\phi}_{[\ell-2,\ell+2]} \wt{f}\|_{L^2}
 \lesssim 2^{\ell/2} {\|\phi_{[\ell-2,\ell+2]}(k)\|}_{L^2} {\| \partial_{k} \wt{f} \|}_{L^2}
 \lesssim 2^{\ell}{\| \partial_{k} \wt{f} \|}_{L^2},
\end{align}
which is \eqref{interpestpr4}.
\end{proof}

\medskip
\subsection{Improved local decay and smoothing-type estimates}
We will also need some local decay estimates for the case in which the 
linear operator is spectrally localized. 
These estimates will be essential for controlling certain growing terms in the nonlinear analysis.
Particularly useful is the scale-refined local decay estimate
\eqref{local-estimate-time-cut-off-refined} which quantifies the gain obtained
by localizing near the bad/resonant frequency $|k|$ with $\lr{k}=2\lambda$.

Fix a small constant $0 < r \ll 2\lambda-1$ and denote by $\chi_r$ a nonnegative bump function
supported in an interval around zero with size $r$.   
\begin{lemma}
\label{lem:local-estimate-time-cut-off}
For all $t>0$ we have the estimates  
\begin{align}
\label{local-estimate-time-cut-off}
\left\| \lr{x}^{-2} L^{-1} e^{itL} \chi_{r}(L-2\lambda) P_c h \right\|_{L_x^\infty} 
  & \lesssim \lr{t}^{-1} {\big\| \wt{h} \big\|}_{H^1_k},
  \\
\label{local-estimate-time-cut-off-internal-mode}
\left\| \lr{x}^{-2} \frac{(1-\chi_{r}(L-2\lambda))}{L(L-2\lambda)} e^{it(L-2\lambda)}  P_c h \right\|_{L_x^\infty} 
  & \lesssim \lr{t}^{-1} {\big\|\widetilde{h} \big\|}_{H_k^1}.
\end{align}
Finally, the refined version of \eqref{local-estimate-time-cut-off} holds true:
for all $n\in\mathbb{Z}$ such that $2^n \ll 2\lambda-1$
\begin{equation}
\label{local-estimate-time-cut-off-refined}
\left\| \lr{x}^{-2}  L^{-1} e^{it(L-2\lambda)} \phi_{\leq n}(L-2\lambda) P_c h 
  \right\|_{L_x^\infty} \lesssim \Big( \frac{2^{n/2}}{\lr{t}} + \frac{2^{-n}}{\lr{t}^2} \Big) 
  {\big\| \wt{h} \big\|}_{H^1_k}
\end{equation}
where $\phi_{\leq n}$ is a Littlewood-Paley cutoff at scale below $2^{n+1}$ 
as defined in \ref{ssec:notations}. The estimate \eqref{local-estimate-time-cut-off-refined}
also holds true if we replace $\phi_{\leq n}$ with $\phi_{n}$
\end{lemma}

\begin{proof}
For $t\leq 1$ the estimates follow directly by Cauchy-Schwarz so we only consider $t>1$ in 
the proofs below.

\medskip
{\it Proof of \eqref{local-estimate-time-cut-off}.}
To prove \eqref{local-estimate-time-cut-off}, we use integration
by parts via $e^{it \lr{k}} = (itk)^{-1} \lr{k} \partial_k e^{it \lr{k}}$ and write
\begin{align}\label{letc-pr1}
\begin{split}
I_+ & = -\frac{1}{it }
\int_{ \mathbb{R}^+} \partial_k (k^{-1} T(k) ) \chi_{r}(\lr{k}-2\lambda) e^{ixk} m_+(x, k) e^{it\lr{k}} \widetilde{h}(k)dk\\[2pt]
& \quad - \frac{x}{t }  
\int_{ \mathbb{R}^+} k^{-1} T(k) e^{ixk} m_+(x, k) \chi_{r}(\lr{k}-2\lambda) \widetilde h(k) e^{it\lr{k}}dk\\[2pt]
& \quad - 
\frac{1}{it}\int_{ \mathbb{R}^+} k^{-1} T(k) e^{ixk} \partial_k m_+(x, k) \chi_{r}(\lr{k}-2\lambda) \tilde h(k) e^{it\lr{k}}dk\\[2pt]
& \quad - 
\frac{1}{it}\int_{ \mathbb{R}^+} k^{-1} T(k) e^{ixk} m_+(x, k) \chi_{r}(\lr{k}-2\lambda) \partial_k\tilde h(k) e^{it\lr{k}}dk\\[2pt]
& \quad - 
\frac{1}{it}\int_{ \mathbb{R}^+} k^{-1} T(k) e^{ixk} 
  m_+(x, k) r^{-1} \chi'_{r}(\lr{k}-2\lambda) \frac{k}{\lr{k}} \tilde h(k) e^{it\lr{k}}dk
 \\
 & = : I_{+,6} + I_{+,7} + I_{+,8} + I_{+,9}+I_{+,10}.
\end{split}
\end{align}
By Lemmas \ref{lem:oscillatory-Jost} and \ref{lem:T-R-prop}, 
and applying H\"older's inequality, and Sobolev's embedding, 
we directly obtain, for all $j=6,\dots ,10$, 
\begin{equation*}
\lr{x}^{-2}|I_{+,j}| 
  \lesssim 
  t^{-1} \|\widetilde h\|_{H_k^1}. 
\end{equation*}

\medskip
{\it Proof of \eqref{local-estimate-time-cut-off-internal-mode}.}
We use a similar strategy. Set $\chi^c_r=1-\chi_r$ and write 
\begin{equation}
 L^{-1}(L-2\lambda)^{-1} e^{it(L-2\lambda)} \chi^c_r(L-2\lambda) P_c h(x) = \int \lr{k}^{-1}(\lr{k}-2\lambda)^{-1}e^{it(\lr{k}-2\lambda)}\chi^c_r(\lr{k}-2\lambda) \psi(x, k)  \widetilde h(k) dk.
\end{equation}
The desired estimate clearly holds for $0< t\lesssim 1$ using Cauchy-Schwarz.
For $t\gtrsim 1$, we have 
\begin{align*}
\int \lr{k}^{-1}(\lr{k}-2\lambda)^{-1}e^{it(\lr{k}-2\lambda)} \chi^c_r(\lr{k}-2\lambda)\psi(x, k) \widetilde h(k) dk= \bigg(\int_{\R^+}+\int_{\R^-}\bigg)\cdots dk=L_++L_-.   
\end{align*}
We will only focus our attention on $L_+$ since $L_-$ can be treated similarly.
By \eqref{eigenfunctions} and integration by parts in $k$, we find \begin{equation*}
\begin{aligned}
L_+ & = -\frac{1}{it }
\int_{\mathbb{R}^+} \partial_k (k^{-1} T(k) )(\lr{k}-2\lambda)^{-1} m_+(x, k) e^{ixk}e^{it(\lr{k}-2\lambda)}\chi^c_r(\lr{k}-2\lambda) \widetilde{h}(k)dk\\[2pt]
& \quad + 
\frac{1}{it}\int_{ \mathbb{R}^+}  T(k) \lr{k}^{-1}(\lr{k}-2\lambda)^{-2} m_+(x, k)e^{ixk} e^{it(\lr{k}-2\lambda)}\chi^c_r(\lr{k}-2\lambda) \widetilde{h}(k)dk\\[2pt]
& \quad - 
\frac{1}{it}\int_{ \mathbb{R}^+} k^{-1} T(k)  (\lr{k}-2\lambda)^{-1}\partial_k m_+(x, k) e^{ixk}e^{it(\lr{k}-2\lambda)}\chi^c_r(\lr{k}-2\lambda) \widetilde{h}(k)dk\\[2pt]
& \quad - \frac{x}{t }  
\int_{ \mathbb{R}^+} k^{-1} T(k) (\lr{k}-2\lambda)^{-1} m_+(x, k)e^{ixk} e^{it(\lr{k}-2\lambda)}\chi^c_r(\lr{k}-2\lambda) \widetilde{h}(k)dk\\[2pt]
& \quad - 
\frac{1}{it}\int_{ \mathbb{R}^+} \lr{k}^{-1} T(k) (\lr{k}-2\lambda)^{-1}e^{ixk}
  m_+(x, k) e^{it(\lr{k}-2\lambda)} \partial_k \chi^c_r(\lr{k}-2\lambda) \widetilde{h}(k)dk\\[2pt]
&\quad - 
\frac{1}{it}\int_{ \mathbb{R}^+} k^{-1} T(k) (\lr{k}-2\lambda)^{-1}e^{ixk} m_+(x, k) e^{it(\lr{k}-2\lambda)}\chi^c_r(\lr{k}-2\lambda) \partial_k\widetilde{h}(k)dk\\[2pt]
& = : L_{+,1} + L_{+,2} + L_{+,3} + L_{+,4}+L_{+,5}+L_{+,6}.
\end{aligned}
\end{equation*}
Note that boundary terms do not appear since $\wt{h}(0)=0$.
As above, we rely on Lemmas \ref{lem:oscillatory-Jost} and \ref{lem:T-R-prop}, and can estimate
\begin{align*}
|\lr{x}^{-2}L_{1,+}|&\lesssim t^{-1}  \int_{\mathbb{R}^+} |\partial_k (k^{-1} T(k))||\lr{k}-2\lambda|^{-1} |\lr{x}^{-2}m_+(x, k)|\chi^c_{r}(\lr{k}-2\lambda) |\widetilde{h}(k)|dk  \\
&\lesssim t^{-1}\bigg(\int_{\R^+}\lr{k}^{-4} |\lr{k}-2\lambda|^{-2}(\chi^c_r)^2(\lr{k}-2\lambda)dk\bigg)^{1/2}\|\wt{h}\|_{L^2_k}\lesssim t^{-1}\|\wt{h}\|_{L^2_k}.
\end{align*}
Similarly, we can obtain 
\begin{equation*}
\lr{x}^{-2}\big( |L_{+,2}| + |L_{+,3}|+|L_{+,4}| + |L_{+,6}| \big) \lesssim 
  t^{-1}\|\wt{h}\|_{H^1_k}.
\end{equation*}
Finally, for the penultimate term, using that $|\lr{k} - 2\lambda| \approx r$ we have
\begin{align*}
|\lr{x}^{-2}L_{+,5}|&\lesssim t^{-1}  \int_{\mathbb{R}^+} \lr{k}^{-1}|T(k)||\lr{k}-2\lambda|^{-1} |\lr{x}^{-2}m_+(x, k)| |\widetilde{h}(k)\left|\left(\chi^c_r\right)'(\lr{k}-2\lambda)\right|dk  
\\
&\lesssim t^{-1} 
\|\wt{h}(t)\|_{L^2_k}.
\end{align*}

\medskip
{\it Proof of \eqref{local-estimate-time-cut-off-refined}.}
To obtain this more refined version of \eqref{local-estimate-time-cut-off}, 
we use again integration by parts in $k$ and reduce matters 
to estimating the same terms in \eqref{letc-pr1}.
Using again Lemmas \ref{lem:oscillatory-Jost} and \ref{lem:T-R-prop}, 
and applying H\"older's inequality we can immediately see that all the terms except the last one satisfy
the desired bound:
\begin{equation*}
\lr{x}^{-2}|I_{+,j}|
  \lesssim 2^{n/2} t^{-1} \|\widetilde h\|_{H_k^1} ,\quad j=6,7,8,9.
\end{equation*}
For the last term we have
\begin{align}\label{letc-pr2}
|I_{+,10}(t,x)| \lesssim \frac{1}{t \cdot 2^n}
  \Big| \int_{ \mathbb{R}^+} k^{-1} T(k) e^{ixk} m_+(x, k)
  \phi'_{\leq n}(\lr{k}-2\lambda) \tilde h(k) e^{it\lr{k}} dk \Big|,
\end{align}
and we need to integrate by parts once more. Noticing that $|k|$ is
uniformly bounded away from zero on the support of the integrand in \eqref{letc-pr2}, 
we see that all the terms arising from the integration by parts can be estimated
as before, yielding acceptable bounds of the order of $t^{-2}2^{-n/2} {\| \wt{h}\|}_{H^1_k}$. 
with the exception of the term where $\partial_k$ hits the cutoff; this latter contributes
\begin{align*}
\lr{x}^{-2} \frac{1}{t^2 \cdot 2^{2n}} 
\Big| \int_{ \mathbb{R}^+} k^{-1} T(k) e^{ixk} m_+(x, k) \phi''_{\leq n}(\lr{k}-2\lambda) 
  \tilde h(k) e^{it\lr{k}}dk \Big| \lesssim
 \frac{1}{t^2 \cdot 2^{n}} {\| \tilde h(k) \|}_{L^\infty_k}.
\end{align*}
This is consistent with the right-hand side of \eqref{local-estimate-time-cut-off-refined}
and completes the proof of the lemma.
\end{proof}

\medskip
Finally, we record the following smoothing-type estimate. 
\begin{lemma}\label{lem:smoothing-effect}
Let $\mathcal{Q}:\mathbb R_x\times \mathbb R_k\rightarrow \mathbb C$
and $\Psi:\mathbb R \rightarrow \mathbb C$ be such that
\begin{equation*}
\sup_{x,k}|\mathcal{Q}(x,k)|<\infty, 
\quad\quad|\Psi(k)|\leq |k|/\lr{k}.
\end{equation*}
Let $1< p\leq 2$. Then, for all $t>0$, we have 
\begin{equation}\label{smooth-est}
\left\|\int_0^t\int_{\mathbb R}e^{-is\lr{k}}\overline{\Psi}(k)\overline{\mathcal{Q}}(x,k)
  F(s,x)dxds\right\|_{L^{p'}_k}
  \lesssim \|F\|_{L^1_x L^p_s([0,t])}.    
\end{equation}
\end{lemma}

A similar estimate has been proved in \cite{KP23}, see also \cite{CP24} in the Schr\"odinger setting.
We provide the details  following the same argument using the Hausdorff-Young inequality. 

\begin{proof}[Proof of Lemma \ref{lem:smoothing-effect}]
Let $p\in (1,2]$. By duality, it suffices to establish the estimate  
\begin{equation}
  \bigg\|\int_{\R} \mathcal{Q} (x,k)\Psi(k)e^{is\lr{k}}f(k)dk\bigg\|_{L^{\infty}_xL^{p'}_s}\lesssim \|f\|_{p}.  
\end{equation}
We set $\lr{k}=\lambda$ and assume $k=\sqrt{\lambda^2-1}$, since the case $k<0$ can treated similarly.  We have 
\begin{align*}
I(x,s)&=\int_{\R} \mathcal{Q} (x,k)\Psi(k)e^{is\lr{k}}f(k)dk
    \\
 &=\int_{1}^{\infty} e^{is\lambda}\mathcal{Q} (x,\sqrt{\lambda^2-1})\Psi(\sqrt{\lambda^2-1})f(\sqrt{\lambda^2-1})\frac{\lambda}{\sqrt{\lambda^2-1}}d\lambda  \\
 &=\widehat{\mathcal{F}}^{-1}(\lambda \mathbbm{1}_{[1,\infty)}(\lambda)\Gamma(x,\sqrt{\lambda^2-1}))(s)
\end{align*}
where 
$\Gamma(x,\lambda)= \mathcal{Q} (x,y)\Psi(y)\frac{f(y)}{y}$. Thus, 
applying the Hausdorff-Young inequality and the hypotheses of the lemma
we can estimate
\begin{align*}
 {\|I(x,s)\|}^p_{L^{\infty}_xL^{p'}_s} & \lesssim \sup_x \int_1^{\infty}\bigg|\mathcal{Q} (x,\sqrt{\lambda^2-1})\Psi(\sqrt{\lambda^2-1})f(\sqrt{\lambda^2-1})\frac{\lambda}{\sqrt{\lambda^2-1}}\bigg|^pd\lambda\\
 &\lesssim \int_1^{\infty}\bigg|f(\sqrt{\lambda^2-1})\bigg|^pd\lambda\lesssim \|f\|_{L^p}^{p}
\end{align*}
which is the desired conclusion.
\end{proof}


\begin{remark}\label{rem:smoothing-effect}
Throughout the manuscript, this lemma will often be applied  with the choice of
\begin{align*}
\mathcal{Q}(x,k)= \lr{x}^{-1} \cdot \begin{cases}
   m_+(x,k)e^{ixk} \quad\,\, \mbox{if}\,\,\, k\geq 0\\ 
   m_-(x,-k)e^{ixk} \,\,\, \mbox{if}\,\,\, k<0
\end{cases}   \quad \mbox{and}\quad  \Psi(k)=T(k)
\end{align*}
both of which satisfy the required assumptions according to Lemma \ref{lem:oscillatory-Jost} and Lemma \ref{lem:T-R-prop} or similar ones.
With the above choice we will need to put an $\lr{x}$ weight on the argument $F$ 
\end{remark}



\medskip
\section{Bootstrap for the discrete component}\label{sec:ODE}
In this section, we analyze the decay property of the discrete component $a(t)$.

\subsection{Main proposition and a useful lemma}
To start, we introduce the notation 
\begin{equation}
\label{rho-defn}
\rho(t) := |A(t)|^2, \quad \rho_0 := \rho(0) \approx \varepsilon_0^2, 
\end{equation}
and recall that $A$ satisfies the ODE \eqref{evolution-A}.
The main result of this section is the following bootstrap proposition:

\begin{proposition}
\label{prop:bootstrap-on-rho}
Assume that for $T>0$ we have 
\begin{equation}
\label{rho-assumption-bootstrap}
\sup\limits_{t \in [0, T]} \left| \left(1 + \rho_0 \frac{\Gamma}{\lambda} t \right) \frac{\rho(t)}{\rho_0} - 1 \right| < \frac{1}{2},
\end{equation}
and, recalling also \eqref{choice-of-csts} and the notation \eqref{phi-decomp}, assume 

\begin{equation}
\label{w-assumptions-bootstrap}
\begin{aligned}
&\sup\limits_{t \in [0, T]} \lr{t}^{-\delta_1} \|\partial_k \wt{h}(t)\|_{L^2_k} 
  \leq 2\varepsilon_0^{\alpha} ,
  \\
&\sup\limits_{t \in [0, T]} 
\lr{t}^{1/2-\delta}
\|L^{-1} w(t)\|_{L_x^\infty}  \leq 2\varepsilon_0^{\alpha}, 
\\
&\sup\limits_{t \in [0, T]} \|L^{-1} w(t)\|_{H_x^N}  \leq C\varepsilon_0, 
\\
& \sup\limits_{t \in [0, T]}
\lr{t}^{1-\delta_2} \left(\|\lr{x}^{-2} (L^{-1} e^{\pm itL})_Sf(t)\|_{L_x^\infty}
  + \|\lr{x}^{-2} (L^{-1} e^{\pm itL})_Rf(t)\|_{L_x^\infty} \right) 
  \leq 2\varepsilon_0^{\alpha}
\end{aligned}
\end{equation}
for some large constant $C>1$.
Then, there exists $\overline{\varepsilon}_0>0$ such that for 
$\varepsilon_0\leq \overline{\varepsilon}_0$  we have
\begin{equation}
\label{rho-conclusion-bootstrap}
\sup\limits_{t \in [0, T]} \left| \left(1 + \rho_0 \frac{\Gamma}{\lambda} t \right) \frac{\rho(t)}{\rho_0} - 1 \right| < \frac{1}{4}.
\end{equation}
\end{proposition}

\smallskip
\begin{remark}\label{remSobGagNir}
Using Sobolev-Gagliardo-Nirenberg to interpolate between the 
high Sobolev norm and the local decay a priori bounds in \eqref{w-assumptions-bootstrap}, we can estimate,
with $\theta = 2/N$
\begin{align}\label{a prioriSGN}
\begin{split}
\| \lr{x}^{-3} L^{-1} w(t) \|_{H^2_x} \lesssim \| \lr{x}^{-3} L^{-1} w(t) \|_{L_x^2}^{1-\theta} 
  \| L^{-1} w(t) \|_{H_x^N}^{\theta} \leq \min(\varepsilon^{\tilde{\alpha}} 
  \lr{t}^{-1+\widetilde{\delta_2}},\varepsilon_0)
\end{split}
\end{align}
where $\tilde{\alpha} := \alpha(1-2/N) + 2/N$ and $\widetilde{\delta_2}: = (-1+\delta_2)(1-2/N) + 1$.
\end{remark}




\medskip
To prove Proposition \ref{prop:bootstrap-on-rho}
we are going to make use of Lemma \ref{lemODEdamp} below, which gives sufficient conditions on 
the terms on the right-hand side of the ODE \eqref{evolution-A} so that they can be considered 
remainder with respect to the leading damping dynamics given by the ODE 
$$\dot{A}(t) \approx -\frac{\Gamma}{2\lambda} |A|^2(t) A(t) \qquad \Longrightarrow 
\qquad \rho(t) \approx \frac{\rho_0}{1+\frac{\Gamma}{\lambda} \rho_0 t}. $$

\begin{lemma}\label{lemODEdamp}
Let $z > 0$ be a solution of the ODE
\begin{align}\label{lemODEdampeq}
\dot{z}(t) = -\Gamma_0 z^2(t) + R(t), \qquad z(0) = \rho_0, 
  \qquad \Gamma_0 > 0,
\end{align}for $t\in[0,T]$. 
Define, for any $\Gamma >0$
\begin{align}\label{defzG}
z_{\Gamma}(t) = \dfrac{\rho_0}{1+\rho_0 \Gamma t} \approx_\Gamma \min\big( \rho_0, t^{-1} \big).
\end{align} 
Assume that the term $R$ in \eqref{lemODEdampeq} 
can be written as the sum
\begin{align}\label{lemODEdampR}
R = R_a + R_h 
\end{align} 
with the following properties:

\begin{itemize}

\smallskip
\item For short times we assume
\begin{align}\label{lemODEdampR0}
|R(t)| \leq  \rho_0^{3/2},\quad t \in[0, c\rho_0^{-1/2}]
\end{align} 
where $c$ is a sufficiently small absolute constant.

\smallskip
\item For larger times we assume
\begin{align}\label{lemODEdampRa}
\begin{split}
& \Big| \int_{t_1}^{t_2} R_a(s) ds \Big| \leq c_0 z_{\Gamma_0}(t_1), 
  \qquad t_1\leq t_2\in[c\rho_0^{-1/2},T],
\end{split}
\end{align}
where $c_0>0$ is a small enough absolute constant, 
and 
\begin{align} \label{lemODEdampRh}
& R_h = \frac{d}{dt}h, \qquad 
| h(t) | \lesssim z_{\Gamma_0}^{3/2}(t).
\end{align}\end{itemize}
Then, if $\rho_0$ is small enough, for all $t\in [0,T]$,
\begin{align}\label{lemODEdampconc}
\Big| \frac{z(t)}{z_{\Gamma_0}(t)} - 1 \Big| < \frac{1}{4}.
\end{align}

\end{lemma}

\medskip
\begin{proof}[Proof of Lemma \ref{lemODEdamp}]
In view of the assumption \eqref{lemODEdampR0},
a direct integration of \eqref{lemODEdampeq} gives
\begin{align}\label{lemODEdampshort}
|z(t) - \rho_0| \leq \frac{1}{10}\rho_0, \qquad 0 \leq t \leq c\rho_0^{-1/2},
\end{align}
for $c$ small enough, which matches \eqref{lemODEdampconc} on this short time interval.
It then suffices to look at times $t \geq t_0:=c\rho_0^{-1/2}$. 
Define
\begin{align}\label{lemODEdampZz}
Z(t) := z(t) + \int_t^T R_a(s) 
  \,d s - h(t),
\end{align}
so that
\begin{align}\label{lemODEdampZdot}
\dot{Z}(t) = - \Gamma_0 z^2(t).
\end{align}
Under the assumptions \eqref{lemODEdampRa}-\eqref{lemODEdampRh} we have
\begin{align}\label{lemODEdampZ-z}
\big| Z(t)-z(t)\big| \leq 
  C \big( c_0 z_{\Gamma_0}(t) + z_{\Gamma_0}^{3/2}(t) \big) \leq \bar{c} z_{\Gamma_0}(t), 
  \qquad \bar{c} \ll 1
\end{align}
where the smallness of $\bar{c}$ holds provided we choose $c_0$ and $\rho_0$ small enough.

We now perform a comparison argument to show that the desired conclusion \eqref{lemODEdampconc} holds.
At time $t=t_0$, in view of \eqref{lemODEdampshort}, we have 
$(9/10)\rho_0 \leq  z(t_0) \leq (11/10)\rho_0$ and therefore, in view of \eqref{lemODEdampZ-z},
\begin{align}
\frac{4}{5}z_{\Gamma_0}(t_0) < Z(t_0) < \frac{6}{5}z_{\Gamma_0}(t_0),
\end{align}
since $z_{\Gamma_0}(c\rho^{-1/2}) \approx \rho_0$.

Suppose that there is a first time $t_\ast > t_0$ such that $Z(t_\ast) = (6/5)z_{\Gamma_0}(t_\ast)$.
In particular, from \eqref{lemODEdampZ-z} we would have
\begin{align}\label{lemODEdampz>}
z(t_\ast) \geq |Z(t_\ast)| - \bar{c} z_{\Gamma_0}(t_\ast) \geq \big(\frac{6}{5}-\bar{c} \big)
  z_{\Gamma_0}(t_\ast).
\end{align}
It then follows from \eqref{lemODEdampZdot} and \eqref{lemODEdampz>} that
\begin{align}
\frac{d}{dt}\Big( {Z} - \frac{6}{5}z_{\Gamma_0} \Big) (t_\ast) 
& = - \Gamma_0 z^2(t_\ast) + \frac{6}{5}\Gamma_0 z^2_{\Gamma_0}(t_\ast) 
  \leq - \Big[ \big(\frac{6}{5}-\bar{c} \big)^2 - \frac{6}{5} \Big] 
  \Gamma_0 z_{\Gamma_0}^2(t_\ast) < 0,
\end{align}
provided $\bar{c}$ is small enough;
the last inequality contradicts the fact that $t_\ast$ 
is the first time at which $Z$ crosses $(6/5)z_{\Gamma_0}$.
Using again \eqref{lemODEdampZ-z} we obtain
\begin{align}
z(t) \leq Z(t) + \bar{c} z_{\Gamma_0}(t) < \frac{5}{4}z_{\Gamma_0}(t).
\end{align}

The lower bound can be proven similarly.
Indeed, suppose that there is a first time $t_\ast > t_0$ such that $Z(t_\ast) = (4/5)z_{\Gamma_0}(t_\ast)$.
In particular, from \eqref{lemODEdampZ-z} we would have
\begin{align}\label{lemODEdampz<}
z(t_\ast) \leq |Z(t_\ast)| + \bar{c} z_{\Gamma_0}(t_\ast) \leq \big(\frac{4}{5}+\bar{c} \big)
  z_{\Gamma_0}(t_\ast).
\end{align}
It then follows from \eqref{lemODEdampZdot} and \eqref{lemODEdampz<} that
\begin{align}
\frac{d}{dt}\Big( {Z} - \frac{4}{5}z_{\Gamma_0} \Big) (t_\ast) 
& = - \Gamma_0 z^2(t_\ast) + \frac{4}{5}\Gamma_0 z^2_{\Gamma_0}(t_\ast) 
  \geq \Big[ -\big(\frac{4}{5}+\bar{c} \big)^2 + \frac{4}{5} \Big] 
  \Gamma_0 z_{\Gamma_0}^2(t_\ast) > 0,
\end{align}
provided $\bar{c}$ is small enough;
the last inequality contradicts the fact that $t_\ast$ 
is the first time at which $Z$ crosses $(4/5)z_{\Gamma_0}$.
Using again \eqref{lemODEdampZ-z} we obtain
\begin{align}
z(t) \geq Z(t) - \bar{c} z_{\Gamma_0}(t) > \frac{3}{4}z_{\Gamma_0}(t).
\end{align}
This concludes the proof of the Lemma.
\end{proof}

\smallskip
We record here a variant of the above lemma which will be useful to 
derive asymptotics for the internal mode's amplitude in Subsection \ref{sec:Asymp-A}. 

\smallskip
\begin{remark}\label{rem:lemODE}
Let $z:[0,\infty) \mapsto \R$ be a global solution of \eqref{lemODEdampeq}, 
under the same assumptions of Lemma \ref{lemODEdamp} with $T=\infty$,
and assume in addition the inequality
\begin{align}\label{lemODEdampRaasy}
\begin{split}
& \Big| \int_{t_1}^{t_2} R_a(s) ds \Big| \leq c_0 z_{\Gamma_0}^{1+\gamma}(t_1),
  \qquad t_0 \leq t_1\leq t_2 \leq \infty
\end{split}
\end{align}
in place of \eqref{lemODEdampRa}, where $\gamma \in (0,1/2)0$, and $t_0 \gtrsim \rho_0^{-1}$.
Then, letting 
\begin{align}
y_0(t) = \frac{z(t_0)}{1+\Gamma_0 z(t_0)(t-t_0)} \qquad (\dot{y}_0 = -\Gamma_0y_0^2, \quad y_0(t_0) = z(t_0)),
\end{align}
we have the quantitative conclusion
\begin{align}\label{lemODEdampconcasy}
\Big| \frac{z(t)}{y_0(t)} - 1 \Big| \lesssim y_0^{\gamma}(t) \qquad t \geq t_0.
\end{align}
\end{remark}

\smallskip
\subsection{Estimates of the remainder terms}
Our aim now is to apply Lemma \ref{lemODEdamp} to $\rho=|A|^2$.
To this end, we first multiply \eqref{evolution-A} by $2\overline{A}$ and take the real part
to derive the equation
\begin{equation}\label{evolution-A-rho}
\frac{d}{dt}{\rho}(t) = - \frac{\Gamma}{\lambda}\rho^2(t) + 2\Re \big(P(t) \overline{A}(t) \big) + 2\Re(R(t) \overline{A}(t)), 
\end{equation}
where, we recall, that $P$ and $R$ are defined in Lemma \ref{lemevolA}.
We then recall the definitions from Lemma \ref{lemevolA} and,
for the sake of this proof, use the formulas \eqref{RsplitDF}-\eqref{R_Q-defn}, 
to re-organize these remainders as follows: 
\begin{align}\label{Rsplit'}
R = R_D + R_Q + R_E = R_D + R_Q + \sum_{j=0}^4 R_{E,j},
\end{align}
where $R_Q$ is defined in \eqref{R_Q-defn}, and 
we split the contributions from $R_E$ in \eqref{R_F-defn} into
\begin{align}
\label{R_E0-defn}
R_{E,0}(t) & := \frac{1}{i \lambda} e^{-i\lambda t} a(t) \int_\R \beta \varphi^2 
  \frac{1}{L} \Im \left( e^{itL} w_0 \right) dx,
\\
\label{R_E1-defn}
R_{E,1}(t) & := \frac{1}{2i \lambda} e^{-i\lambda t} 
  \int_\R \left(\frac{1}{L} \Im(w)\right)^2 \beta \varphi dx,
\\ 
\label{R_E2-defn}
R_{E,2}(t) & := \frac{1}{i \lambda} e^{-i\lambda t} a(t) \int_\R \beta \varphi^2 \frac{1}{L}
  \Im \left( e^{itL} R_w \right) dx,
\\
\label{R_E3-defn}
R_{E,3}(t) & := \frac{1}{i \lambda} e^{-i\lambda t} a(t) \int_\R \beta \varphi^2 \frac{1}{L}
  \Im \left( \rest - \int_0^t e^{i(t-s)L} P_c v^4(s) ds \right) dx,
\\
\label{R_E4-defn}
R_{E,4}(t) & := \frac{1}{2i \lambda} e^{-i\lambda t} \int_\R (a\varphi + v)^4 \varphi dx.
\end{align}

To obtain Proposition \ref{prop:bootstrap-on-rho} we want to 
apply Lemma \ref{lemODEdamp} with $z(t) = \rho(t)$ and $\Gamma_0 = \Gamma \lambda^{-1}$, 
by verifying that all the remainder terms satisfy one of the sufficient
conditions from Lemma \ref{lemODEdamp}.
In particular, we see that the conclusion of the Proposition \ref{prop:bootstrap-on-rho} will hold
provided we show the following estimates:

\setlength{\leftmargini}{2em}
\begin{itemize}

\smallskip
\item For short times $t \leq t_0=c\rho^{-1/2}_0$, the remainders terms appearing on
the right-hand side of \eqref{evolution-A-rho} satisfiy
\begin{align}\label{ODEdampestshort}
|P(t)|, \, |R(t)| \lesssim \rho_0;
\end{align}

\item Adopting the notation \eqref{defzG} for $\rho_{\Gamma_0}$, we have
\begin{align}
\label{ODEdampestP}
& \Re \big(P(t) \bar{A}(t) \big) = \frac{d}{dt}h_P(t), \qquad | h_P(t) | \lesssim \rho_{\Gamma_0}^{3/2}(t); 
\end{align}

\smallskip
\item With $R_D$ defined as in \eqref{R_D-defn} we have
\begin{align}
\label{ODEdampestRD}
& \Re \big(R_D(t) \bar{A}(t) \big) = \frac{d}{dt}h_D(t) + O(\rho_{\Gamma_0}^{5/2}(t)), 
  \qquad | h_D(t) | \lesssim \rho_{\Gamma_0}^{3/2}(t);
\end{align}

\smallskip
\item With $R_Q$ given by \eqref{R_Q-defn} we have
\begin{align}\label{ODEdampestRQ}
\Big| \int_{t_1}^{t_2} R_Q(s) \bar{A}(s) \, ds \Big| \lesssim z_{\Gamma_0}^{1+\delta_0}(t_1),
  \qquad t_1\leq t_2\in[c\rho_0^{-1/2},T].
\end{align}
Notice that we are estimating the full complex value and not just the real part, as this will
be useful when deriving asymptotics later in Section \ref{sec:Asymp-A}.


\smallskip
\item For the terms in \eqref{R_E0-defn}-\eqref{R_E4-defn} we have,
for all $t \geq c\rho_0^{-1/2}$,
\begin{align}
\label{ODEdampestRE0}
& \Big| \int_{t_1}^{t_2} R_{E,0}(s) \bar{A}(s) \, ds \Big| 
  \lesssim \rho_{\Gamma_0}^{3/2}(t_1), \qquad t_1\leq t_2\in[c\rho_0^{-1/2},T],
\\
\label{ODEdampestRE1}
& \big| R_{E,1}(t) \bar{A}(t) \big| \lesssim \rho^{9/4}_{\Gamma_0}(t),
\\
\label{ODEdampestRE2}
& \Big| \int_{t_1}^{t_2} \Re( R_{E,2}(s) \bar{A}(s)) \, ds \Big| 
  \lesssim \rho_{\Gamma_0}^{5/4}(t_1), 
  \qquad t_1\leq t_2\in[c\rho_0^{-1/2},T],
\\
\label{ODEdampestRE3}
& \Big| \int_{t_1}^{t_2} R_{E,3}(s) \bar{A}(s) \, ds \Big| \lesssim \, \rho^{5/4}_{\Gamma_0}(t_1)^,
  \qquad t_1\leq t_2\in[c\rho_0^{-1/2},T],
\\
\label{ODEdampestRE4}
& \big| R_{E,4}(t) \bar{A}(t) \big| \leq c_0 \, \rho^{5/2}_{\Gamma_0}(t).
\end{align}
Notice how we have kept the real part only for the term in \eqref{ODEdampestRE2}.

\end{itemize}

\medskip
\begin{proof}[Proof of \eqref{ODEdampestshort}]
We begin by proving the short-time estimate \eqref{ODEdampestshort}.
By the bootstrap assumption \eqref{rho-assumption-bootstrap} for $\rho=|A|^2$, we have
\begin{equation}\label{short-time-A-a-bound}
        |A(t)|+|a(t)|\lesssim \rho_0^{1/2} .
\end{equation}
From \eqref{P-defn} we see that $P(t)$ is quadratic in $A$ and 
therefore $|P(t)|\lesssim |A(t)|^2\lesssim \rho_0$.
Next, recall \eqref{Rsplit'}.
Since $R_D$ is cubic in $A,\overline A$, see \eqref{R_D-defn}, we have
$|R_D(t)|\lesssim |A(t)|^3\lesssim \rho_0^{3/2}$
which is stronger than what is needed.  

For \(R_{E,0}\), using \eqref{R_E0-defn}, the localization of \(\beta\varphi^2\), 
and the assumption on the data \eqref{mainthmdata}, we get
\begin{equation}\label{short-time-RE0}
\begin{aligned}
        |R_{E,0}(t)|
        &\lesssim
        |a(t)|\,
        \|L^{-1}e^{itL}w_0\|_{H^1_x}       \lesssim
        \rho_0 .
\end{aligned}
\end{equation}
For \(R_{E,1}\), using \eqref{R_E1-defn} and Sobolev embedding, we have
\begin{equation}\label{short-time-RE1}
\begin{aligned}
        |R_{E,1}(t)|
        &\lesssim
        \|L^{-1}w(t)\|_{H^1_x}^2     \lesssim
        \rho_0 .
\end{aligned}
\end{equation}
For \(R_{E,2}\), we first look at the contribution of \(R_w\). 
From \eqref{Rwextra}--\eqref{Rw}, and Lemma \ref{lemlocal-decay-0-plus}, we have
\begin{equation}\label{short-time-Rw-local-bound}
\begin{aligned}
        \sup_{0\leq s\leq t_0}
        \|\langle x\rangle^{-10}L^{-1}e^{isL}R_w(s)\|_{L^1_x}
        &\lesssim
        |A(0)|^2
        +
        \int_0^{t_0}|A(\tau)\dot A(\tau)|\,d\tau       
        +
        \int_0^{t_0}|A(\tau)|^2\,d\tau \lesssim
\rho_0^{1/2}.
\end{aligned}
\end{equation}
Therefore, from \eqref{R_E2-defn},
\begin{equation}\label{short-time-RE2}
\begin{aligned}
        |R_{E,2}(t)|
        &\lesssim
        |a(t)|\,
        \big\|\langle x\rangle^{-10}L^{-1}e^{itL}R_w(t)\big\|_{L^1_x}       \lesssim
        \rho_0 .
\end{aligned}
\end{equation}
For $R_{E,3}$, recall from \eqref{R_E3-defn} that this term contains all the Duhamel contributions
in \eqref{Rest}, except for the pure $v^4$ term. Thus the sources are schematically given by
\[
        a\,\beta\varphi v,\qquad
        \beta v^2,\qquad
        a^4\varphi^4,
\]
together with terms containing more factors of \(a\) or \(v\), which are smaller;
in particular, they are all localized in space.
We can then use the local decay estimate 
\eqref{local-decay-faster} 
to obtain
\begin{equation}\label{short-time-RE3}
\begin{aligned}
        |R_{E,3}(t)|
        &\lesssim
        |a(t)|
        \int_0^t
        \langle t-s\rangle^{-3/2}
        \Big(
          |a(s)|\|\langle x\rangle^{-3}v(s)\|_{H^2_x}
          +
          \|\langle x\rangle^{-3}v(s)\|_{H^2_x}^2
          +
          |a(s)|^4
        \Big)\,ds \lesssim \rho_0^{1/2} \cdot \rho_0.
\end{aligned}
\end{equation}
For $R_{E,4}$ in \eqref{R_E4-defn}, we use the localization of $\varphi$ and \eqref{short-time-A-a-bound}Jesu to obtain,
\begin{equation}\label{short-time-RE4}
\begin{aligned}
        |R_{E,4}(t)|
        &\lesssim
        |a(t)|^4+\|v(t)\|_{H^1_x}^4   \lesssim
        \rho_0^2 .
\end{aligned}
\end{equation}
Finally, from the definition \eqref{R_Q-defn}, using Sobolev's embedding and 
\eqref{w-assumptions-bootstrap},
we estimate
\begin{equation}\label{short-time-RQ}
|R_Q(t)| \lesssim |a(t)|
        \int_0^t \|v(s)^4\|_{H^1_x}\,ds                                     
        \lesssim \rho_0^{1/2} \,(c\rho_0^{-1/2})\,\varepsilon_0^4
        \lesssim \rho_0^2 .
\end{equation}
This concludes the proof of \eqref{ODEdampestshort}.
\end{proof}

\begin{remark}\label{remAdotbasic}
From the evolution equation \eqref{evolution-A-initial},
and using the bootstrap assumptions \eqref{rho-assumption-bootstrap}-\eqref{w-assumptions-bootstrap},
it is easy to see that arguments similar to the ones just given above, give
the basic bound
\begin{equation}
|\dot{A}(t)| \lesssim \rho_{\Gamma_0}(t).
\end{equation} 
In what follows we are going to use this bound multiple times, often without 
explicitly referring to it.
\end{remark}

\medskip
We now proceed to prove \eqref{ODEdampestP}-\eqref{ODEdampestRE4} under the assumptions of 
Proposition \ref{prop:bootstrap-on-rho}.

\medskip
\begin{proof}[{\bf Proof of \eqref{ODEdampestP}}]
Recall the definition of $P$ in \eqref{P-defn}. 
We denote $c:=  \int_\R \beta \varphi^3 dx$, and calculate
\begin{align}
\Re(P(t)\overline{A}(t)) & = c \, \Re\Big( \frac{1}{2i\lambda}e^{-i\lambda t} \overline{A}(t)  
  (A e^{i\lambda t} + \overline{A} e^{-i \lambda t})^2\Big)
  \\
  & = - c \frac{1}{4i\lambda}\Big(e^{i\lambda t}A(t) - e^{-i\lambda t} \overline{A}(t) \Big) \, a(t)^2
  \\
  & = c \frac{1}{4\lambda^2} \, \dot{a}(t) \, a(t)^2 = \frac{c}{12\lambda^2} \partial_t a^3(t)
\end{align}
where, for the last two identities, we have used the definition of the profile $A$ in \eqref{a-in-A-defn}.
This implies $\Re(P(t)\overline{A}(t)) = \partial_t h_P(t)$ with $h_P(t) = O(|a(t)|^3) = O(\rho_{\Gamma_0}(t)^{3/2})$
which concludes the proof of \eqref{ODEdampestP}.
\end{proof}

\medskip
\begin{proof}[{\bf Proof of \eqref{ODEdampestRD}}]
Recall that $R_D$ is defined in \eqref{R_D-defn} as a cubic expression in $A$.
To show \eqref{ODEdampestRD} we first observe that $R_D$ has non-vanishing oscillations, that is,
it can be written as a linear combination of terms of the form
\begin{align}
e^{i (\iota_1+\iota_2+\iota_3-1) \lambda s} A_{\iota_1}A_{\iota_2}A_{\iota_3}, \qquad \iota_j \in \{+,-\},
  \quad A_+:=A, \quad A_-:= \overline{A},
  \qquad \{\iota_1,\iota_2,\iota_3\}\neq \{+,+,-\}.
\end{align}
Then, we can write $R_D(t) \bar{A}(t)$ 
as a linear combination of terms of the form
\begin{align}
e^{i (\iota_1+\iota_2+\iota_3-1) \lambda t} (A_{\iota_1}A_{\iota_2}A_{\iota_3} \bar{A})(t)
 & = \frac{d}{dt} \Big[ \frac{1}{i(\iota_1+\iota_2+\iota_3-1) \lambda}
 e^{i (\iota_1+\iota_2+\iota_3-1) \lambda t} (A_{\iota_1}A_{\iota_2}A_{\iota_3} \bar{A})(t) \Big]
 \\
 & - \frac{1}{i (\iota_1+\iota_2+\iota_3-1) \lambda}
 e^{i (\iota_1+\iota_2+\iota_3-1) \lambda t} \frac{d}{dt} (A_{\iota_1}A_{\iota_2}A_{\iota_3} \bar{A})(t)
 \\
 & =: \frac{d}{dt}h_D(t) + O(|A(t)|^3 |\dot{A}(t)|)
\end{align}
with the natural definition of $h_D(t)$, satisfying $h_D(t) = O(|A|^4) = O(\rho_{\Gamma_0}^{2}(t))$ which is stronger than the desired bound.
Using that $\dot{A} = O(\rho_{\Gamma_0}(t))$ we have obtained \eqref{ODEdampestRD}.
\end{proof}

\medskip
\begin{proof}[{\bf Proof of \eqref{ODEdampestRQ}}]
We start by recalling the definition of $R_Q$ from \eqref{R_Q-defn} 
and that $a(t) = e^{i\lambda t}A(t) + e^{-i\lambda t}\overline{A}(t)$ and write 
\begin{align}
\nonumber
R_Q(s) \bar{A}(s) & =
 \frac{1}{i \lambda} e^{-i\lambda s} a(s) \int_\R \beta \varphi^2 \frac{1}{L} 
  \Im \left( \int_0^s e^{i(s-\tau)L} P_c v(\tau)^4 \, d\tau \right) dx \, \overline{A}(s)
  = J(s) + K(s),
  \\
  \label{ODEdampestQ1}
J(s) & :=\frac{1}{i \lambda}
  e^{-2i\lambda s} \overline{A}^2(s) \int_\R \beta \varphi^2 \frac{1}{L} 
  \Im \left( \int_0^s e^{i(s-\tau)L} P_c v(\tau)^4 \, d\tau \right) dx,
  \\
  \label{ODEdampestQ2}
K(s) & :=
  \frac{1}{i \lambda}
  |A(s)|^2 \int_\R \beta \varphi^2 \frac{1}{L} 
  \Im \left( \int_0^s e^{i(s-\tau)L} P_c v(\tau)^4 \, d\tau \right) dx . 
\end{align}
We then introduce the notation
\begin{align}\label{notQ-0}
Q(s,x) := \int_0^s e^{-i\tau L} P_c v(\tau)^4 \, d\tau
\end{align}
where we use the apex $\pm$ to denote $Q^+ = Q$, and $Q^- = \overline{Q}$, and recall the estimates 
\begin{align}\label{lem:Q-tilde'}
{\big\| \widetilde{Q^\pm}(s,\cdot) \big\|}_{H^1_k} 
  \lesssim \varepsilon_0^{3\al}\lr{s}^{8\delta},
  \qquad {\big\| Q^\pm(s,\cdot) \big\|}_{H^1_x} \lesssim \varepsilon_0^2 
\end{align}
where the first follows from Lemma \ref{lem:Q-tilde} and 
the second, which here is only used for short-times $s\leq \varepsilon_0^{-2}$, 
is easily obtained from the a priori bounds.


We first look at the term $J$ in \eqref{ODEdampestQ1}.
To prove the 
desired bound we need to exploit the oscillations in the $ds$ integral.
We introduce 
a smooth bump function $\chi_{r}(x) = \phi(x/r)$ (see the notation in $\S$\ref{ssec:notations}), 
with $r \ll 2\lambda-1$,
and write
\begin{align}\label{estRQJ}
\Big|  \int_{t_1}^{t_2} J(s) \, ds \Big| & 
  \lesssim J^{(1)}_+ + J^{(1)}_- + J^{(2)}_+ + J^{(2)}_-, 
\\
\label{estRQJ1}
J^{(1)}_\pm & := \Big| \int_{t_1}^{t_2} \overline{A}^2(s) e^{-2i\lambda s}
  \int_{\R} \beta \varphi^2 L^{-1} e^{\pm isL} \chi_{r} (L-2\lambda) Q^\pm(s,x) dx ds \Big|,
  \\
\label{estRQJ2}
J^{(2)}_\pm & :=  \Big| \int_{t_1}^{t_2} \overline{A}^2(s) e^{-2i\lambda s}
  \int_{\R} \beta \varphi^2 L^{-1} e^{\pm isL} \left(1 - \chi_{r} (L-2\lambda) \right) 
  Q^\pm(s,x) dx ds \Big|.
\end{align}

\smallskip
{\it Estimate of \eqref{estRQJ1}}. 
We only deal with $J^{(1)}_+$ since the other case has
a non-resonant phase and is easier to treat.
Here we are going to use the refined local decay estimate \eqref{local-estimate-time-cut-off-refined}.
Note that if we were to apply directly \eqref{local-estimate-time-cut-off}, together with
the bound \eqref{Q-tildeL2} on $Q$, we would not close the desired bound 
just because of an excess of powers of $s^\delta.$
We decompose the time integral using the time-cutoff $\tau_m$ defined in Subsection \ref{ssec:notations},
and then split the contributions into high and low modulations $L-2\lambda$ relative 
to the scale $2^{-m/4} \approx s^{-1/4}$:
\begin{align}
\label{estRQJ1.0}
& J^{(1)}_+ \leq J_{\mathrm{low}} + J_{\mathrm{high}}, 
\\
\label{estRQJ1.1}
& J_{\mathrm{low}} := \sum_m 
  \Big| \int_{t_1}^{t_2} \overline{A}^2(s)
  \int_{\R} \beta \varphi^2 L^{-1} e^{is(L-2\lambda)} 
  \phi_{\leq -m/4} (L-2\lambda) Q(s,x) dx \, \tau_m(s) ds \Big|,
  \\
\label{estRQJ1.2}
& J_{\mathrm{high}} :=  \sum_m 
  \Big| \int_{t_1}^{t_2} \overline{A}^2(s)
  \int_{\R} \beta \varphi^2 L^{-1} e^{is(L-2\lambda)} 
  \phi_{(-m/4,-D]} (L-2\lambda) Q(s,x) dx \, \tau_m(s) ds \Big|.
\end{align}
Note that, without loss of generality, we have disregarded the superfluous 
cutoff $\chi_{r} (L-2\lambda)$, but consistently restricted the size of $L-2\lambda$
to $2^{-m/4} \leq |L-2\lambda| \leq 2^{-D} \approx r$.

For the first term $J_{\mathrm{low}}$ 
we can directly invoke \eqref{local-estimate-time-cut-off-refined} 
combined with the bound on $Q$ from \eqref{Q-tildeL2}, to obtain
\begin{align}
J_{\mathrm{low}} & \lesssim \sum_m 
  \int_{t_1}^{t_2} \Rho(s)
  {\Big\| \lr{x}^{-2}  L^{-1} e^{is(L-2\lambda)} 
  \phi_{\leq -m/4} (L-2\lambda) Q(s) \Big\|}_{L^\infty_x} \, \tau_m(s) ds
  \\
  & \lesssim \int_{t_1}^{t_2} \Rho(s) \cdot 2^{-m/8} \cdot \frac{1}{\lr{s}} 
  {\big\| \wt{Q}(s) \big\|}_{H^1_k} \, \tau_m(s) ds
  \\
  & \lesssim \int_{t_1}^{t_2} \Rho(s) \cdot 2^{-9m/8} \cdot \varepsilon_0^{3\alpha} 2^{8\delta m}
  \, \tau_m(s) ds \lesssim 
  \Rho^{1+1/20}(t_1),
\end{align}
since $\delta$ is small enough. This is sufficient for \eqref{ODEdampestRQ}
by choosing $\delta_0$ small enough.

For the term $J_{\mathrm{high}}$ in \eqref{estRQJ1.2} we use the fact that $L-2\lambda$ is not too small
to gain by integrating by parts in $s$, which produces four terms:
\begin{align}
\label{estRQJ1B}
& J_{\mathrm{high}}\leq B_1 + B_2 + B_3 + B_4,
\\
\label{estRQJ1B1}
& B_1 := \sum_m 
  \Big| \overline{A}^2(s)
  \int_{\R} \beta \varphi^2 L^{-1} e^{is(L-2\lambda)} 
  \frac{\phi_{(-m/4,-D]} (L-2\lambda)}{L-2\lambda} Q(s,x) dx \, \tau_m(s) \Big|_{s=t_1}^{s=t_2}\Big|,
  \\
\label{estRQJ1B2}
& B_2 := \sum_m 
  \Big| \int_{t_1}^{t_2} \frac{d}{ds} \overline{A}^2(s)
  \int_{\R} \beta \varphi^2 L^{-1} e^{is(L-2\lambda)} 
  \frac{\phi_{(-m/4,-D]} (L-2\lambda)}{L-2\lambda} Q(s,x) dx \, \tau_m(s) ds \Big|,
\\
\label{estRQJ1B3}
& B_3 := \sum_m 
  \Big| \int_{t_1}^{t_2} \overline{A}^2(s)
  \int_{\R} \beta \varphi^2 L^{-1} e^{is(L-2\lambda)} 
  \frac{\phi_{(-m/4,-D]} (L-2\lambda)}{L-2\lambda} \frac{d}{ds} Q(s,x) dx \, \tau_m(s) ds \Big|,
\\
\label{estRQJ1B4}
& B_4 := \sum_m 
  \Big| \int_{t_1}^{t_2} \overline{A}^2(s)
  \int_{\R} \beta \varphi^2 L^{-1} e^{is(L-2\lambda)} 
  \frac{\phi_{(-m/4,-D]} (L-2\lambda)}{L-2\lambda} Q(s,x) dx \, \tau_m'(s) ds \Big|.
\end{align}
We will skip the details for the estimate of the boundary terms \eqref{estRQJ1B1}
and of the last term \eqref{estRQJ1B4}
since they can be treated very similarly to the term \eqref{estRQJ1B2}, for which we now provide full details.

Recalling that $|A(s)| \lesssim \Rho^{1/2}(s)$ and $|\dot{A}(s)| \lesssim \Rho(s)$,
and using \eqref{local-estimate-time-cut-off-refined} on each dyadic piece $\phi_n$ with $n \in [-m/4,-D] \cap \mathbb{Z}$, 
we can bound
\begin{align}
| B_2 | & \lesssim \sum_m 
  \int_{t_1}^{t_2} \Rho^{3/2}(s)
  {\Big\| \lr{x}^{-2} L^{-1} e^{isL} 
  \frac{\phi_{>-m/4}(L-2\lambda)}{L-2\lambda} Q(s) \Big\|}_{L^\infty_x} \, \tau_m(s) ds
\\
& \lesssim \int_{t_1}^{t_2} \Rho^{3/2}(s)
  \cdot 2^{m/8} \cdot \frac{1}{\lr{s}} {\|  \tilde{Q}(s) \big\|}_{H^1_k}\, \tau_m(s) ds
  \lesssim \Rho^{3/2}(t_1) t_1^{1/4} \varepsilon_0^{3\alpha}
\end{align}
having used \eqref{Q-tildeL2}, which suffices.

For the term \eqref{estRQJ1B3} we use Plancherel,
recall the definition of $Q$ in \eqref{notQ-0}, 
and the a priori decay bound in \eqref{w-assumptions-bootstrap} and estimate
\begin{align}
| B_3 | & \lesssim \sum_m 
  \int_{t_1}^{t_2} \Rho(s)
  {\Big\| L^{-1}  \phi_{>-m/4} \chi_{r} (L-2\lambda) 
  \frac{\chi_{>-m/4}(L-2\lambda)}{L-2\lambda} e^{isL} 
  \partial_s Q^\pm(s) \Big\|}_{L^2_x} \, \tau_m(s) ds
\\
& \lesssim \sum_m \int_{t_1}^{t_2} \Rho(s)
  \cdot 2^{m/4} \cdot {\|  v^4(s) \big\|}_{L^2_x} \, \tau_m(s) ds
  \\
  & \lesssim  \sum_m \int_{t_1}^{t_2} \Rho(s)
  \cdot 2^{m/4} \min\big( \varepsilon_0^4, \lr{s}^{-3/2+3\delta}\varepsilon_0^{1+\alpha} \big)
  \, \tau_m(s) ds
\end{align}
which is comfortably bounded by the right-hand side of
\eqref{ODEdampestRQ} for $\delta_0$ small enough. 

\smallskip
{\it Estimate of \eqref{estRQJ2}.}
Moving on to the term \eqref{estRQJ2} we concentrate on the case with the $+$ sign,
since the case with the $-$ sign is easier. 
We integrate by parts using that $|L-2\lambda| \gtrsim r$
on the support of the integral and write:
\begin{align}\label{estRQJ2split}
\begin{split}
J^{(2)}_+ & \lesssim C_1 + C_2 + C_3,
\\
C_1 & := \left| \overline{A}^2(s)
  \int_{\R} \beta \varphi^2 \frac{\left(1 - \chi_{r} (L-2\lambda) \right)}{L(L-2\lambda)} 
    e^{is(L-2\lambda)}  Q(s,x) dx \Big|_{s=t_1}^{s=t_2} \right|,
\\
C_2 & := \Big| \int_{t_1}^{t_2} \frac{d}{ds}\big(\overline{A}^2(s)\big)
  \int_{\R} \beta \varphi^2 \frac{\left(1 - \chi_{r} (L-2\lambda) \right)}{L(L-2\lambda)} 
    e^{is(L-2\lambda)} Q(s,x) dx ds \Big|,
\\
C_3 & := \Big| \int_{t_1}^{t_2} \overline{A}^2(s)
  \int_{\R} \beta \varphi^2 \frac{\left(1 - \chi_{r} (L-2\lambda) \right)}{L(L-2\lambda)} 
    e^{is(L-2\lambda)} \, \partial_s Q(s,x) dx ds \Big|.
\end{split}
\end{align}
To control the boundary terms we apply \eqref{local-estimate-time-cut-off-internal-mode} 
from Lemma \ref{lem:local-estimate-time-cut-off}, 
and combine it with the bound from Lemma \ref{lem:Q-tilde} to obtain
\begin{align}
|C_1| & \lesssim \left| \overline{A}^2(s)
  \int_{\R} \beta \varphi^2 \frac{\left(1 - \chi_{r} (L-2\lambda) \right)}{L(L-2\lambda)} 
  e^{is(L-2\lambda)}  Q(s,x) dx \Big|_{s=t_1}^{s=t_2} \right|
  \lesssim \rho_{\Gamma_0}(t_1) \lr{t_1}^{-1} \cdot \varepsilon_0^{3\alpha} t_1^{8\delta} 
\end{align}which suffices. 
Using $\dot{A} = O(\rho_{\Gamma_0})$, and combining again 
\eqref{local-estimate-time-cut-off-internal-mode} and Lemma \ref{lem:Q-tilde}, we can bound
\begin{align}
C_2 & \lesssim \int_{t_1}^{t_2} \rho_{\Gamma_0}^{3/2}(s)
   \cdot \lr{s}^{-1} \cdot \varepsilon_0^{3\alpha} s^{8\delta} ds
    \lesssim \rho_{\Gamma_0}^{3/2-8\delta}(t_1)
\end{align}which is again more than sufficient.
Finally, for the last term in \eqref{estRQJ2split}, we use the a priori bounds 
on $v = \Im (L^{-1}w)$ from \eqref{w-assumptions-bootstrap} and H\"older to get
\begin{align*}
{\| \partial_s Q(s,x) \|}_{L^2_x} = {\big\| v(s)^4 \big\|}_{L^2_x} 
  \lesssim \min\big( \varepsilon_0^{1+3\alpha} \lr{s}^{-3/2+3\delta}, \varepsilon_0^4\big)
\end{align*}
and then bound
\begin{align}
\begin{split}
C_3 & \lesssim \int_{t_1}^{t_2} \rho_{\Gamma_0}(s)
  {\Big\|\frac{\left(1 - \chi_{r} (L-2\lambda) \right)}{L(L-2\lambda)} 
  e^{is(L-2\lambda)} \, \partial_s Q(s,x) \Big\|}_{L^2} ds
  \\
  & \lesssim \int_{t_1}^{t_2} \rho_{\Gamma_0}(s) \cdot 
  \min \big( \varepsilon_0^{1+3\alpha} \lr{s}^{-3/2+3\delta}, \varepsilon_0^4\big) \, ds 
  \lesssim \rho^{3/2-3\delta}_{\Gamma_0}(t_1)
\end{split}
\end{align}
which is much stronger than needed. We have obtained \eqref{ODEdampestRQ}.

\smallskip
{\it Estimate of \eqref{ODEdampestQ2}}.
The term $K$ in \eqref{ODEdampestQ2} can be treated similarly to the term $J$, and is in fact easier 
since the time oscillations are given by $e^{\pm i sL}$ whose phase never vanishes.
In particular, integrating by parts in $s$ as above we can directly adapt the estimates above
to obtain
\begin{align}
\Big| \int_{t_1}^{t_2} |A(s)|^2
  \int_{\R} \beta \varphi^2 L^{-1} e^{\pm isL} Q^\pm(s,x) dx ds \Big|
  \lesssim \rho^{3/2-8\delta}_{\Gamma_0}(t_1).
\end{align}
This concludes the proof of \eqref{ODEdampestRQ}.
\end{proof}

\medskip
\begin{proof}[{\bf Proof of \eqref{ODEdampestRE0}}]

Using \eqref{a-in-A-defn} in \eqref{R_E0-defn}, and recalling that
$\theta=P_c(\beta\varphi^2)$, we write
\begin{align*}
R_{E,0}(s)\bar A(s)
  =\frac{1}{i\lambda}\big(\rho(s)+e^{-2i\lambda s}\bar A^2(s)\big)
    \Im\int_\R\theta L^{-1}e^{isL}w_0\,dx.
\end{align*}
Expanding the imaginary part and taking complex conjugates of the
resulting integrals, it suffices to estimate by $\rho_{\Gamma_0}^{3/2}(t_1)$
the three terms
\begin{align}
\label{eq:3pieces}
\begin{aligned}
D_1 & := \int_{t_1}^{t_2}\rho(s)
    \int_\R\theta L^{-1}e^{isL}w_0\,dx\,ds,\\
D_2 & := \int_{t_1}^{t_2}\bar A^2(s)
    \int_\R\theta L^{-1}e^{is(L-2\lambda)}w_0\,dx\,ds,\\
D_3 & := \int_{t_1}^{t_2}A^2(s)
    \int_\R\theta L^{-1}e^{is(L+2\lambda)}w_0\,dx\,ds.
\end{aligned}
\end{align}
Recall that $c\rho_0^{-1/2}\leq t_1\leq t_2\leq T$.

We first consider the second term from \eqref{eq:3pieces} 
and split it into two parts by inserting the localization $\chi_r(L-2\lambda)$, with the same cutoff as in Lemma \ref{lem:local-estimate-time-cut-off}.
We let $D_{2,r}$ be the contribution corresponding to the localization,
and denote $D_{2,nr}$ the complementary contribution.
By Plancherel, we can write
\begin{align*}
D_{2,r}:= \int_\R\theta L^{-1}e^{is(L-2\lambda)}\chi_r(L-2\lambda)w_0\,dx
=\int_\R e^{is(\lr{k}-2\lambda)}\frac{\chi_r(\lr{k}-2\lambda)}{\lr{k}}\overline{\wt{\theta}(k)}\wt{w}_0(k)\,dk.
\end{align*}Since $r\ll 2\lambda-1$, the support of the integrand is compact
and bounded away from $k=0$. Motivated by smoothing estimates  in \cite{GC21},  we integrate by parts in $k$ to see that
\begin{align*}
s\int_\R e^{is(\lr{k}-2\lambda)}
\frac{\chi_r(\lr{k}-2\lambda)}{\lr{k}}\overline{\wt{\theta}(k)}\wt{w}_0(k)\,dk
=i\int_\R e^{is(\lr{k}-2\lambda)}\partial_k\left(\frac{\chi_r(\lr{k}-2\lambda)}{k}\overline{\wt{\theta}(k)}\wt{w}_0(k)\right)dk.
\end{align*}
On each half-line, the change of variables $\xi=\lr{k}-2\lambda$
has Jacobian $|d\xi/dk|=|k|/\lr{k}$, and using Plancherel in time yields
\begin{align*}
&\left\|\int_{\R_\pm}e^{is(\lr{k}-2\lambda)}\partial_k\left(\frac{\chi_r(\lr{k}-2\lambda)}{k}\overline{\wt{\theta}(k)}\wt{w}_0(k)\right)dk\right\|_{L^2_s(\R)}^2\\
&\qquad=2\pi\int_{\R_\pm}\frac{\lr{k}}{|k|}\left|\partial_k\left(
\frac{\chi_r(\lr{k}-2\lambda)}{k}
\overline{\wt{\theta}(k)}\wt{w}_0(k)\right)\right|^2dk
  \lesssim\varepsilon_0^2.
\end{align*}
Summing the estimates on the two half-lines, we obtain
\begin{align*}
\left\|s\int_\R\theta L^{-1}e^{is(L-2\lambda)}
\chi_r(L-2\lambda)w_0\,dx\right\|_{L^2_s(\R)}\lesssim\varepsilon_0.
\end{align*}
Consequently, by Cauchy--Schwarz and the bootstrap bound on $A$, we obtain
\begin{align*}
| D_{2,r} | & \lesssim \left|\int_{t_1}^{t_2}\bar A^2(s) \int_\R\theta L^{-1}e^{is(L-2\lambda)}\chi_r(L-2\lambda)w_0\,dx\,ds\right|\\
&\qquad\lesssim\varepsilon_0
\left(\int_{t_1}^{t_2}\frac{\Rho(s)^2}{s^2}\,ds\right)^{1/2}\lesssim\varepsilon_0\Rho(t_1)t_1^{-1/2}.
\end{align*}

For the complementary part, we first note that, using the localization
of $\theta$ and \eqref{local-estimate-time-cut-off-internal-mode} we have
\begin{align}\label{locestuse}
\left|\int_\R\theta \frac{1-\chi_r(L-2\lambda)}{L(L-2\lambda)}
e^{is(L-2\lambda)}w_0\,dx\right| \lesssim\varepsilon_0\lr{s}^{-1}.
\end{align}
Integration by parts in $s$ gives
\begin{align*}
D_{2,nr} & = \int_{t_1}^{t_2}\bar A^2(s)
\int_\R\theta L^{-1}e^{is(L-2\lambda)} (1-\chi_r(L-2\lambda))w_0\,dx\,ds\\
& =\left[\frac{\bar A^2(s)}{i}
\int_\R\theta \frac{1-\chi_r(L-2\lambda)}{L(L-2\lambda)}e^{is(L-2\lambda)}w_0\,dx\right]_{s=t_1}^{t_2}\\
&\quad-\frac{2}{i}\int_{t_1}^{t_2}\bar A(s)\overline{\dot A(s)} \int_\R\theta\frac{1-\chi_r(L-2\lambda)}{L(L-2\lambda)}
e^{is(L-2\lambda)}w_0\,dx\,ds.
\end{align*}
Using $|\dot A|\lesssim\Rho$ from Remark \ref{remAdotbasic}, and \eqref{locestuse}, we bound 
\begin{align}\label{D2nrbound}
|D_{2,nr}| \lesssim \Rho(t_1) \, \varepsilon_0 t_1^{-1}
    +\int_{t_1}^{t_2} \Rho(s)^{3/2} \, \varepsilon_0 s^{-1} \,ds
    \lesssim \Rho^{3/2}(t_1)
\end{align}
as desired.

The first and third integrals from \eqref{eq:3pieces} are easier to deal with,
and can be estimated by means of the same integration by parts argument above, 
with $(\bar A^2,L-2\lambda,1-\chi_r(L-2\lambda))$ replaced by
$(\rho,L,1)$ and $(A^2,L+2\lambda,1)$, respectively.  In these cases,
\eqref{local-l-infty-decay} gives
\begin{align*}
\left|\int_\R\theta L^{-2}e^{isL}w_0\,dx\right|+\left|\int_\R\theta\frac{e^{is(L+2\lambda)}}{L(L+2\lambda)}w_0\,dx\right|
\lesssim\varepsilon_0\lr{s}^{-1},
\end{align*}
and from Remark \ref{remAdotbasic} we have
$$|\dot\rho(s)|+|\partial_s A^2(s)|
\lesssim\Rho(s)^{3/2},$$
so that the same bound in \eqref{D2nrbound} follows for the two integrals $D_1$ and $D_3$
in \eqref{eq:3pieces}.

\end{proof}

\medskip
\begin{proof}[{\bf Proof of \eqref{ODEdampestRE1}}]

Recall the definition of $R_{E,1}$ in \eqref{R_E1-defn},
let 
\begin{align}\label{dampestRE1prI}
I(t) := \int_\R \left(\frac{1}{L} \Im(w(t))\right)^2 \beta \varphi dx, 
\end{align}
and observe that it suffices to show that 
\begin{align}\label{dampestRE1pr1}
|I(t)| \lesssim \rho_{\Gamma_{0}}^{7/4}(t).
\end{align}
We begin by noting that 
\begin{align}\label{dampestRE1pr2}
|I(t)| \lesssim \min \big( \|\lr{x}^{-5} L^{-1} w(t)\|_{L_x^\infty}^2, 
  \|\lr{x}^{-5}L^{-1}w(t)\|_{L_x^2}^2 \big),
\end{align}
and subdivide the proof of \eqref{dampestRE1pr1} into cases depending on the size
of $t$ relative to $\varepsilon_0$. We let $c$ be a sufficiently small absolute constant.


\smallskip
\noindent
{\it Case 1: $c\varepsilon_0^{-1} \leq t \leq \varepsilon_0^{-2}$.}
For this range of times we claim that 
\begin{align}\label{dampestRE1pr3}
{\| \lr{x}^{-5} L^{-1} w(t) \|}_{L^2} \lesssim \varepsilon_0^2\approx \rho_{\Gamma_0}(t). 
\end{align}
Using this bound in \eqref{dampestRE1pr2} is more than sufficient to obtain 
the desired \eqref{dampestRE1pr1}. 
Let us now verify \eqref{dampestRE1pr3}. 
Recall Duhamel's formula \eqref{duhamel-for-w}:
\begin{align}
\label{duhamel-for-w'}
w(t) & = e^{itL}w(0) +  \int_0^t e^{i(t-s)L} 
  P_c \left[ \beta (a(s) \varphi + v(s))^2 + (a(s) \varphi + v(s))^4 \right] ds.
\end{align}
We look at each of terms in \eqref{duhamel-for-w'}.
Using the linear estimate \eqref{local-l-infty-decay},
\eqref{derivative-in-fourier-bound} and the assumption on the data \eqref{mainthmdata},
we have
\begin{align}\label{dampestRE1pr5}
{\| \lr{x}^{-5} L^{-1}e^{itL}w(0) \|}_{L^2} \lesssim t^{-1} \varepsilon_0 \lesssim \varepsilon_0^2. 
\end{align}
For the quadratic terms we use the improved local decay estimate \eqref{local-decay-faster}
to see that
\begin{align}
{\Big\| \lr{x}^{-5} L^{-1} \int_0^t e^{i(t-s)L} P_c[\beta( \varphi a(s) + v(s) )^2] ds \Big\|}_{L^2}
\lesssim
  \int_0^t \lr{s-t}^{-3/2} \big( a^2(s) + {\| v(s) \|}_{H^2}^2 \big) ds \lesssim \varepsilon_0^2
\end{align}
since $a \approx \varepsilon_0$ for short times and having used the a priori bounds 
on the Sobolev norm in \eqref{w-assumptions-bootstrap}.
For the quartic terms in $a$, using again \eqref{local-decay-faster},   we have
\begin{align}
{\Big\| \lr{x}^{-5} L^{-1}\int_0^t e^{i(t-s)L} P_c (a(s) \varphi)^4 ds  \Big\|}_{L^2}
  \lesssim \int_0^t \lr{s-t}^{-3/2} a^4(s) ds \lesssim \varepsilon_0^4
\end{align}
while, just using Sobolev's embedding,  we can estimate the quartic terms in $v$ by
\begin{align}
{\Big\| L^{-1}\int_0^t e^{i(t-s)L} P_c (v(s))^4 ds  \Big\|}_{L^2}
 \lesssim t \, \sup_{s\in[0,t]} {\| v(s)^4 \|}_{H^1} \lesssim \varepsilon_0^2. \end{align}
The other quartic terms are localized products of  $a$ and $v$ and can be treated as above.

\smallskip
\noindent
{\it Case 2: $t \geq \varepsilon_0^{-2}$.}
In the case of large times, the a priori improved local decay estimate in
\eqref{w-assumptions-bootstrap} gives
\begin{align}
\|\lr{x}^{-5} L^{-1} w(t)\|_{L_x^\infty}^2 
  \lesssim \big( \varepsilon_0^{\alpha} \lr{t}^{-1+\delta_2} \big)^2
  \lesssim  \varepsilon_0^{2\alpha} t^{-2+2\delta_2}.
\end{align}
This bound is smaller than the right-hand side of \eqref{dampestRE1pr1}.
The proof of \eqref{ODEdampestRE1} is complete.
\end{proof}

\medskip
\begin{proof}[{\bf Proof of \eqref{ODEdampestRE2}}]
Consistently with the definition of the term \eqref{R_E2-defn}, let us introduce, just for the
proof of these bounds, the notation
\begin{align}\label{R_E2-defn'}
T[F](t) & := \frac{1}{i \lambda} e^{-i\lambda t} a(t) \int_\R \beta \varphi^2
  L^{-1} \Im \left( e^{itL} F(t) \right) dx,
\end{align}
so that $R_{E,2}(t) = T[R_w](t)$. 
According to \eqref{Rwextra}-\eqref{Rw}, we split
\begin{align}
R_w & = R_{w,1} + R_{w,2},
\\
\label{Rw1}
R_{w,1}(t) & := - 2i \int_0^t \frac{e^{-is (L-2\lambda)} }{L-2\lambda + i0^+}\theta \, A(s)\dot{A}(s) ds, 
  \qquad \theta := P_c \left( \beta \varphi^2 \right),
\\
\label{Rw2}
R_{w,2}(t) & := -\frac{i A^2(0)}{L-2\lambda + i0^+} \theta
  + 2 \int_0^t e^{-isL} |A(s)|^2 \theta~ds + \int_0^t e^{-is(L+2\lambda)} \overline{A(s)}^2 \theta~ds,
\end{align}
and estimate the two contributions separately. 

Using local decay for the regularized resolvent from Lemma \ref{lemlocal-decay-0-plus}, we see that 
\begin{align}\label{estER21}
\Big| T[R_{w,1}](t) \bar{A}(t) \Big| 
& 
  \lesssim |a(t)|\cdot|A(t)|\int_0^t \lr{t-s}^{-3/2} |A(s)| |\dot{A}(s)| ds 
  \lesssim \rho_{\Gamma_0}^{5/2}(t). 
\end{align}
which is stronger than the desired bound.

To deal with the contributions from $R_{w,2}$ we first integrate by parts in $s$ in 
the integrals in \eqref{Rw2} and write
\begin{align}\label{Rw2split}
\begin{split}
R_{w,2} & = I + II + III,
\\
I & := -\frac{i}{L-2\lambda + i0^+} \theta \, A^2(0) + \frac{2}{iL} \theta \, |A(0)|^2 
  + \frac{1}{i(L+2\lambda)} \theta \, \overline{A(0)}^2,
\\
II(t)& := \frac{2}{iL} \int_0^t e^{-isL}  \theta \frac{d}{ds}|A(s)|^2~ds 
  + \frac{1}{i(L+2\lambda)} \int_0^t e^{-is(L+2\lambda)}  \theta \frac{d}{ds}\overline{A(s)}^2~ds,
\\
III(t) & :=- \frac{2}{iL} e^{-it L}\theta \, |A(t)|^2
  - \frac{e^{-2i\lambda t}}{i(L+2\lambda)} e^{-itL}\theta \, \overline{A(t)}^2. 
\end{split}
\end{align}
Using Lemma \ref{lem:stronger-local-decay} and Lemma \ref{lemlocal-decay-0-plus},
 we obtain the bounds 
\begin{align}
& \Big| T[I](t) \overline{A}(t) \Big| 
  \lesssim \lr{t}^{-3/2} |a(t)|\cdot |A(t)| \cdot |A(0)|^2 
  \lesssim \rho_{\Gamma_0}^{9/4}(t),
\\
& \Big| \int_{t_1}^{t_2} T[I](s) \bar{A}(s) \, ds \Big| 
  \lesssim \int_{t_1}^{t_2}  |a(s)| \lr{s}^{-3/2}|A(0)|^2 \cdot |A(s)| \, ds
  \lesssim \rho_{\Gamma_0}^{9/4}(t_1)
\end{align}
which shows that this contribution is a remainder of the type $R_a$ as in \eqref{lemODEdampRa}.

Similarly, using Lemma \ref{lem:stronger-local-decay} we can estimate
\begin{align}
& {\big\| \lr{x}^{-3} L^{-1} \Im \left( e^{itL} II(t) \right) \big\|}_{L^\infty_x}
 \lesssim \int_0^t \lr{t-s}^{-3/2} |A(s)| |\dot{A}(s)| ds \lesssim   \rho_{\Gamma_0}^{3/2}(t)
\end{align}
which implies
\begin{align}
\big|  T[II](t) \bar{A}(t) \big| \lesssim  |a(t)|\cdot |A(t)|\cdot\rho_{\Gamma_0}^{3/2}(t)=\rho_{\Gamma_0}^{5/2}(t).
\end{align}

For contribution from the last term in \eqref{Rw2split} 
we first write, according to the definition \eqref{R_E2-defn} 
and the notation \eqref{R_E2-defn'},
\begin{align}\label{RE2pr1}
T[III](s) \overline{A}(s) & = \frac{1}{i \lambda} e^{-i\lambda s} a(s) 
  \Big( \int_\R \beta \varphi^2
  L^{-1} \Im \left( e^{is L} III(s) \right) dx \Big) \cdot \overline{A}(s) 
\\
& = \frac{2}{i \lambda} e^{-i\lambda s} a(s) 
  \Big( \int_\R \beta \varphi^2
  L^{-2} \theta(x) dx \Big) |A(s)|^2 \cdot \overline{A}(s) 
\\
& \quad -\frac{1}{i \lambda} e^{-i\lambda s} a(s) 
  \Big( \int_\R \beta \varphi^2
  L^{-1} \Im \left( \frac{e^{-2i\lambda s}}{i(L+2\lambda)}  \theta(x) \overline{A}(s)^2 
  \right)dx \Big) \cdot \overline{A}(s) . 
\end{align}
Then, recalling that 
$a(s) = e^{i\lambda s}A(s) + e^{-i\lambda s}\overline{A}(s)$,
we group terms according to their oscillations:
\begin{align}
\label{TIII}
T[III](s) \overline{A}(s) & = J_0(s) + J_1(s) + J_2(s) + J_3(s),  
\\
J_0(s) & := \frac{2}{i \lambda}
  \Big( \int_\R \beta \varphi^2
  L^{-2} \theta(x) dx \Big) \cdot |A(s)|^4
  \\ & \qquad +\frac{1}{2i\lambda}
  \Big( \int_\R \beta \varphi^2
  L^{-1} \frac{1}{(L+2\lambda)}\theta(x) dx \Big) \cdot |A(s)|^4,
\\
J_1(s) & := \frac{2}{i \lambda} e^{-2i\lambda s}
  \Big( \int_\R \beta \varphi^2
  L^{-2} \theta(x) dx \Big) |A(s)|^2 \overline{A}(s)^2, 
\\
J_2(s) & := -\frac{1}{i \lambda}
  \Big( \int_\R \beta \varphi^2
  L^{-1} \Im \left( \frac{e^{-2i\lambda s}}{i(L+2\lambda)}
  \theta(x) \overline{A}(s)^2 \right)  dx \Big) \cdot |A(s)|^2,
\\
J_3(s) & := \frac{1}{2i \lambda} e^{-4i\lambda s}
  \Big( \int_\R \beta \varphi^2 \frac{1}{L(L+2\lambda)} \theta(x) dx \Big) \cdot \overline{A}(s)^{4}.
\end{align}
We notice that $J_0$ has the form
\begin{align}\label{extracorrection1}
J_0(s) & := i \beta_1 |A(s)|^4, 
  \qquad \beta_1:= -\frac{1}{2\lambda} \Big( \int_\R \beta \varphi^2 
  \Big( \frac{4}{L^2} + \frac{1 }{L(L+2\lambda)} \Big)  \theta(x) dx \Big) \in \R
\end{align}
so that that it does not contribute to the real part of $T[III] \overline{A}$;
however, $J_0$ will be relevant later when extracting asymptotics for $A$.
The term $J_2$ is purely imaginary as well, but since it has non-vanishing oscillations
we will show that it is a remainder like $J_1$ and $J_3$

To estimate the contribution from $J_1$ we notice that it has non-vanishing oscillations $e^{-2i\lambda s}$
so that we can integrate by parts in $s$ and bound it as follows:
\begin{align}\label{extracorrectionJ1}
\Big|\int_{t_1}^{t_2} J_1(s) \, ds \Big| & \lesssim \Big| 
  \overline{A}^2(s) |A(s)|^2 \Big( \int_\R \beta \varphi^2 L^{-2} \theta dx \Big) \Big|_{s=t_1}^{t_2} \Big|
  \\
  & + \Big| \int_{t_1}^{t_2} 
  e^{-2i\lambda s} \frac{d}{ds} \big[ |A(s)|^2  \overline{A}^2(s)\big] \Big( \int_\R \beta \varphi^2 L^{-2}\theta dx \Big)
  \, ds \Big|
  \lesssim \rho_{\Gamma_0}^{3/2}(t_1), 
\end{align}
having used again $\dot{A}(s) =O( \rho_{\Gamma_0}(s))$.
For $J_3$ we can use an identical integration by parts argument and see that
\begin{align}\label{extracorrectionJ3}
\Big|\int_{t_1}^{t_2} J_3(s) \, ds \Big| 
  & \lesssim  
  \Big| \int_{t_1}^{t_2} e^{- 4i\lambda s} 
  \overline{A}(s)^4 \Big( \int_\R \beta \varphi^2 \frac{1}{L(L+2\lambda)} \theta \, dx\Big)\, ds \Big|
  \lesssim  \rho_{\Gamma_0}^{3/2}(t_1),
\end{align}
and, similarly, we can obtain
\begin{align}\label{extracorrectionJ2}
\Big|\int_{t_1}^{t_2} J_2(s) \, ds \Big| \lesssim \rho_{\Gamma_0}^{3/2}(t_1).
\end{align}
This concludes the proof of \eqref{ODEdampestRE2}.
We also record that from \eqref{TIII} and \eqref{extracorrectionJ1}-\eqref{extracorrectionJ2} we have
\begin{align}\label{J0toref}
T[III](s) \overline{A}(s) & = J_0(s) + J_R(s), \qquad  J_0(s) = i \beta_1 |A(s)|^4, \qquad \int_{t_1}^{t_2} J_R(s) \, ds \lesssim \Rho^{3/2}(t_1).
\end{align}
\end{proof}

\medskip
\begin{proof}[{\bf Proof of \eqref{ODEdampestRE3}}]
To estimate the contribution from the term \eqref{R_E3-defn},
we first recall the definition of `$\rest$' from \eqref{Rest} and see that we can schematically write
\begin{align}
R_{E,3}(t) & := \frac{1}{i \lambda} e^{-i\lambda t} a(t) 
  \int_\R \beta \varphi^2 \frac{1}{L}
  \Im \Big( \int_0^t e^{i(t-s)L} N_3(s) \, ds \Big) \, dx, 
  \\
  N_3(s) & := O\big( a(s) P_c(\beta \varphi v(s))\big) 
  + O\big( P_c(\beta v^2(s))\big) + O\big( a^4(s) P_c(\varphi^4) \big),
\end{align}
where all the other nonlinear terms are omitted since they contain more factors of $a$ or $\beta v$ 
and, therefore, are lower order.
To estimate these contributions we need to distinguish short-intermediate $t\lesssim \varepsilon_0^{-2}$, 
and larger times $t\gtrsim \varepsilon_0^{-2}$.

First, in view of \eqref{a prioriSGN}, and the a priori bounds, 
since all the terms in $N_3$ are localized in space, we have 
\begin{align}\label{ODEdampestRE3pr1}
{\| \lr{x}^2 N_3(s) \|}_{H^2} \lesssim 
  \Rho^{1/2}(s)\cdot\min( \varepsilon_0^{\tilde{\alpha}} \lr{s}^{-1+\widetilde{\delta_2}}, \varepsilon_0)
  + \min( \varepsilon_0^{\widetilde{\alpha}} \lr{s}^{-1+\widetilde{\delta_2}}, \varepsilon_0)^2
  + \Rho^2(s) =: F(s)
\end{align}
and we can directly verify that 
\begin{align}\label{ODEdampestRE3pr2}
\int_0^\infty F(s) \, ds \lesssim \varepsilon_0^{1-(2 \tilde{\delta_2}-\tilde{\alpha})}
\end{align}
by looking at the cases $s \lesssim \varepsilon_0^{-2}$ and $s \gtrsim \varepsilon_0^{-2}$. 

Next, using the local energy decay estimate \eqref{local-decay-faster}, we have
\begin{align}\label{ODEdampestRE3pr3}
\Big| \int_{t_1}^{t_2} R_{E,3}(t) \bar{A}(t) \, dt \Big|  
  \lesssim \Big| \int_{t_1}^{t_2} \Rho(t) \Big( \int_0^t \frac{1}{\lr{t-s}^{3/2}} F(s) \, ds \Big) dt \Big|. 
\end{align}
For $t_1 \lesssim \varepsilon_0^{-2}$, using $\Rho \lesssim \varepsilon_0^2$ and Fubini,
we can estimate
\begin{align}
\Big| \int_{t_1}^{t_2} R_{E,3}(t) \bar{A}(t) \, dt \Big|  
  \lesssim \varepsilon_0^2 \int_0^\infty F(s) \, ds,
\end{align}
which, in view of \eqref{ODEdampestRE3pr2} is stronger than the claimed \eqref{ODEdampestRE3}.

For larger times $t_1 \gtrsim \varepsilon_0^{-2}$, using \eqref{ODEdampestRE3pr3}
and splitting the inner integral we have
\begin{align}\label{ODEdampestRE3pr3'}
\begin{split}
\Big| \int_{t_1}^{t_2} R_{E,3}(t) \bar{A}(t) \, dt \Big|  
  & \lesssim \Big| \int_{t_1}^{t_2} \Rho(t) \Big( \int_0^{\frac{1}{2}t_1} 
  \frac{1}{\lr{t-s}^{3/2}} F(s) \, ds \Big) dt \Big|
  \\
  & + \Big| \int_{t_1}^{t_2} \Rho(t) \Big( \int_{\frac{1}{2}t_1}^t 
  \frac{1}{\lr{t-s}^{3/2}} F(s) \, ds \Big) dt \Big|.
\end{split}
\end{align}
The first integral is bounded by
\begin{align}
  & C\Big| \int_{t_1}^{t_2} \Rho(t) t^{-3/2} \Big( \int_0^\infty F(s) \, ds \Big) dt \Big|
  \lesssim \varepsilon_0^{1-(2 \tilde{\delta_2}-\tilde{\alpha})} t_1^{-3/2},
\end{align}
which suffices.
For the second in integral in \eqref{ODEdampestRE3pr3'} we bound $\Rho(t) \leq \Rho(s)$ 
and exchange the order of integration to bound it by
\begin{align}
& C \Big| \int_{\frac{1}{2}t_1}^{\infty} \Rho(s) F(s) \, ds \Big| \lesssim t_1^{-5/4},
\end{align}
having used the bound \eqref{ODEdampestRE3pr1} for $\tilde{\alpha},\widetilde{\delta_2}$ small enough.
\end{proof}

\medskip
\begin{proof}[{\bf Proof of \eqref{ODEdampestRE4}}]
We now look at the contribution of the last term in \eqref{R_E4-defn} and directly estimate
using the a priori assumptions
\begin{align}
\big| R_{E,4}(t) \big| 
  & \lesssim |a(t)|^4 + {\| \lr{x}^{-2} v^4 \|}_{L^\infty} 
  \lesssim \rho_{\Gamma_0}^2(t) + \min\big( \varepsilon_0^{4\alpha} \lr{t}^{-4+4\delta_2}, \varepsilon_0^4 \big)
  \lesssim \rho_{\Gamma_0}^{2}(t)
\end{align}so that
\begin{align}
\big| R_{E,4}(t)\overline A(t) \big|
  & \lesssim |A(t)|\cdot \rho_{\Gamma_0}^{2}(t)\lesssim \rho_{\Gamma_0}^{5/2}(t),
\end{align}
This concludes the proof of all the estimates for the terms on the right-hand side of \eqref{Rsplit'}
and therefore, via Lemma \ref{lemODEdamp}, we have obtained the main Proposition \ref{prop:bootstrap-on-rho}.
\end{proof}

\smallskip
\begin{remark}\label{rmk:controlR}
From the proof of Proposition \ref{prop:bootstrap-on-rho} 
and the definition of $R_Q$ in \eqref{R_Q-defn},
we can see, in particular, that for all $t>c\rho_0^{-1/2}$,
\begin{align}\label{rmk:controlRD}
|R_D(t)|&\lesssim  \Rho^{3/2}(t),
\\
\label{rmk:controlRE}
|R_E(t)|&\lesssim \Rho^{3/2}(t), 
\\
\label{rmk:controlRQ}
|R_Q(t)|&\lesssim  \varepsilon_0^{3\alpha}\Rho^{1/2}(t)\lr{t}^{-1+8\delta}. 
\end{align}
These estimates will be useful when deriving the asymptotics 
of the renormalized amplitude $B$ later on.
Note that \eqref{rmk:controlRE} can be extracted from the proof of \eqref{ODEdampestRE3} as follows:
with the same notation for $F$ in \eqref{ODEdampestRE3pr1} we have that 
\begin{align}
|R_{E,3}(t)| \lesssim \Rho^{1/2}(t) 
  \int_0^\infty \frac{1}{\lr{t-s}^{3/2}} F(s) \, ds \lesssim \Rho^{3/2}(t),
\end{align}
having used the definition of $F$ in \eqref{ODEdampestRE3pr1}.
\end{remark}


\smallskip
We conclude this section by introducing a new variable 
that varies more slowly than $A$ over time;  this will be helpful to establish
some estimates later on in the proof, and also to derive precise asymptotics and lower bounds.

\begin{lemma}
\label{lem:decay-B'}
Assume the hypotheses of Proposition \ref{prop:bootstrap-on-rho} and define 
\begin{equation}\label{B}
B := A + \frac{c}{\lambda^2} \left( \frac{e^{i\lambda t}}{2} A^2 - \frac{e^{-3i\lambda t}}{6} \overline{A}^2 - e^{-i\lambda t} |A|^2 \right), 
\end{equation}
where  $c: = \int \beta \varphi^3 dx$. Then, 
$B \to A$ as $|A| \to 0$ and for all $t> 0$
\begin{equation}\label{Bdot}
|\dot{B}(t)| \lesssim \Rho(t)^{3/2} 
  + \varepsilon_0^{3\al}\Rho^{1/2}(t)\lr{t}^{-1+8\delta}.
\end{equation}
\end{lemma}

\begin{proof}
    Integrate \eqref{evolution-A} over the interval $[0,t]$ and integrate by parts in
$\int_0^tP(s)ds$ (using the nonvanishing phase) to find 
\begin{align*}
 A(t)=A(0)-\frac{c}{\lambda^2}\Big(\frac{e^{i\lambda s}}{2}A^2(s)-e^{-i\lambda s}|A(s)|^2-\frac{e^{-3i\lambda s}}{6}\overline{A}(s)^2\Big)\Big|_{s=0}^t   +\int_0^tR_B(s)ds
\end{align*}
where
\begin{equation*}
R_B(s) := 
  \frac{c}{\lambda^2}\bigg(-e^{-i\lambda s}\frac{d}{ds}|A|^2(s)
  -\frac{e^{-3i\lambda s}}{6}\frac{d}{ds}\overline{A}^2(s)+\frac{e^{i\lambda s}}{2}\frac{d}{ds}A^2(s)\bigg)+\frac{1}{2i\lambda}(\Lambda-i\Gamma)|A|^2A+R(s).
\end{equation*}Therefore,  
\begin{equation}\label{B'}
\dot{B}(t)=R_B(t)\approx O(|\dot{A}(t)||A(t)|)+O(|A|^3) + R(t).  
\end{equation}
The desired bound follows from $A=O(\Rho^{1/2})$, $\dot A=O(\Rho)$ and Remark \eqref{rmk:controlR}.
\end{proof}


\bigskip
\section{Bootstrap for the continuous component I: setup and basic weighted bounds}
\label{sec:PDE-part0}
In this section we begin our nonlinear analysis for the continuous component $v = P_c u.$
We first decompose the profile $f=e^{-itL}w$ 
into a `good' component ($h$) and a `bad' component ($g$) ,
and organize the nonlinear terms in Duhamel's formula for $h$.
We state the main bootstrap result in Proposition \ref{prop:bootstrap-on-v}
and then give some estimate for the `bad' component. 
We finally state the main bootstrap on the weighted-type norm of the good component in Proposition \ref{prop:good-part},
and bound some of the relatively easier nonlinear terms, while postponing the more difficult estimates
to Section \ref{sec:PDE-part2}.

\smallskip
\subsection{Good and bad decomposition and the main bootstrap}
Recall the Duhamel formula for the radiation profile (see \eqref{f-evolution}):
\begin{equation*}
\partial_t \widetilde f = e^{-it\lr{k}} \distF P_c \left[ \beta (a(t) \varphi + v(t))^2 + (a(t) \varphi + v(t))^4 \right]. 
\end{equation*}
This implies that 
\begin{equation}\label{duhamel-tilde-f-1}
\begin{aligned}
\widetilde f(t) - \widetilde f_0 & =
\int_0^t a^2(s) e^{-is\lr{k}} \widetilde{P_c (\beta \varphi^2)}(k) ds 
\\
& \qquad + 2 \int_0^t a(s) e^{-is\lr{k}} \widetilde{P_c (\beta \varphi v)}(k) ds + \int_0^t e^{-is\lr{k}} \widetilde{P_c (\beta v^2)}(k) ds 
\\
& \qquad + \int_0^t a^4(s) e^{-is\lr{k}} \widetilde{P_c (\varphi^4)}(k) ds + 4 \int_0^t a^3(s) e^{-is\lr{k}} \widetilde{P_c (\varphi^3 v)}(k) ds 
\\
& \qquad + 6 \int_0^t a^2(s) e^{-is\lr{k}} \widetilde{P_c (\varphi^2 v^2)}(k) ds + 4 \int_0^t a(s) e^{-is\lr{k}} \widetilde{P_c (\varphi v^3)}(k) ds 
\\
& \qquad + \int_0^t e^{-is\lr{k}} \widetilde{P_c (v^4)}(k) ds.
\end{aligned}
\end{equation}
Note that, using \eqref{a-in-A-defn}, we can rewrite the first line of the right-hand side of \eqref{duhamel-tilde-f-1} as
\begin{equation}
\label{duhamel-tilde-f-line1}
\begin{aligned}
& \int_0^t a^2(s) e^{-is\lr{k}} \widetilde{P_c (\beta \varphi^2)}(k) ds 
\\
& = \int_0^t e^{-is(\lr{k}-2\lambda)} B^2(s) \widetilde{P_c (\beta \varphi^2)}(k) ds + \int_0^t e^{-is(\lr{k}-2\lambda)} (A^2(s) - B^2(s)) \widetilde{P_c (\beta \varphi^2)}(k) ds
\\
& + 2 \int_0^t e^{-is \lr{k}} |A|^2(s) \widetilde{P_c (\beta \varphi^2)}(k) ds + \int_0^t e^{-is(\lr{k}+2\lambda)} \overline{A}^2(s) \widetilde{P_c (\beta \varphi^2)}(k) ds. 
\end{aligned}
\end{equation}

Since $\lr{k} - 2\lambda$ vanishes at frequency $|k| = k_0 = \sqrt{4\lambda^2-1}$, 
we localize the corresponding resonant term in $\wt{f}$ around $k_0$. 
This leads to the introduction of a `bad' variable $g$ and a `good' variable $h$ as follows:
\begin{equation}\label{bad-g-defn}
\widetilde{g} (t, k) := - \phi_{\leq -C} (\lr{k} - 2\lambda) 
  \int_0^t e^{-is(\lr{k}-2\lambda)} B^2(s) \widetilde{P_c (\beta \varphi^2)}(k) ds,
  \qquad \widetilde{h} := \widetilde{f} + \widetilde{g}. 
\end{equation}
where, recall, $B$ is defined in \eqref{B}.
By rearranging the terms we have the following:

\begin{lemma}\label{lem:h}
Let $\widetilde{h}$ be defined as in \eqref{bad-g-defn}, then
\begin{equation}
\label{tilde-h-decomposition}
h(t) = f_0 + S(t) + M (t) + Q(t) + R_{v,a}(t),
\end{equation}
where $S$ denotes the source terms from \eqref{duhamel-tilde-f-line1} which we write as
\begin{equation}
\label{term-S}
\begin{aligned}
&S = S_1 + S_2 + S_3, 
\\ 
&\widetilde S_1 (t) := \left(1 - \phi_{\leq -C} (\lr{k} - 2\lambda) \right) \int_0^t e^{-is(\lr{k}-2\lambda)} B^2(s) \widetilde{P_c (\beta \varphi^2)}(k) ds, 
\\
&\widetilde S_2 (t):= \int_0^t e^{-is(\lr{k}-2\lambda)} (A^2(s) - B^2(s)) \widetilde{P_c (\beta \varphi^2)}(k) ds, 
\\
&\widetilde S_3 (t):= 2 \int_0^t e^{-is \lr{k}} |A|^2(s) \widetilde{P_c (\beta \varphi^2)}(k) ds 
  + \int_0^t e^{-is(\lr{k}+2\lambda)} \overline{A}^2(s) \widetilde{P_c (\beta \varphi^2)}(k) ds,
\end{aligned}
\end{equation}
$M$ denotes the mixed term defined by
\begin{equation}
\label{term-M}
\begin{aligned}
\widetilde M (t):= 2 \int_0^t a(s) e^{-is\lr{k}} \widetilde{P_c (\beta \varphi v)}(k) ds,
\end{aligned}
\end{equation}
$Q$ denotes the quartic term defined by
\begin{equation}
\label{term-Q}
\widetilde Q (t):= \int_0^t e^{-is\lr{k}} \widetilde{P_c (v^4)}(k) ds,
\end{equation}
and $R_{v,a}$ denotes the remaining terms given by
\begin{equation}
\label{term-RM}
\begin{aligned}
\widetilde{R_{v,a}}(t) := &\int_0^t e^{-is\lr{k}} \widetilde{P_c (\beta v^2)}(k) ds + \int_0^t a^4(s) e^{-is\lr{k}} \widetilde{P_c (\varphi^4)}(k) ds + 4 \int_0^t a^3(s) e^{-is\lr{k}} \widetilde{P_c (\varphi^3 v)}(k) ds \\
& \qquad + 6 \int_0^t a^2(s) e^{-is\lr{k}} \widetilde{P_c (\varphi^2 v^2)}(k) ds + 4 \int_0^t a(s) e^{-is\lr{k}} \widetilde{P_c (\varphi v^3)}(k) ds. 
\end{aligned}
\end{equation}
\end{lemma}


Based on the above Duhamel's formula for $h$ we can proceed to prove our main bootstrap
proposition where we control the weighted norm of the good part with a slow growth in time,
and prove almost sharp decay and local decay for $v$, as well as control its high Sobolev norm.

\begin{proposition}\label{prop:bootstrap-on-v}
Let $\alpha,\delta,\delta_1,\delta_2$ be as in \eqref{choice-of-csts}.  
Assume that the following a priori assumptions are satisfied for $T>0$:  
\begin{equation}
\label{rho-assumption-bootstrap-pde}
\sup\limits_{t \in [0, T]} \left| \left(1 + \rho_0 \frac{\Gamma}{\lambda} t \right) \frac{\rho(t)}{\rho_0} - 1 \right| < \frac{1}{2},
\end{equation}
and 
\begin{equation}
\label{w-assumptions-bootstrap-pde}
\begin{aligned}
&\sup\limits_{t \in [0, T]} \lr{t}^{-\delta_1} {\big\|\partial_k \wt{h}(t)\big\|}_{L^2_k} 
  \leq 2\varepsilon_0^{\al},
\\
& \sup\limits_{t \in [0, T]} 
  \lr{t}^{1/2-\delta} {\big\|L^{-1} w(t)\big\|}_{L_x^\infty} 
  \leq 2\varepsilon_0^{\al}, 
  \\
&\sup\limits_{t \in [0, T]} {\big\|L^{-1} w(t) \big\|}_{H_x^N} \leq 2C \varepsilon_0, 
\\
& \sup\limits_{t \in [0, T]} \lr{t}^{1-\delta_2} \left( 
  {\big\|\lr{x}^{-2} \big(L^{-1} e^{\pm itL}f^{(\pm)}(t)\big)_S \big\|}_{L_x^\infty}
  + {\big\|\lr{x}^{-2} \big(L^{-1} e^{\pm itL}f^{(\pm)}(t)\big)_R\big\|}_{L_x^\infty}\right) \leq 2 \varepsilon_0^{\al},
\end{aligned}
\end{equation}
for some sufficiently large $C>1$, where $f^+=f, f^-=\overline{f}$.
Then, we have the improved bounds
\begin{equation}
\label{w-conclusion-bootstrap}
\begin{aligned}
&\sup\limits_{t \in [0, T]} \lr{t}^{-\delta_1} {\big\|\partial_k \wt{h}(t)\big\|}_{L^2_k}  \leq  \varepsilon_0^{\al} ,
\\
& \sup\limits_{t \in [0, T]} \lr{t}^{1/2-\delta} 
  {\big\|L^{-1} w(t)\big\|}_{L_x^\infty}
  \leq \varepsilon_0^{\al}, 
\\
&\sup\limits_{t \in [0, T]} {\big\|L^{-1} w(t)\big\|}_{H_x^N} \leq C\varepsilon_0,
\\
& \sup\limits_{t \in [0, T]} \lr{t}^{1-\delta_2} \left(
  {\big\|\lr{x}^{-2} \big(L^{-1} e^{ \pm i tL}f^{(\pm)}(t)\big)_S\big\|}_{L_x^\infty}
  + {\big\|\lr{x}^{-2} \big(L^{-1} e^{\pm itL}f^{(\pm)}(t)\big)_R\big\|}_{L_x^\infty}\right) \leq \varepsilon_0^{\al}. 
\end{aligned}
\end{equation}
\end{proposition}

\medskip
\subsection{Basic bounds for the good and bad parts}
We first record some estimates on the bad part.

\begin{lemma}\label{lem:bad-part-bds} 
Let $\widetilde{g}$ be defined in \eqref{bad-g-defn}. Then, under the assumption 
of Proposition \ref{prop:bootstrap-on-v} we have for all $t>0$, the estimates 
\begin{align}
\label{g-in-H1}
& {\| \widetilde{g}(t,\cdot) \|}_{L^2_k}\lesssim \varepsilon_0,  
\\
\label{nablag-in-H1} 
& {\|\partial_k \widetilde{g}(t,\cdot)\|}_{L^2_k}\lesssim 
\min \big( \varepsilon_0^2 \lr{t}^{3/2} , \lr{t}^{1/2}\big).
\end{align}

\end{lemma}

\medskip
\begin{proof}[Proof of Lemma \ref{lem:bad-part-bds}]
Let us start with the proof of \eqref{g-in-H1}. Since $|k| \approx 1$
on the support of $\widetilde{g}$ (constant here are understood to depend on the fixed $\lambda$),
using the smoothing-type estimate in Lemma \ref{lem:smoothing-effect}
(see Remark \ref{rem:smoothing-effect} as well), we find that
\begin{equation*}
\begin{aligned}
\|\wt{g}(t,\cdot)\|_{L^2_k} & = \left\|\int_0^t \phi_{\leq -C}(\lr{k}-2\lambda) 
  B^2(s) e^{-is(\lr{k}-2\lambda)} \widetilde{\beta \varphi^2}(k) ds \right\|_{L_k^2} 
  \lesssim \|B^2(s) \lr{x}\beta\varphi^2\|_{L_x^1 L_s^2[0,t]} \lesssim \varepsilon_0.
\end{aligned}
\end{equation*}
For the weighted estimate \eqref{nablag-in-H1}, we write 
\begin{align*}
\|\partial_k \widetilde{g}(t,\cdot)\|_{L^2_k}\lesssim \| I(t,\cdot)\|_{L_k^2}
  +\| II(t,\cdot)\|_{L_k^2} 
  + \| III(t,\cdot)\|_{L_k^2}   
\end{align*}
where 
\begin{equation}\label{partial_kg-tilde}
\begin{aligned}
I(t,k) & := \frac{k}{\langle k\rangle}\phi_{\leq -C}(\langle k\rangle-2\lambda) \widetilde{\beta \varphi^2}(k)  \int_0^tsB^2(s)e^{-is(\langle k\rangle-2\lambda)} ds,
\\
II(t,k) & := \phi_{\leq -C}(\langle k\rangle-2\lambda)  \partial_k \widetilde{\beta \varphi^2}(k)\int_0^tB^2(s)e^{-is(\langle k\rangle-2\lambda)} ds, 
\\
III(t,k) & := \partial_k(\phi_{\leq -C}(\langle k\rangle-2\lambda)) \widetilde{\beta \varphi^2}(k) \int_0^tB^2(s)e^{-is(\langle k\rangle-2\lambda)} ds.
\end{aligned}
\end{equation}
The estimate for $I$ is again a direct consequence of Lemma \ref{lem:smoothing-effect};
in fact,
\begin{equation}
\begin{aligned}
\| I\|_{L^2_k} & = \left\|\int_0^t s B^2(s) \frac{k}{\lr{k}} \phi_{\leq -C}(\lr{k}-2\lambda) 
  e^{-is(\lr{k}-2\lambda)} \widetilde{\beta \varphi^2}(k) ds \right\|_{L_k^2} 
  \\
& \lesssim \|s B^2(s) \varphi^2(x)\|_{L_x^1 L_s^2 [0,t]} 
  \lesssim \|s B^2(s)\|_{L_s^2 [0,t]} 
  \lesssim 
  \min \big( \varepsilon_0^2t^{3/2} , t^{1/2}\big).
\end{aligned}
\end{equation}
The claimed estimates 
for the terms $II$ and $III$ are easier to obtained.
\end{proof}

Next, we record an additional decay estimate to be used in the course
of the nonlinear analysis in Section \ref{sec:Quartic}.

\begin{lemma}\label{lem:BoundednessW}
Recall the notation for the coefficients $A_{\star,\dagger}$ in \eqref{A-functions}.
Define 
\begin{equation}\label{lemdecaytousew} 
W(t,x) := \langle D\rangle^{-1}
e^{it\langle D\rangle}
  \widehat{\mathcal{F}}^{-1}\big( A_{\ast,\dagger}\widetilde{f}(t,\cdot) \big)(x),
  \qquad \ast,\dagger \in \{+,-\},
\end{equation}
where $f$ is the profile of $w$.
Then, under the a priori assumptions \eqref{w-assumptions-bootstrap-pde}, we have  for all $t>0$,  
\begin{align}
   \|W(t,\cdot)\|_{L^\infty}\lesssim \varepsilon_0^\al \lr{t}^{-1/2+2\delta}.
\end{align}
\end{lemma}

\begin{proof}
We first apply Sobolev's embedding
$W^{s,p}(\mathbb R)\hookrightarrow L^\infty(\mathbb R)$ with $s= 2/p$, $p=\infty-$,
and write
\begin{align}\label{WLp-Linfty}
    \|W(t,\cdot)\|_{L^{\infty}_x}&\lesssim \|\lr{D}^{-1+s}e^{it\lr{D}}\widehat{\mathcal{F}}^{-1}(A_{\ast,\dagger}\wt{f})\|_{L^p_x}   \lesssim \|\lr{D}^{-1+s}e^{it\lr{D}}A_{\ast,\dagger}(D)\mathcal{W}^\ast{f}\|_{L^p_x}
\end{align}
where we recall that the wave operator $\mathcal{W}$ and its adjoint $\mathcal{W}^\ast$
are defined in Lemma \ref{lem:waveoper}.
From the definitions in \eqref{A-functions},
we see that $A_{\ast,\dagger}$ is a bounded operator on $L^p$, $p\in(1,\infty)$, since 
it is a linear combination of half-line Fourier projections with
Mikhlin multipliers $T(D)$ and $R_\pm(D)$.
Using also the boundedness of the wave operator from Lemma \ref{lem:waveoper},  
and the a priori bounds \eqref{w-assumptions-bootstrap-pde}, it follows from \eqref{WLp-Linfty}
 that 
 \begin{align*}
   \|W(t,\cdot)\|_{L^{\infty}_x} &\lesssim \|\mathcal{W}^\ast(L^{-1+s}e^{it L}f)\|_{L^p_x}
   \\
   &\lesssim \|{L}^{-1}e^{itL}f\|_{W^{s,p}}
   \\
   &\lesssim \|{L}^{-1}e^{itL}f\|^{\theta}_{H^1}\|{L}^{-1}e^{itL}f\|^{1-\theta}_{L^\infty}
   \\
   &\lesssim \varepsilon_0^{\theta}\cdot (\varepsilon_0^\al\lr{t}^{-1/2+\delta})^{1-\theta}=\varepsilon_0^{\theta+\al(1-\theta)}
   \lr{t}^{-1/2+2\delta},
\end{align*}
where the last inequality follows from the Sobolev-Gagliardo-Nirenberg inequality with 
$\theta = 2/p$ and choosing $p$ large enough.
\end{proof}

\medskip
We now turn to the good part. The key bound in our whole bootstrap argument is given by the following
weighted bound for $h$:

\begin{proposition}\label{prop:good-part}  
Recall the definitions from Lemma \ref{lem:h}, and assume the a priori bound 
\eqref{w-assumptions-bootstrap-pde} in Proposition \ref{prop:bootstrap-on-v}.
Then for all $t>0$ we have
\begin{align}
\label{S-tildeL2}
{\|\partial_k \widetilde S(t,\cdot) \|}_{L_k^2} \lesssim \varepsilon_0^{2\al}\lr{t}^{\delta_1},
\\
\label{M-tildeL2}
{\|\partial_k \wt{M}(t,\cdot)\|}_{L_k^2} \lesssim \varepsilon_0^{2\al}\lr{t}^{\delta_1},
\\
\label{Q-tildeL20}
{\|\partial_k \wt{Q}(t,\cdot)\|}_{L_k^2} \lesssim \varepsilon_0^{3\al}\lr{t}^{8\delta},
\\
\label{partial-k-R}
{\|\partial_k \wt{R_{v,a}}(t,\cdot) \|}_{L_k^2} \lesssim \varepsilon_0^{2\al} \lr{t}^{\delta_1}. 
\end{align}
In particular, according to \eqref{tilde-h-decomposition} and \eqref{choice-of-csts}, 
and also in view of the initial data assumption \eqref{mainthmdata}, we have
\begin{equation}\label{partial-h-tilde}
\|\partial_k \widetilde h(t,\cdot)\|_{L_k^2} \lesssim \varepsilon_0^{2\al} \lr{t}^{\delta_1}
\end{equation}
which gives the first estimate in \eqref{w-conclusion-bootstrap}.
\end{proposition}

We postpone the estimate \eqref{Q-tildeL20} for the quartic term $\widetilde{Q}$
until Section \ref{sec:PDE-part2} as the proof is more technically involved 
and relies on multilinear analysis which is discussed later.
In the rest of this section we estimate the other relatively simpler contributions 
on the right-hand side of \eqref{tilde-h-decomposition}.

\medskip
\textbf{Proof of \eqref{S-tildeL2}.} 
We recall the splitting in \eqref{term-S}
and start by estimating the contribution from $S_1$.
Set $\chi_C (k): = 1- \psi_{\leq -C} (\lr{k}-2\lambda)$. 
The expression for $\partial_k S_1$ contains several terms, 
but we concentrate on the most difficult one, 
when $\partial_k$ hits the phase $e^{-is(\lr{k} - 2\lambda)}$ producing
an additional $s$ in the integrand. Then, using integration by parts in $s$ we get
\begin{align}\label{S1-term-estimate}
\nonumber &\left\| \chi_C(k) \int_0^t s \frac{k}{\lr{k}} e^{-is(\lr{k} - 2\lambda)}  B^2(s) \distF P_c (\beta \varphi^2) (k) ds \right\|_{L_k^2} 
\\
\nonumber & \qquad \qquad \lesssim \left\| t B^2(t) \frac{k}{\lr{k} (\lr{k}-2\lambda)} e^{-it(\lr{k} - 2\lambda)} \chi_C (k) \distF P_c (\beta \varphi^2) (k) \right\|_{L_k^2} 
\\
\nonumber & \qquad \qquad \qquad  + \left\| \frac{\chi_C(k)}{\lr{k} - 2\lambda} \int_0^t \frac{k}{\lr{k}} e^{-is(\lr{k} - 2\lambda)} \left[ 2s \dot{B}(s) B(s) + B^2(s) \right] \distF P_c (\beta \varphi^2) (k) ds \right\|_{L_k^2} 
\\
& \qquad \qquad 
\lesssim t \Rho(t) + \int_0^t \big(\varepsilon_0^{3\al}s\Rho(s) \lr{s}^{-1+8\delta} 
  + s \Rho^2(s) \big) ds
\\
& \qquad\qquad  \lesssim t\Rho(t)
  + \varepsilon_0^{3\al} \lr{t}^{8\delta} 
  + \log (1 + \rho_0 t)
\end{align}
where we used the decay properties of $B$ and $\dot{B}$ from Lemma \ref{lem:decay-B'}. 
The estimate \eqref{S1-term-estimate} is acceptable for the bound \eqref{S-tildeL2}
since, in view of our choices of parameters \eqref{choice-of-csts}, we have 
$\alpha, \delta \ll \delta_1$
and therefore, for all $t \geq 0$,
\begin{equation}\label{convert}
\Rho(t)t \leq \varepsilon_0^{2\al}\lr{t}^{\delta_1}, 
\qquad \varepsilon_0^{3\al} \lr{t}^{8\delta} \ll \varepsilon_0^{2\al} \lr{t}^{\delta_1},
\qquad \log(1+\rho_0t) \ll \varepsilon_0^{2\al} \lr{t}^{\delta_1}.
\end{equation}
For the term $S_2$, we first rewrite $A^2 - B^2$ using the definition of $B$ from Lemma \ref{lem:decay-B'} as 
\begin{equation*}
A^2 - B^2 = -c I (2A + c I), \qquad
I = \frac{e^{i\lambda t}}{2\lambda^2} A^2 - \frac{e^{-3i\lambda t}}{6\lambda^2} \overline{A}^2 - \frac{e^{-i\lambda t}}{\lambda^2} |A|^2.
\end{equation*}
Direct multiplication leads to many terms in the expression for $A^2-B^2$, and so in $\partial_k S_2$. Most of these terms either have fast decay (like those associated with quartic powers of $A$ and $\overline{A}$) or non-vanishing phase, which can be handled integrating by parts as in \eqref{S1-term-estimate}. 
The only problematic term in $\partial_k S_2$ is the resonant one associated with $A^3$,
that is,
\begin{equation*}
\int_0^t s \frac{k}{\lr{k}} e^{-is(\lr{k}-3\lambda)} A^3(s) \distF P_c (\beta \varphi^2) (k) ds.
\end{equation*}
To deal with the $L_k^2$ norm of this integral, 
we use the smoothing-type estimate from Lemma \ref{lem:smoothing-effect} to get
\begin{equation*}
\begin{aligned}
\left\| \int_0^t s \frac{k}{\lr{k}} e^{-is(\lr{k}-3\lambda)} A^3(s) \distF P_c (\beta\varphi^2) (k) ds \right\|_{L_k^2} & \lesssim \left\| s A^3(s) \varphi^2 \right\|_{L_x^1 L_s^2[0, t]} 
\\
& \lesssim \left\| s \Rho^{3/2}(s) \right\|_{L_s^2[0, t]} \lesssim \log (1+\rho_0 t),
\end{aligned}
\end{equation*}
where the last inequality can be verified by looking separately at the small and large times regimes.
The remaining term $S_3$ can be handled in a similar manner after integrating by parts
using the non-vanishing of the phases.

\medskip
\textbf{Proof of \eqref{M-tildeL2}.} 
Recall the definition of $M$ in \eqref{term-M}.
The main contribution to $\partial_k M$ arises again when the derivative hits the exponential term;
using once more the smoothing-type estimate in Lemma \ref{lem:smoothing-effect},
followed by H\"older's inequality, and using the a priori local decay assumption on $v=\Im(L^{-1}w)$ 
from \eqref{w-assumptions-bootstrap-pde}, we have
\begin{equation}
\label{dkM-difficult-term}
\begin{aligned}
{\left\| \int_0^t a(s) s \frac{k}{\lr{k}} e^{-is\lr{k}} \widetilde{P_c (\beta \varphi v)}(k) ds \right\|}_{L_k^2} 
  & \lesssim {\big\|s a(s) \varphi \, v \big\|}_{L_x^1 L_s^2[0,t]} 
  \lesssim {\big\|s a(s) \lr{x}^{-10} v \big\|}_{L_s^2[0,t] L_x^\infty} 
\\
& \lesssim \varepsilon_0^{\al} {\big\| \Rho^{1/2}(s) \lr{s}^{\delta_2} \big\|}_{L_s^2[0,t]} 
  \lesssim \varepsilon_0^{2\al}\lr{t}^{\delta_1},
\end{aligned}
\end{equation}
where the last inequality can be verified by splitting 
the intergral at $s \approx \rho^{-1}_0$ as usual, 
and using that $\al, \delta_2 \ll \delta_1$ by \eqref{choice-of-csts}.

\medskip
\textbf{Proof of \eqref{partial-k-R}.} 
Recall the definition of $R_{a,v}$ from \eqref{term-RM}.
As before, we only look at the main terms where the $\partial_k$-derivative hits the exponential;
the corresponding contribution from the first term in \eqref{term-RM} is handled using the smoothing-type estimate
of Lemma \ref{lem:smoothing-effect}:
\begin{equation*}
\begin{aligned}
\left\|\int_0^t s \frac{k}{\lr{k}} e^{-is\lr{k}} \distF P_c (\beta v^2) (k) ds \right\|_{L_k^2}
& \lesssim \|s \lr{x} \beta v^2\|_{L_x^1 L_s^2[0,t]} 
\lesssim \|s \lr{x}^{-10} v^2\|_{L_s^2[0,t] L_x^\infty} \\
& \lesssim \varepsilon_0^{2 \al} \|s \lr{s}^{-2+2\delta_2}\|_{L_s^2[0,t]} \lesssim \varepsilon_0^{2\al},
\end{aligned}
\end{equation*}
having used the improved local decay assumption in \eqref{w-assumptions-bootstrap-pde}.
For the second term on the right-hand side of  \eqref{term-RM}, the main contribution is bounded as follows:
\begin{equation*}
\begin{aligned}
\left\|\int_0^t a^4(s) s \frac{k}{\lr{k}} e^{-is\lr{k}} \distF P_c (\varphi^4) (k) ds \right\|_{L_k^2} & \lesssim \int_0^t s \Rho^2(s) ds \lesssim \log (1 + \rho_0 t).
\end{aligned}
\end{equation*}
where we used Minkowski inequality and the exponential decay of $\varphi$. 
The two estimates above are consistent with the bound \eqref{partial-k-R} due to \eqref{convert}.
The rest of the terms in \eqref{term-RM} are easier to deal with 
and can be estimated exactly as the mixed terms in \eqref{dkM-difficult-term}.  
This concludes the proof of the proposition. $\hfill \Box$


\medskip
\section{Bootstrap for the continuous component II: dispersive and Sobolev estimates}\label{sec:PDE-part1}
In this section we begin the proof of the pointwise decay, local decay, and Sobolev a priori bounds from
the main Proposition \ref{prop:bootstrap-on-v}.
We will bound all the nonlinear terms with the exception of the 
pure radiation term $v^4$, whose treatment is postponed to Section \ref{sec:PDE-part2}. 
For ease of reference we recall that, under the hypotheses of Proposition \ref{prop:bootstrap-on-v}, 
we need to establish the following improved bounds for the continuous component $P_cu = \Im(L^{-1}w)$,
where the profile of $w$ is $f=e^{-itL} w$:
\begin{align}\label{linear-decay-w}
& \|L^{-1} w(t)\|_{L_x^\infty} \lesssim  \varepsilon_0^{2\al}\lr{t}^{-1/2+\delta},
\\
\label{local-decay-w}
& \left\|\lr{x}^{-2}\big(L^{-1}e^{\pm i tL}f^{(\pm)}(t)\big)_{E}\right\|_{L^{\infty}_x}
 \lesssim \varepsilon_0^{2\al}\lr{t}^{-1+8\delta},\qquad E\in\{S,R\}
  \\
\label{Sobolev}
&\|L^{-1}w\|_{H^N_x}\lesssim \varepsilon_0,
\end{align}
for all times $t\in  [0,T]$.
Here we use the notation $f^{(+)}=f$ and  $f^{(-)}=\overline{f}.$ 

From the  Duhamel formula for $w$ in \eqref{duhamel-for-w} we write
\begin{align}\label{L-1w}
L^{-1} w(t) & = L^{-1} e^{itL} w_0 + I_{a} + I_v+I_{a,v}+I_Q,
\end{align}
where $I_a$, $I_v$  denote the forcing generated by the discrete component and the continuous component respectively while $I_{a,v}$  represents the mixed terms:
\begin{align}
\label{Ia}
I_a(t) & := \int_0^t L^{-1} e^{i(t-s)L} \Big[a^2(s)P_c \big(\beta  \varphi^2)+a^4(s)P_c(\varphi^4)\Big]ds, 
\\
\label{Iv}
I_v(t) & := \int_0^t L^{-1} e^{i(t-s)L} P_c( \beta v^2(s) ) ds,
\\\label{Iav}
I_{a,v}(t)&:=I_{a,v}^{(0)}(t)+\sum_{\ell=1}^{3}I_{a,v}^{(\ell)}(t),
\end{align}
where  
\begin{align*}
I^{(0)}_{a,v}(t)&= \int_0^t L^{-1} e^{i(t-s)L}2a(s)P_c (\beta  \varphi \, v(s))ds
\\
I^{(\ell)}_{a,v}(t)&=\int_0^t  a(s) ^{4-\ell}L^{-1} e^{i(t-s)L} 
  P_c(\varphi_\ell \, v^\ell(s))ds,\quad \ell=1,2,3  , 
\end{align*}
with $\varphi_\ell = \big(\begin{array}{c}
                          \!4 \!\\ \!\ell\!
                          \end{array}
\big)\varphi^{4-\ell}$ a Schwartz function. 
The last term $I_Q$ represents the integral involving the quartic term in $v$, namely,
\begin{equation}\label{I_Q}
I_Q(t) = \int_0^t L^{-1} e^{i(t-s)L} P_c v^4(s) ds = L^{-1} e^{itL}Q(t),    
\end{equation}
where the last identity comes from the definition of $Q$ in \eqref{term-Q}.
We will systematically use the lemmas from Section \ref{sec:linear-decay} to bound all terms in \eqref{L-1w}
except for $I_Q$ which requires a more involved analysis 
and is treated in the next section. 
We will also often use, without explicit mention, that
\begin{align}
|a(t)| \approx \Rho^{1/2} \lesssim \varepsilon_0^{2\gamma} \lr{s}^{-1/2+\gamma}, 
  \qquad \gamma \in (0,1/2 ).
\end{align}

More precisely, in this section we prove the following propositions:

\begin{proposition}\label{prop:decay_boostrap}
Under the assumptions \eqref{rho-assumption-bootstrap-pde}-\eqref{w-assumptions-bootstrap-pde} 
we have, for all $t \in [0,T]$,
\begin{equation}
\label{I12-w-linear-decay} 
{\big\| L^{-1} w(t) - I_Q(t) \big\|}_{L_x^\infty} 
  \lesssim \varepsilon_0^{2\al} \lr{t}^{-1/2+\delta}.
\end{equation}
\end{proposition}

Next,  recall the notation in \eqref{generalized-eigenf-decomp}-\eqref{psi-regular-part}.

\begin{proposition}\label{lem:local-decay-bds}
Under the assumptions \eqref{rho-assumption-bootstrap-pde}-\eqref{w-assumptions-bootstrap-pde} 
we have, for all $t \in [0,T]$,
\begin{equation}
\label{T-f-local-decay}
{\left\| \lr{x}^{-2} \big(L^{-1}w(t) - I_Q(t) \big)_{E} \right\|}_{L_x^\infty} 
  \lesssim \varepsilon_0^{2\alpha} \lr{t}^{-1+6\delta}, \qquad E \in \{S,R\}.
\end{equation} \end{proposition}

\begin{proposition}\label{prop:Sobolev}
Assume \eqref{rho-assumption-bootstrap-pde}-\eqref{w-assumptions-bootstrap-pde}. 
Then, for all $t \in [0,T]$, we have the Sobolev estimate
\begin{equation}\label{H^N}
{\| L^{-1}w(t) \|}_{H^N_x} \leq C\varepsilon_0.
\end{equation}
\end{proposition}

\smallskip
We now proceed to prove the above propositions.

\smallskip
\begin{proof}[{\bf Proof of Proposition \ref{prop:decay_boostrap}}]
For the linear evolution, using \eqref{dispersive-estimates} 
and the bounds \eqref{derivative-in-fourier-bound}-\eqref{k-s-h-tilde-bound}, we have 
\begin{equation*}
\|L^{-1} e^{itL} w_0\|_{L_x^\infty} \lesssim \lr{t}^{-1/2} \left( \|\lr{x} w_0\|_{L_x^2} + \| w_0\|_{H_x^2}\right)\lesssim \varepsilon_0\lr{t}^{-1/2},
\end{equation*}
which is sufficient in view of our assumptions on the initial data \eqref{mainthmdata}.
Next, we  show that all the terms in \eqref{Ia}-\eqref{Iav}  have acceptable bounds.

\smallskip
Using \eqref{dispersive-estimates}, we  estimate
\begin{align}\label{Iatoref}
\begin{split}
\left\|I_{a}(t)\right\|_{L_x^\infty} & \lesssim \int_0^t\Big( a^2(s) \| L^{-1} e^{i(t-s)L} P_c(\beta \varphi^2) \|_{L_x^\infty} +a^4(s) \| L^{-1} e^{i(t-s)L} P_c( \varphi^4) \|_{L_x^\infty}\Big)ds \\
& \lesssim \int_0^t \lr{t-s}^{-1/2}(\Rho(s) +\Rho^2(s) ) ds \lesssim\lr{t}^{-1/2} \log (1+\rho_0t)\lesssim \rho_0^{\al}\lr{t}^{-1/2+\delta}
\end{split}
\end{align}
which is acceptable since $\al\leq \delta$.

Next, for $\theta\in [0,1]$ and $\alpha\in (0,1)$, 
using Gagliardo-Nirenberg-Sobolev interpolation
and the a priori assumption in \eqref{w-assumptions-bootstrap-pde}
we have, for any Schwartz function $\phi$ and $2\leq m \leq N$,
\begin{align}\label{gagliardo-local-1}
\begin{split}
\|\phi v(\tau)\|_{H_x^2} \lesssim \|\phi v(\tau) \|^{1-\theta}_{L_x^2} \|\phi v(\tau)\|^{\theta}_{H_x^m}
  \lesssim \|\lr{x}^{-10} v(\tau)\|_{L_x^\infty}^{1-\theta} \|v(\tau)\|_{H_x^m}^{\theta} 
  \\
 \lesssim \varepsilon_0^{\theta_\al} \lr{\tau}^{-1+\theta_{\delta_2}}, \qquad \theta=2/m,
\end{split}
\end{align}
where $\theta_\al = \theta + \alpha(1-\theta)$ and 
$\theta_{\delta_2} = -(1-\theta)(1-\delta_2) + 1 = \delta_2(1-\theta) + \theta.$

\smallskip
To estimate  $I_{v}$ we apply \eqref{dispersive-estimates},
and \eqref{gagliardo-local-1} and use the standard product estimate for Sobolev spaces to obtain
\begin{align}\label{Ivtoref}
\begin{split}
{\left\|I_{v}(t)\right\|}_{L_x^\infty} & \lesssim \int_0^t \| L^{-1} e^{i(t-s)L} P_c(\beta v^2(s)) \|_{L_x^\infty} ds \\
& \lesssim \int_0^t \lr{t-s}^{-1/2} \left( \|\lr{x}\beta v^2(s) \|_{L_x^2} + \|\beta v^2(s)\|_{H_x^2} \right) ds 
\\
& \lesssim \int_0^t  
  \lr{t-s}^{-1/2} \left( \varepsilon_0^{2\alpha} \lr{s}^{2(-1+\delta_2)} 
  + \varepsilon_0^{2\alpha(1-\theta)} \lr{s}^{2(-1+\delta_2)(1-\theta)} 
    \cdot \varepsilon_0^{2\theta} \right) ds \lesssim \varepsilon_0^{2\al} \lr{t}^{-1/2},
\end{split}
\end{align}
where the last inequality holds for $m$ large enough so that $2(-1+\delta_2)(1-\theta) <-1$.

\smallskip
Using once again \eqref{gagliardo-local-1}
applied with $m=N$,  we have
 \begin{equation}\label{Iav0}
\begin{aligned}
{\|I_{a,v}^{(0)}(t)\|}_{L_x^\infty} &\lesssim \int_0^t \lr{t-s}^{-1/2}
\Rho^{1/2}(s) \left( \varepsilon_0^{\al} \lr{s}^{-1+\delta_2} + \varepsilon_0^{\theta_\al} \lr{s}^{-1+\theta_{\delta_2}} \right) ds 
\\
& \lesssim \int_0^t \varepsilon_0^{2\al}
  \lr{t-s}^{-1/2} 
  \left( \varepsilon_0^{\al} \lr{s}^{-3/2+\delta_2 + \alpha} + \varepsilon_0^{\theta_{\alpha}} \lr{s}^{-3/2+\theta_{\delta_2}+\al} \right) ds \lesssim \varepsilon_0^{2\al}\lr{t}^{-1/2} , 
  \end{aligned}
\end{equation}
where we have used \eqref{choice-of-csts} to deduce $\theta_{\delta_2}+\al \ll 1/2$ for $N$ sufficiently large.

\smallskip
Next, for each $\ell=1,2,3$ we have 
\begin{equation}
{\|L^{-1}e^{i(t-s)L}P_c(\varphi_\ell v^\ell(s))\|}_{L^\infty_x}\lesssim   \lr{t-s}^{-1/2} \Big(\varepsilon_0^{\alpha \ell}\lr{s}^{(-1+\delta_2)\ell}+\varepsilon_0^{\theta_\al \ell}\lr{s}^{(-1+\theta_{\delta_2})\ell}\Big),
\end{equation}
so that 
\begin{align}\label{Iav0'}
{\|I^{(\ell)}_{a,v}\|}_{L^\infty_x} 
  \lesssim \int_0^t\lr{t-s}^{-1/2}\Rho^{2-\ell/2}(s)\Big(\varepsilon_0^{\alpha \ell}\lr{s}^{(-1+\delta_2)\ell}+\varepsilon_0^{\theta_\al \ell}\lr{s}^{(-1+\theta_{\delta_2})\ell}\Big)ds\lesssim \varepsilon_0^{2\al}\lr{t}^{-1/2}.
\end{align}
Therefore, summing over all $\ell$ and combining it with \eqref{Iav0} 
we obtained the desired bound for $I_{a,v}$.

This concludes the proof of the proposition.  
\end{proof}


\medskip
\begin{proof}[{\bf Proof of Proposition \ref{lem:local-decay-bds}}]
We only prove the estimate for the singular term with $f^{(+)}$ since the bound 
for the term with $f^{(-)}$ as well as the regular term can be obtained in the same way.

From \eqref{L-1w}, 
we can write
\begin{equation}\label{sing-f}
\big( L^{-1}w- I_Q \big)_{S} = \big( L^{-1}e^{ itL}w_0 \big)_{S} +  (I_a)_S +(I_v)_S+(I_{a,v})_S.
\end{equation}

The term involving the initial data in \eqref{sing-f} can be controlled using Remark \ref{rmk:r-s-local-decay}
and \eqref{derivative-in-fourier-bound}:
\begin{align}
\|\lr{x}^{-2}(L^{-1}e^{itL}w_0)_S\|_{L^\infty_x}\lesssim \lr{t}^{-1}(\|w_0\|_{L^2}+\|\lr{x}w_0\|_{L^2})\lesssim \rho^{1/2}_0\lr{t}^{-1}
\end{align}
which is stronger than the desired estimate.
For $(I_a)_S$, we use \eqref{improved-local-decay-sing/reg}
and estimate 
\begin{equation*}
\begin{aligned}
&\left\|\lr{x}^{-2}(I_{a}(t))_S \right\|_{L_x^\infty} \\
&\quad  \lesssim \int_0^t \Big[a^2(s) \|\lr{x}^{-2} (L^{-1} e^{ i( t-s)L} P_c(\beta \varphi^2))_S \|_{L_x^\infty} +a^4(s) \|\lr{x}^{-2} (L^{-1} e^{ i( t-s)L} P_c(\varphi^4))_S \|_{L_x^\infty}\Big]ds \\
&\quad \lesssim \int_0^t  \lr{t- s}^{-3/2}(\Rho(s)+\Rho^2(s)) ds
\lesssim 
  \Rho(t)
\end{aligned}
\end{equation*}which is consistent with the desired bound \eqref{T-f-local-decay} since
$\Rho(t) \lesssim \Rho^{\al}(t) \lr{t}^{-1+\al} \lesssim \varepsilon_0^{2\al} \lr{t}^{-1+3\delta}$
and $\alpha\leq \delta$ (see \eqref{choice-of-csts}).

For the term $(I_v)_S$ we apply again \eqref{improved-local-decay-sing/reg}, 
together with the Gagliardo-Nirenberg interpolation \eqref{gagliardo-local-1} with $m= N$ to obtain
\begin{equation*}
\begin{aligned}
\left\|\lr{x}^{-2}(I_{v}(t))_S\right\|_{L_x^\infty} & \lesssim \int_0^t \Big\|\lr{x}^{-2} \big(L^{-1} e^{i(t-s)L} P_c(\beta v^2(s))\big)_{S} \Big\|_{L_x^\infty} ds 
\\
& \lesssim \int_0^t \lr{t- s}^{-3/2} \left( \|\lr{x}^2\beta v^2 \|_{L_x^2} 
  + \|\beta v^2\|_{H_x^2}+\|\lr{x}\beta v^2\|_{H^2_x} \right) ds \\
& \lesssim \int_0^t \lr{t- s}^{-3/2} \left[ \|\lr{x}^{-2} v \|_{L_x^\infty}^2 
  + \|\lr{x} \beta v^2\|_{L^2_x}^{1-\theta} \big( 
  \|\beta v^2\|_{H_x^N}^{\theta} + \|\lr{x} \beta v^2\|_{H_x^N}^{\theta} \big) \right] ds 
\\
& \lesssim \int_0^t  \lr{t- s}^{-3/2} \left( \varepsilon_0^{2\al} \lr{s}^{-2+2\delta_2} + \varepsilon_0^{2 \theta_{\al}} \lr{s}^{-2+2\theta_{\delta_2}} \right) ds
  \lesssim  \varepsilon_0^{2 \al} \lr{t}^{-3/2}.
\end{aligned}
\end{equation*}
Similarly, for the quadratic mixed term in $(I_{a,v})_S$ we can estimate 
\begin{equation}
|  a(s)| {\big\| \lr{x}^{-2}\big(L^{-1} e^{ i(t-s)L} P_c(\beta \varphi v(s))\big)_S \big\|}_{L^\infty}  \lesssim  \Rho^{1/2}(s)\lr{t- s}^{-3/2} \big(\varepsilon_0^{\alpha}\lr{s}^{-1+\delta_2}+ \varepsilon_0^{\theta_{\al}}\lr{s}^{-1+\theta_{\delta_2}} \big)
\end{equation}
and for each $\ell=1,2,3$,  
\begin{align}
& |a(s) ^{4-\ell}|\big\|\lr{x}^{-2}\big( L^{-1} e^{ i(t-s)L} 
  P_c(\varphi^{4-\ell} v^\ell(s)\big)_S\big\|_{L^\infty}
\\ 
& \lesssim \Rho^{2-\ell/2}(s)\lr{t-s}^{-3/2}\Big(\varepsilon_0^{\alpha \ell}\lr{s}^{(-1+\delta_2)\ell}+\varepsilon_0^{\theta_\al \ell}\lr{s}^{(-1+\theta_{\delta_2})\ell}\Big).
\end{align}
Integrating and summing, we obtain 
\begin{align}
\left\|\lr{x}^{-2}(I_{a,v}(t))_S\right\|_{L_x^\infty} 
  & \lesssim  
  \varepsilon_0^{2\al}\lr{t}^{-5/4},
\end{align}
which is more than sufficient for \eqref{T-f-local-decay}. The proof is now complete.
\end{proof}


\medskip
\begin{proof}[{\bf Proof of Proposition \ref{prop:Sobolev}}]
Below, we will use the following version of the Gagliardo-Nirenberg inequality:
\begin{equation}
\label{GN-endpoint}
\|\partial^{\ell}_x u\|_{L^\infty_x} \lesssim \|\partial^{N}_x u\|^{\theta_\ell}_{L^2_x}
  \|u\|^{1-\theta_\ell}_{L^\infty_x} , \qquad \theta_\ell = \frac{\ell}{N-1/2}
\end{equation}
for all $u\in H^{N}_x\cap L^{\infty}_x$ and $0 \leq \ell < N$. 
Note that using the standard $L^2$ theory from Lemma \ref{lemma-distorted-fourier} we have
$\|f\|_{L^2}+\|L^N f\|_{L^2} \sim \|f\|_{H^N}$,
so that in all $L^2$ estimates below, 
we can freely exchange the differentiation with respect to $L$ and $\partial_x$.
We use the same notation as in \eqref{Ia}-\eqref{Iav}.

\smallskip
First note that 
\begin{equation*}
\|L^{-1}e^{itL}w_0\|_{H^{N}_x}\lesssim  \|e^{itL}w_0\|_{H^{N-1}_x} \lesssim \|w_0\|_{H^{N-1}_x},
\end{equation*}
which is okay given the assumption on the data \eqref{mainthmdata}.

\smallskip
For the term $I_{a}$, using Lemma \eqref{lemma-distorted-fourier}  we have 
\begin{equation}
  \|I_a(t)\|_{H^N}\lesssim    \|I_a(t)\|_{L^2}+  \|L^NI_a(t)\|_{L^2}.  
\end{equation}
Thus, it suffices to estimate the second term on the right-hand
side as the same argument works for the first term.
We apply the smoothing estimate \eqref{smooth-est} to obtain
\begin{align*}
\left\|L^NI_{a}(t)\right\|_{L^2_x}&\lesssim  \left\|\int_0^ta^2(s)e^{-isL}L^{N-1}P_c(\beta \varphi^2)ds\right\|_{L^{2}_x}+\left\|\int_0^ta^4(s)e^{-isL}L^{N-1}P_c( \varphi^4)ds\right\|_{L^{2}_x}
\\
&\lesssim \left\|\int_0^t\int_\R a^2(s)e^{-is\lr{k}}\overline{\psi(x,k)}L^{N-1}(\beta\varphi^2)(x)dxds\right\|_{L^2_k}
\\
& \quad
+\left\|\int_0^t\int_\R a^4(s)e^{-is\lr{k}}\overline{\psi(x,k)}L^{N-1}(\varphi^4)(x)dxds\right\|_{L^2_k}    \\
&\lesssim \|\Rho(s) \|_{L^2_t}+\|\Rho^2(s) \|_{L^2_t} \lesssim \sqrt{\rho_0} \approx \varepsilon_0.
\end{align*}
Note that this recovers \eqref{H^N} provided the constant $C$ that appears there 
and in \eqref{w-assumptions-bootstrap-pde} is chosen large enough
depending on the absolute constant involved in the smoothing estimate just used above.

\smallskip
For $I_v$ in \eqref{Iv} we proceed similarly:
we first commute $L$ and distribute the derivatives, and then apply
the smoothing estimate \eqref{smooth-est} followed by H\"older to obtain
\begin{align}\label{IvH^N}
\left\|L^N I_v(t)\right\|_{L^2_x} 
& \lesssim \left\|\int_0^te^{i(t-s)L}L^{N-1}P_c(\beta v^2(s))ds\right\|_{L^{2}_x}  
\\
& \lesssim \sum_{0\leq \ell\leq N-1}
  \left\|\int_0^t\int_\R e^{-is\lr{k}}\overline{\psi(x,k)}
  \beta_\ell \, \partial_x^\ell v^2(s,x)]dxds \right\|_{L^2_k}
  \\
& \lesssim \sum_{0\leq \ell\leq N-1} {\Big\| 
  {\big\| \beta_\ell \, \partial_x^\ell v^2 (s,x) \big\|}_{L^1_x(\R)}
  \Big\|}_{L^2_s([0,t])}
\end{align}
where $\beta_\ell$ denote some generic Schwartz functions which may be different from
one line to the next, and we have switched the time and space integration by adding more weights in $x$.
Then we bound each of the terms in the last sum above using standard
product Sobolev estimate and the a priori bounds on the local decay and the high Sobolev norm
in \eqref{w-assumptions-bootstrap-pde}:
\begin{align}
{\Big\| {\big\| \beta_\ell\partial_x^\ell (  v^2(s,\cdot)) \big\|}_{L^1_x(\R)}
  \Big\|}_{L^2_s([0,t])}
  & \lesssim {\Big\|  
  {\big\| \lr{x}^{10} \beta_\ell \, v(s,\cdot) \big\|}_{H^\ell_x(\R)} \, 
  {\big\| \lr{x}^{-10} \, v(s,\cdot) \big\|}_{L^2_x(\R)}
  \Big\|}_{L^2_s([0,t])}
  \\
  & \lesssim
  \varepsilon_0^{1+\alpha} {\big\| \lr{s}^{-1+\delta_2} \big\|}_{L^2_s([0,t])} 
  \lesssim \varepsilon_0^{1+\alpha}.
\end{align}This concludes the desired bound for $I_v$ in \eqref{Iv}.

\smallskip
Next, we look at the quadratic mixed term in $I_{a,v}$
$$ I_{a,v}^{(0)}(t) = \int_0^t a(s) L^{-1} e^{i(t-s)L} P_c(\beta \varphi \, v(s)) ds.$$
Commuting $L$ through, using that $N-1$ is even and distributing the derivatives, we see that 
\begin{align}
\big\| L^N I_{a,v}^{(0)} (t)\big\|_{L^2}& \lesssim 
  \left\|\int_0^ta(s)e^{i(t-s)L}L^{N-1}(\beta \varphi v(s))ds\right\|_{L^{2}_x}\lesssim
  \sum_{0\leq \ell\leq N-1} \left\|\int_0^t a(s) e^{i(t-s)L}( \beta_\ell(x)\, \partial_x^\ell v(s) \big) ds\right\|_{L^{2}_x}
\end{align}
where $\beta_\ell$ denote some generic Schwartz functions.
We disregard the term $\ell=0$ which is easier as it does not require any interpolation.
For each $\ell=1,\dots N-1$ we look at the terms in the last inequality above and apply the smoothing
estimate \eqref{smooth-est}:
\begin{align}
{\left\|\int_0^t a(s) e^{i(t-s)L}( \beta_\ell(x)\, \partial_x^\ell v(s) \big) ds \right\|}_{L^{2}_x}
&= {\left\|\int_0^t \int_\R e^{i(t-s)\lr{k}} \overline{\psi(x,k)}
   \big( a(s) \beta_\ell\, \partial_x^\ell v(s) \big)  dx ds\right\|}_{L^{2}_k}
   \\
   &\lesssim {\big\| a(s) \beta_\ell \, \partial_x^\ell v(s) \big\|}_{L^{1}_xL^2_s([0,t])}.
\end{align}We use  the fast decay of $\beta_\ell$ and Cauchy-Schwarz
to exchange the order of integration, and then H\"older and the interpolation \eqref{GN-endpoint}
to obtain
\begin{align}
{\big\| a(s) \beta_\ell \, \partial_x^\ell v(s) \big\|}_{L^{1}_xL^2_s([0,t])}
  & \lesssim {\Big\| a(s) {\big\| \lr{x}^2 \beta_\ell(x) \partial_x^\ell v(s) \big\|}_{L^\infty_x} \Big\|}_{L^2_s([0,t])}
  \\
  & \lesssim {\Big\| a(s) \,
  {\big\| \lr{x}^{-10} v(s) \big\|}_{L^\infty(\R)}^{1-\theta_\ell} 
  {\big\| v(s) \big\|}_{H^N_x}^{\theta_\ell} \Big\|}_{L^2_s([0,t])}
\end{align}
where, recall, $\theta_\ell := 2\ell/(2N-1)$.
Using the last two a priori bounds in \eqref{w-assumptions-bootstrap-pde},
we can estimate
\begin{align}\label{cond0}
\begin{split}
{\big\| a(s) \beta_\ell \, \partial_x^\ell v(s) \big\|}_{L^{1}_xL^2_s([0,t])}
  & \lesssim {\Big\| \min\big( \varepsilon_0,s^{-1/2} \big) \cdot
  \big( \varepsilon_0^\alpha\lr{s}^{-1+\delta_2}  \big)^{1-\theta_\ell} 
  \cdot \varepsilon_0^{\theta_\ell} \Big\|}_{L^2_s([0,t])} 
  \\
  & \lesssim 
  \varepsilon_0^{2\gamma_\ell} \cdot \varepsilon_0^{\alpha(1-\theta_\ell)} \cdot \varepsilon_0^{\theta_\ell}. 
\end{split}
\end{align}
for any $\gamma_\ell \in (0,1/2)$ such that 
\begin{align}\label{cond1}
\gamma_\ell < (1-\theta_\ell)(1-\delta_2).
\end{align}
The estimate \eqref{cond0} is sufficient to recover the a priori bound provided that
\begin{align}
\label{cond2}
2\gamma_\ell + \alpha(1-\theta_\ell) + \theta_\ell > 1.
\end{align}
We can then choose, for example, $\gamma_\ell = (1-\theta_\ell)/2$ so that both condition 
\eqref{cond2} and \eqref{cond1} are satisfied.

\smallskip
For the other term  in \eqref{Iav}, it will  suffice to estimate
\begin{equation*}
  I^{(\ell)}_{a,v}(t)= \int_0^t a(s)^{4-\ell} L^{-1} e^{i(t-s)L} P_c(\varphi_\ell \, v^\ell(s))  ds
\end{equation*}
for each $\ell=1,2,3$, and the same arguments as before apply
since it contains polynomials of higher degree in $a(s)$ and $v(s)$,
with localized coefficients.

Finally, for the term $I_Q$ in \eqref{I_Q}, using the bootstrap assumptions
on the $L^\infty$ norm and high Sobolev norms in \eqref{w-assumptions-bootstrap-pde} 
we can easily estimate:
\begin{align*}
{\|L^{N}I_Q(t) \|}_{L^2_x} 
& = {\left\|\int_0^tL^{N-1}e^{i(t-s)L}P_c( v^4(s))ds\right\|}_{L^2_x}
\\
& \lesssim \int_0^t {\left\|v^4(s)\right \|}_{H^{N-1}_x} ds 
  \lesssim \int_0^t {\|v(s)\|}^3_{L^\infty_x}
  {\|v(s)\|}_{H^{N-1}_x}ds\lesssim \varepsilon_0^{3\al+1}.
\end{align*}The proof of Proposition \ref{prop:Sobolev} is complete.
\end{proof}


\bigskip
\section{Bootstrap for the continuous component III: the quartic terms}\label{sec:PDE-part2}
This section is devoted to proving all the remaining (and harder) estimates on the quartic contributions
so that we can conclude the proof of the bootstrap estimates in Proposition \ref{prop:bootstrap-on-v}.
To describe how we will proceed to do this, let us first recall the notation from \eqref{I_Q} and \eqref{term-Q}, that is
\begin{align}\label{term-Q'}
I_Q(t) = L^{-1} e^{itL} Q(t), \qquad Q(t,x) = \int_0^t e^{-isL} P_c (v^4)(s,x) ds.
\end{align}
In this section we will do the following:

\begin{itemize}
 \smallskip
 \item[1.] Prove the weighted-type estimate \eqref{Q-tildeL20}, that is,
 \begin{align}
 {\big\|\partial_k \wt{Q}(t,\cdot) \big\|}_{L_k^2}
 \lesssim \varepsilon_0^{3\al}\lr{t}^{8\delta}.
 \end{align}
 This concludes the proof of Proposition \ref{prop:good-part} and recovers the first improved 
 a priori bound 
 in \eqref{w-conclusion-bootstrap}.
 
 \smallskip
 \item[2.] Prove the decay estimate
 \begin{equation}
 \left\| I_Q(t) \right\|_{L^\infty_x}\lesssim \varepsilon_0^{2\al} \lr{t}^{-1/2+\delta};
 \end{equation}
 together with Proposition \ref{prop:decay_boostrap}
 this will recover the second a priori bound \eqref{w-conclusion-bootstrap}.

 \smallskip
 \item[3.] Prove the local decay estimate
 \begin{equation}
 \left\| \lr{x}^{-2} \big( I_Q(t) \big)_{E} \right\|_{L_x^\infty}
 \lesssim \varepsilon_0^{3\al} \lr{t}^{-1+8\delta} \qquad E\in\{S,R\};
 \end{equation}
 together with Proposition \ref{lem:local-decay-bds}
 this will recover the last a priori bound in \eqref{w-conclusion-bootstrap}.

 
 
\end{itemize}

\medskip
\subsection{Set-up and reductions}
Here are all the statements of the results we will prove in this section:

\begin{lemma}\label{lem:Q-tilde}
Under the hypotheses of Proposition $\ref{prop:bootstrap-on-v}$ 
we have, for all $t \in [0,T]$, 
\begin{equation}\label{Q-tildeL2}
\|\partial_k\wt{Q}(t,\cdot)\|_{L^2_k}\lesssim \varepsilon_0^{3\al}\lr{t}^{8\delta}.
\end{equation}
\end{lemma}

\begin{lemma}\label{lem:Q-in-Linfty} Under the hypotheses of 
Proposition $\ref{prop:bootstrap-on-v}$, 
for all $t \in [0,T]$, there holds the estimate
\begin{equation}\label{linear-decay-quartic}
\left\|\int_0^t L^{-1} e^{i(t-s)L} P_c v^4(s) ds \right\|_{L_x^\infty}
 \lesssim \varepsilon_0^{1/2} \lr{t}^{-1/2+\delta/2}.
\end{equation} \end{lemma}

\begin{lemma}\label{lem:local-Q-in-Linfty}
Assume the hypotheses of Proposition $\ref{prop:bootstrap-on-v}$.
Then, for all $t \in [0,T]$,
\begin{equation}\label{T_Q}
\left\|\lr{x}^{-2} \Big( L^{-1}e^{\pm itL} 
  Q^{(\pm)}(t) \Big)_{E}
  \right\|_{L_x^\infty}\lesssim \varepsilon_0^{3\al} \lr{t}^{-1+8\delta},
  \quad E \in \{S,R\}.
\end{equation} 
\end{lemma}

The proofs of all the lemmas above can be reduced to proving a
bound on the derivative in distorted Fourier space of a time localized version of
the quartic term $Q$ in \eqref{term-Q'}, stated in the next lemma:

\begin{lemma}[The main quartic bound]\label{lemreduction}
Let $\{\tau_m\}_{m=0,\dots,L+1}$ be a partition of the time interval $[0,t]$
as 
in \eqref{timeloc}, and define
\begin{align}\label{reductionestQm}
Q_m(x) :=\int_0^\infty e^{-isL}P_c(v^4)(s,x)\tau_m(s)ds, \qquad  Q(t,x) = \sum_{m=0}^{L+1} Q_m(x).
\end{align}
Then, under the a priori assumption in Proposition \ref{prop:bootstrap-on-v}, we have,
for all $m \in \mathbb{Z} \cap [0,L+1]$,
\begin{align}\label{reductionest}
{\big\| \partial_k \wt{Q_m}(\cdot) \big\|}_{L^2_k}
  \lesssim \varepsilon_0^{3\al}2^{8\delta m}.
\end{align}
\end{lemma}

We now show how  Lemma \ref{lemreduction} implies all the needed results
and will then proceed to prove it in Subsection \ref{sec:Quartic}.

\begin{proof}[Proof of Lemma \ref{lem:Q-tilde}, Lemma \ref{lem:Q-in-Linfty} 
and Lemma \ref{lem:local-Q-in-Linfty} using Lemma \ref{lemreduction}]
First, observe that \eqref{reductionest} directly implies \eqref{Q-tildeL2}
upon summing over $m$ with $2^m \lesssim t$.

Applying the linear interpolation-type estimate \eqref{lininterpest},
according to the notation in \eqref{term-Q'}, 
we can bound the left-hand side of \eqref{linear-decay-quartic} as follows:
\begin{align}\label{redpr1}
\begin{split}
\left\|\int_0^t L^{-1} e^{i(t-s)L} P_c v^4(s) ds \right\|_{L_x^\infty}
  & = \left\| e^{itL} L^{-1} Q(t,\cdot) \right\|_{L_x^\infty}
  \lesssim \sum_{m=0}^{L+1} \left\| e^{itL} L^{-1} Q_m(\cdot) \right\|_{L_x^\infty}
  \\
  & \lesssim
  \lr{t}^{-1/2} \log(2+t) \sum_{m=0}^{L+1}\Big( 
  {\big\|  L^{-1} Q_m(\cdot) \big\|}_{H^4}^{1/2} 
  {\big\| x  L^{-1} Q_m(\cdot) \big\|}_{L^2}^{1/2}+\big\|  L^{-1} Q_m(\cdot) \big\|_{H^N}\Big).
\end{split}
\end{align}   
It is straightforward to see, using our bootstrap assumptions, that 
\begin{align}\label{redpr2}
\begin{split}
{\big\|  L^{-1} Q_m(\cdot) \big\|}_{H^\ell}
  & \lesssim \int_0^\infty {\left\| P_c(v^4)(s,\cdot)  \right\|}_{H^{\ell-1}}\tau_m(s)ds   
  \\
  & \lesssim \int_0^\infty {\left\| v(s,\cdot)  \right\|}_{H^{\ell-1}}
  {\left\| v(s,\cdot) \right\|}_{L^\infty}^3 \tau_m(s)ds  
  \\
  & \lesssim 2^m \cdot \varepsilon_0 \cdot \big( \varepsilon_0^\alpha \, 2^{-\frac{m}{2}+\delta m}\big)^3    
  = \varepsilon_0^{1+3\al} 2^{-\frac{m}{2}+3\delta m}.    
\end{split}
\end{align}
Using this with $\ell=4$,
\eqref{derivative-in-fourier-bound}, and the estimate \eqref{reductionest},
we can bound
\begin{align}
{\big\| L^{-1} Q_m(\cdot) \big\|}_{H^4}
  {\big\| x  L^{-1} Q_m(\cdot) \big\|}_{L^2}
  \lesssim \varepsilon_0^{1+6\al} 2^{-\frac{m}{2} + 11\delta m}.
\end{align}
Plugging this bound, and \eqref{redpr2} with $\ell=N$, in \eqref{redpr1} we obtain
\begin{align}\label{IQtoref}
{\big\| e^{itL} L^{-1} Q(t,\cdot) \big\|}_{L_x^\infty}
  \lesssim \lr{t}^{-1/2} \log(2+t) 
  \sum_{m=0}^{L+1} \varepsilon_0^{1/2+3\al} 2^{-\frac{m}{4}+\frac{11}{2}\delta m}
  \lesssim \varepsilon_0^{1/2} \lr{t}^{-1/2} \log(2+t) 
\end{align}   
for $\delta$ small enough. This bound is stronger than \eqref{linear-decay-quartic}.

Finally, to obtain \eqref{T_Q} we can apply the linear estimates
\eqref{local-l-infty-decay-sing}-\eqref{local-l-infty-decay-reg} and bound,
for $E \in \{S,R\}$,
\begin{align}
{\left\|\lr{x}^{-2} \Big(\int_0^t L^{-1} e^{itL} e^{-isL}
  P_c( v^4(s)) ds \Big)_{E} \right\|}_{L_x^\infty}
  &= {\left\| \lr{x}^{-2} \Big( e^{itL} L^{-1} Q(t,\cdot) \Big)_{E} \right\|}_{L_x^\infty}
  \\
  &\lesssim \sum_{m=0}^{L+1} 
  {\left\| \lr{x}^{-2} \Big( e^{itL} L^{-1} Q_m(\cdot) \Big)_E \right\|}_{L_x^\infty}
  \\
&  \lesssim
  \frac{1}{\lr{t}} \sum_{m=0}^{L+1} 
  \big( {\big\| \widetilde{Q_m} \big\|}_{L^2_k} + {\big\| \partial_k \widetilde{Q_m} \big\|}_{L^2_k} \big)
  \lesssim \frac{1}{\lr{t}} \sum_{m=0}^{L+1} \varepsilon_0^{3\al} 2^{8\delta m},
\end{align}
having used \eqref{reductionest};
summing over $m \lesssim L+1 \approx \log (2+t)$ gives the desired \eqref{T_Q}. 

\end{proof}

\medskip
\subsection{Proof of Lemma \ref{lemreduction}
} \label{sec:Quartic}

In this subsection we prove the estimate \eqref{reductionest}
and therefore conclude the proof of the main bootstrap Proposition \ref{prop:bootstrap-on-v}.
The proofs are based on the use of the distorted Fourier transform
and associated multilinear analysis.

\medskip
\subsubsection{The nonlinear spectral distribution and decomposition of the quartic term}
For a complex-valued function $f$ we denote
\begin{equation}
    f^{(+)}=f, \qquad f^{(-)}=\overline{f}.
\end{equation}
Using the relations $P_cu = \Im(L^{-1}w)$ and $f=e^{-itL} w$, 
and the distorted Fourier inversion formula from Lemma \ref{lemma-distorted-fourier} 
we can write 
\begin{align}\label{Fourier-v^4}
 \begin{split}
2^4e^{-is \langle k\rangle}\wt{\mathcal{F}}(v^4)(k)
=\Suml{\ell_1}{\ell_2}{\ell_3}{\ell_4}\ell_1\ell_2\ell_3\ell_4
  \iiiint \frac{e^{is\Upsilon_{\ell_1\ell_2\ell_3\ell_4}}}
  {\prod_{j=1}^4\langle \xi_j\rangle}\threef{f}{\ell_1}{\xi_1}{f}{\ell_2}{\xi_2}{f}{\ell_3}{\xi_3}\onef{f}{\ell_4}{\xi_4}
\\
\times\, \mu_{\ell_1\ell_2\ell_3\ell_4}(k,\xi_1,\xi_2,\xi_3,\xi_4)
  \, d\xi_{1234}  
 \end{split}
\end{align}
where we are adopting the shorthand notation 
$$d\xi_{1234} = d\xi_1d\xi_2d\xi_3d\xi_4$$ 
and we are defining the ``nonlinear spectral distribution'' (NSD)
\begin{equation}\label{NSD}
\mu_{\ell_1\ell_2\ell_3,\ell_4}(k,\xi_1,\xi_2,\xi_3,\xi_4)
  =\int \firstpsi{\psi}{x}{k}{\psi}{\ell_1}{x}{\xi_1} 
  \scdpsi{\psi}{\ell_2}{x}{\xi_2}{\psi}{\ell_3}{x}{\xi_3}\psi^{(\ell_4)}(x,\xi_4)dx
  \end{equation}
and the nonlinear phase function
\begin{equation}\label{Upsilonphase}
\Upsilon_{\ell_1\ell_2\ell_3\ell_4}(k,\xi_1,\xi_2,\xi_3,\xi_4)
=-\langle k\rangle+\ell_1\langle \xi_1\rangle+
\ell_2\langle \xi_2\rangle+\ell_3\langle \xi_3\rangle+\ell_4\langle \xi_4\rangle.
\end{equation}


According to \eqref{generalized-eigenf-decomp} 
we decompose the nonlinear spectral distribution into a singular (S) and a regular (R)
part:
\begin{equation}\label{NDF}
(2\pi)^{\frac{5}{2}}\mu_{\ell_1\ell_2\ell_3\ell_4}(k,\xi_1,\xi_2,\xi_3,\xi_4)
:=\mu_{\ell_1\ell_2\ell_3\ell_4}^S(k,\xi_1,\xi_2,\xi_3,\xi_4)
+\mu_{\ell_1\ell_2\ell_3\ell_4}^R(k,\xi_1,\xi_2,\xi_3,\xi_4),
\end{equation}
where the singular part collects all singular terms in the decomposition 
of the generalized eigenfunctions \eqref{generalized-eigenf-decomp}, that is,
\begin{equation} \label{sing-mu0}
\mu^{S}_{\ell_1\ell_2\ell_3\ell_4}(k,\xi_1,\xi_2,\xi_3,\xi_4) 
= \mu^+_{\ell_1\ell_2\ell_3\ell_4}(k,\xi_1,\xi_2,\xi_3,\xi_4)
+\mu^{-}_{\ell_1\ell_2\ell_3\ell_4}  (k,\xi_1,\xi_2,\xi_3,\xi_4)
\end{equation}
with (recall the notation \eqref{psi-singular-part} and \eqref{decomposition-sing-part})
\begin{align}
 \label{sing-mu}
 \mu^{\ast}_{\ell_1\ell_2\ell_3\ell_4}(k,\xi_1,\xi_2,\xi_3,\xi_4)
 =\int \Lambda_{\ast}(x)\overline{\psi_{S,\ast}(x,k)}\psi^{(\ell_1)}_{S,\ast}(x,\xi_1)
 \psi^{(\ell_2)}_{S,\ast}(x,\xi_2)\psi^{(\ell_3)}_{S,\ast}(x,\xi_3)
 \psi^{(\ell_4)}_{S,\ast}(x,\xi_4)dx,
 \end{align}
where $ \ast\in \setsign$ and $\Lambda_{\ast} =\sigma_{\ast}^5$. 
Note that, in particular, 
the integrand in \eqref{sing-mu} does not decay as $|x|\rightarrow \infty$.

The regular part $\mu^{R}$ is defined according \eqref{NSD} and \eqref{NDF}.
To describe it more precisely we let
\[Y_R=\{\mathbf{E}=(E_0,...,E_4):\exists \,k\in \{0,...,4\}\,\,\mbox{such that}\,\, E_k=R\}\subset \{S,R\}^5,\]
and decompose
\begin{equation}
 \mu^{R}_{\ell_1\ell_2\ell_3\ell_4}(k,\xi_1,\xi_2,\xi_3,\xi_4)=\mu^{R,1}_{\ell_1\ell_2\ell_3\ell_4}(k,\xi_1,\xi_2,\xi_3,\xi_4)+\mu^{R,2}_{\ell_1\ell_2\ell_3\ell_4}(k,\xi_1,\xi_2,\xi_3,\xi_4)   
\end{equation}
where 
\begin{align}\label{mur1}
\mu^{R,1}_{\ell_1\ell_2\ell_3\ell_4}(k,\xi_1,\xi_2,\xi_3,\xi_4)=\sum_{\mathbf{E}\in Y_R}\int\overline{\psi_{E_0}(x,k)}\psi_{E_1}^{(\ell_1)}(x,\xi_1)\psi_{E_2}^{(\ell_2)}(x,\xi_2)\psi_{E_3}^{(\ell_3)}(x,\xi_3)\psi_{E_4}^{(\ell_4)}(x,\xi_4)dx
\end{align}
and
\begin{align}\label{mur2}
\begin{split}
& \mu^{R,2}_{\ell_1\ell_2\ell_3\ell_4}(k,\xi_1,\xi_2,\xi_3,\xi_4)
\\
& =\int\overline{\psi_{S}(x,k)}\psi_{S}^{(\ell_1)}(x,\xi_1)\psi_{S}^{(\ell_2)}(x,\xi_2)\psi_{S}^{(\ell_3)}(x,\xi_3)\psi_{S}^{(\ell_4)}(x,\xi_4)dx
-    \mu^{S}_{\ell_1\ell_2\ell_3\ell_4}(k,\xi_1,\xi_2,\xi_3,\xi_4)
\end{split}
\end{align}
Notice that $\mu^{R,1}$ contains at least one $\psi_R$ which has fast decay as 
$|x|\rightarrow\infty$, while, according to \eqref{sing-mu0}-\eqref{sing-mu},
the component $\mu^{R,2}$ contains at least a pair of $\sigma_+\psi_{S_+}$ and $
\sigma_-\psi_{S_-}$ so that the integrand is compactly supported in $x$.

According to \eqref{NDF}-\eqref{mur2} above, we can decompose
\begin{align}\label{partial_k-decomposed}
2^4 e^{-is \langle k\rangle}\wt{\mathcal{F}}(v^4(s,\cdot))(k)
  = (2\pi)^{-5/2} \big[ I_{-}(s,k)+I_{+}(s,k)+I_{R,1}(s,k)+I_{R,2}(s,k) \big],   
\end{align}
with the natural definitions
\begin{align}
\begin{split}
I_{\ast}(s,k) := \Suml{\ell_1}{\ell_2}{\ell_3}{\ell_4}\ell_1\ell_2\ell_3\ell_4
  \iiiint \frac{e^{is\Upsilon_{\ell_1\ell_2\ell_3\ell_4}}}
  {\prod_{j=1}^4\langle \xi_j\rangle}\threef{f}{\ell_1}{\xi_1}{f}{\ell_2}{\xi_2}{f}{\ell_3}{\xi_3}\onef{f}{\ell_4}{\xi_4}
\\
\times \,\mu^{\ast}_{\ell_1\ell_2\ell_3\ell_4}(k,\xi_1,\xi_2,\xi_3,\xi_4)
  \, d\xi_{1234}.
 \end{split}
\end{align}


According to this we can then decompose the term $Q_m$ in \eqref{reductionestQm} as follows:
\begin{equation}\label{Qm-decomposed}
2^4 \, \wt{Q_m}(k) = (2\pi)^{-5/2} \big[
  Q_{m,-}(k) + Q_{m,+}(k) + Q_{m,R,1}(k) + Q_{m,R,2}(k)\big],
\end{equation}
with the natural definitions
\begin{align}
\begin{split}
Q_{m,\ast}(k) := \Suml{\ell_1}{\ell_2}{\ell_3}{\ell_4}\ell_1\ell_2\ell_3\ell_4
  \int_0^\infty \iiiint \frac{e^{is\Upsilon_{\ell_1\ell_2\ell_3\ell_4}}}
  {\prod_{j=1}^4\langle \xi_j\rangle}\threef{f}{\ell_1}{\xi_1}{f}{\ell_2}{\xi_2}{f}{\ell_3}{\xi_3}
  \onef{f}{\ell_4}{\xi_4}
\\
\times \,\mu^{\ast}_{\ell_1\ell_2\ell_3\ell_4}(k,\xi_1,\xi_2,\xi_3,\xi_4)
  \, d\xi_{1234} \,\tau_m(s) ds.
 \end{split}
\end{align}

To prove Lemma \ref{lemreduction} it then suffices to show that
\begin{align}\label{reductionestdec}
{\big\| \partial_k Q_{m,\ast}(\cdot) \big\|}_{L^2_k}
  \lesssim \varepsilon_0^{3\al} 2^{8\delta m} \qquad \ast = \{+,-,(R,1),(R,2)\} ,\quad m=0,\dots L+1.
\end{align}
In what follows we will show that every term except one (the term $Q_{m,2}^{(0)}$ below)
satisfies better estimates.
Moreover, we will estimate each dyadic-time block when $m=1,\dots ,L$,
disregarding the cases $m=0$ (which is easy to treat since $\mathrm{supp}(\tau_0) \subset [0,2]$) 
and $m=L+1$ which can be treated similarly to the others
using that 
$|\tau'_{L+1}(t)| \lesssim 2^{-L}$ on $(0,t)$.

\medskip
\subsubsection{Control of $Q_{m,\pm}$}\label{sec:Qpm}
Starting with the singular terms, we write  
\begin{align*}
I_{\ast}(s,k)= \Suml{\ell_1}{\ell_2}{\ell_3}{\ell_4} \ell_1\ell_2\ell_3\ell_4 \iiiint\frac{e^{is\Upsilon_{\ell_1\ell_2\ell_3\ell_4}}}{\prod_{j=1}^4\langle \xi_j\rangle} \threef{f}{\ell_1}{\xi_1}{f}{\ell_2}{\xi_2}{f}{\ell_3}{\xi_3}\onef{f}{\ell_4}{\xi_4}\mu_{\ell_1\ell_2\ell_3\ell_4}^{\ast}d\xi_{1234}
\end{align*}
where $\ast\in \{+,-\}$.
Using \eqref{A-functions} and from \eqref{sing-mu}, we express each component of the singular part of the nonlinear distribution function as \begin{align}  
\mu^{\ast}_{\ell_1\ell_2\ell_3\ell_4}(k,\xi_1,\xi_2,\xi_3,\xi_4)=\sqrt{2\pi}\sum_{s_0,s_1,s_2,s_3,s_4\in \{+,-\}}\overline{A_{\ast,s_0}(k)}A^{(\ell_1)}_{\ast,s_1}(\xi_1)A^{(\ell_2)}_{\ast,s_2}(\xi_2)A^{(\ell_3)}_{\ast,s_3}(\xi_3)A^{(\ell_4)}_{\ast,s_4}(\xi_4)\widehat{\Lambda_{\ast}}(q)
\end{align}
where $q=s_0k-s_1\ell_1\xi_1-s_2\ell_2\xi_2-s_3\ell_3\xi_3-s_4\ell_4\xi_4$.
The following lemma is exploited to show that $I_{\pm}$ can be written as a multiplier operator.

\begin{lemma}\label{lem:FT-varLambda}
The Fourier transform of the cut-off function $\Lambda_{\ast}$ is given by
\begin{equation}
\widehat{\Lambda_{\pm}}(k) = \sqrt{\frac{\pi}{2}}\delta_0 
  \pm \mathrm{p.v.} \frac{\widehat{\eta}(k)}{ik} + \widehat{\overline{\omega}_\pm}(k),   
\end{equation}
where $\eta\in C^{\infty}_c(\R)$ is an even function with  normalized integral and 
$\overline{\omega}$ denotes a function in $ C^{\infty}_c$.
Moreover, \begin{equation}
k \widehat{\Lambda_{\pm}}(k)=\mp i\widehat{\eta}(k)+k\widehat{\overline{\omega}}_\pm(k)
  :=\widehat{\Theta}_\pm(k)    
\end{equation}
for some function $\Theta_\pm \in \mathcal{S}(\mathbb{R})$. 
\end{lemma}

\begin{proof}
Noticing that $\partial_x \Lambda_{\pm}$ is a smooth function with compact support,
the proof of this lemma can be done following the argument in \cite[Page 42]{GP22-mathpi}. 
\end{proof}

\begin{lemma}\label{lem:quadriform}
Let $\Omega$ be a distribution defined on $\mathbb R$. 
For $f_j\in \mathrm{Lip}(\R)$, define the quadrilinear form 
\begin{align*}&\mathcal{Q}^{\ell_1,\ell_2,\ell_3,\ell_4}_{\Omega,s_0,s_1,s_2,s_3,s_4}(f_1,f_2,f_3,f_4)(k)
\\
&\quad =\ell_1\ell_2\ell_3\ell_4\iiiint\frac{e^{is\Upsilon_{\ell_1,\ell_2,\ell_3,\ell_4}}}{\prod_{j=1}^4\langle \xi_j\rangle} 
\threefct{f_1}{\ell_1}{\xi_1}{f_2}{\ell_2}{\xi_2}{f_3}{\ell_3}{\xi_3}\onefct{f_4}{\ell_4}{\xi_4}\Omega(q)d\xi_{1234}\end{align*} where $\ell_j,s_j\in \{+,-\}$ and $q=s_0k-s_1\ell_1\xi_1-s_2\ell_2\xi_2-s_3\ell_3\xi_3-s_4\ell_4\xi_4$.
We have
\begin{align}\label{I*-Quadriform}
\begin{split}
I_{\ast}(s,k)& = \sum_{\substack{(\ell_1,\ell_2,\ell_3,\ell_4)
\in \{+,-\}^4 \\ (s_0,s_1,s_2,s_3,s_4)\in \{+,-\}^5}}  \overline{A_{\ast,s_0}(k)}
\\
& \times \mathcal{Q}^{\ell_1,\ell_2,\ell_3,\ell_4}_{\widehat{\Lambda_{\ast}},s_0,s_1,s_2,s_3,s_4}		(A_{\ast,s_1} \widetilde{f_1},
  A_{\ast,s_2} \widetilde{f_2},
  A_{\ast,s_3} \widetilde{f_3},
  A_{\ast,s_4} \widetilde{f_4})(k).  
\end{split}
\end{align}
\end{lemma}

The above lemma is particularly useful for finding 
a representation formula for $\partial_k I_{\ast}(s,k)$.
By simple differentiation rules and integration by parts, we obtain the 
identity in the following lemma:

\begin{lemma}
\label{partial-k-I}
For all $f_i\in \mathrm{Lip}
(\R)$, we have the following identity
\begin{align*}
\nonumber\langle  k\rangle \partial_k\mathcal{Q}^{\ell_1,\ell_2,\ell_3,\ell_4}_{\Omega,s_0,s_1,s_2,s_3,s_4}
(f_1,f_2,f_3,f_4)(k)=\sum_{j=1}^6 s_0 \mathcal{Q}_j(f_1,f_2,f_3,f_4)(k)
\end{align*}
where  the forms $\mathcal{Q}_\ell$, $\ell=1,...,6$ are defined as
\begin{align*}
\begin{aligned}
\mathcal{Q}_1(f_1,f_2,f_3,f_4)(k) &= 
-is \mathcal{Q}^{\ell_1,\ell_2,\ell_3,\ell_4}_{\zeta \Omega,s_0,s_1,s_2,s_3,s_4}(f_1,f_2,f_3,f_4)(k)
\\
\mathcal{Q}_2(f_1,f_2,f_3,f_4)(k) &=
s_1\mathcal{Q}^{\ell_1,\ell_2,\ell_3,\ell_4}_{\Omega,s_0,s_1,s_2,s_3,s_4}
  (\langle\cdot\rangle \partial_{\xi_1}f_1,f_2,f_3,f_4)(k)
\\
\mathcal{Q}_3(f_1,f_2,f_3,f_4)(k) &=s_2\mathcal{Q}^{\ell_1,\ell_2,\ell_3,\ell_4}_{\Omega,s_0,s_1,s_2,s_3,s_4}(f_1,\langle\cdot\rangle \partial_{\xi_2}f_2,f_3,f_4)(k)
\\
\mathcal{Q}_4(f_1,f_2,f_3,f_4)(k) &=s_3\mathcal{Q}^{\ell_1,\ell_2,\ell_3,\ell_4}_{\Omega,s_0,s_1,s_2,s_3,s_4}(f_1,f_2,\langle\cdot\rangle \partial_{\xi_3}f_3,f_4)(k)
\\
\mathcal{Q}_5(f_1,f_2,f_3,f_4)(k) &=s_4\mathcal{Q}^{\ell_1,\ell_2,\ell_3,\ell_4}_{\Omega,s_0,s_1,s_2,s_3,s_4}(f_1,f_2,f_3,\langle\cdot\rangle \partial_{\xi_4}f_4)(k)
\\
\mathcal{Q}_6(f_1,f_2,f_3,f_4)(k) & = -\ell_1\ell_2\ell_3\ell_4 \iiiint\frac{\Upsilon e^{is\Upsilon}}{\prod_{j=1}^4\langle \xi_j\rangle} f_1^{(\ell_1)}(\xi_1)f_2^{(\ell_2)}(\xi_2)
  f_3^{(\ell_3)}(\xi_3) f_4^{(\ell_4)}(\xi_4)\partial_\zeta \Omega(\zeta)d\xi_{1234}
\end{aligned}
\end{align*}
with 
$\zeta=s_0k-s_1\ell_1\xi_1-s_2\ell_2\xi_2-s_3\ell_3\xi_3-s_4\ell_4\xi_4$
and $\Upsilon:=\Upsilon_{\ell_1,\ell_2,\ell_3,\ell_4}$ as in \eqref{Upsilonphase}.
\end{lemma}

\begin{lemma}\label{lem:transference}
Let $\Omega$ be a tempered distribution on $\R$ with bounded transform. We have (with $s_0 = +$)
  \begin{align}
       \widehat{\mathcal{F}}^{-1}\big(e^{is\langle k \rangle}
       \mathcal{Q}^{\ell_1,\ell_2,\ell_3,\ell_4}_{\Omega,s_0,s_1,s_2,s_3,s_4}(f_1,f_2,f_3,f_4)(k)\big)(x)
       =
       \widehat{\mathcal{F}}^{-1}(\Omega)(x)\prod_{j=1}^4 v_j(s,s_jx)       
  \end{align}
  where $v_j(s,x)=\langle D\rangle^{-1}e^{is\langle D\rangle}
  \widehat{\mathcal{F}}^{-1}(f_j^{(\ell_j)})(x)$, $\langle D\rangle=\sqrt{-\partial^2_{xx}+1}$.
\end{lemma}

\smallskip
Motivated by \eqref{I*-Quadriform} and the previous lemma, we define
\begin{equation}\label{wj} 
w_j(s,x)=\langle D\rangle^{-1}e^{is\langle D\rangle}\widehat{\mathcal{F}}^{-1}(F_j(s,\cdot))(x),
  \qquad F_j(s,\xi)=A_{+,s_j}(\xi)\widetilde{f}(s,\xi). 
\end{equation} 
By Plancherel we have
\begin{equation}\label{wjL2}
\|w_j(s,\cdot)\|_{L^2_x} \lesssim {\| \lr{k}^{-1}\widetilde{f}(s,\cdot)\|}_{L^2_k}.   
\end{equation}
Moreover, from Lemma \ref{lem:BoundednessW} we know that
 \begin{equation}\label{wjLinfty}
\|w_j(s)\|_{L^\infty_x}
 \lesssim  \varepsilon_0^{\al} \lr{s}^{-1/2+2\delta}.   
\end{equation}

Next, note that the functions $A_{\star,s_j}$ exhibit jump discontinuities at $k=0$. To take the $\partial_k$-derivative in \eqref{I*-Quadriform} we need to rearrange the term in $I_{\ast}$.  First, observe that according to \eqref{T-R-near-zero} and \eqref{A-functions}, we have
\begin{align}\label{A+der}
\begin{aligned}
\partial_k A_{+,-}(k)&=\delta_0(k) -\mathbbm{1}_-(k)\partial_k R_{+}(-k),\\
\sum_{s_0\in \{+,-\}}\partial_kA_{+,s_0}(k)&= \mathbbm{1}_+(k)\partial_k T(k)-\mathbbm{1}_-(k)\partial_kR_+(-k),
\end{aligned}
\end{align}
and 
\begin{align}\label{A-der}
  \begin{aligned}\partial_k A_{-,+}(k)&=\delta_0(k) -\mathbbm{1}_-(k)\partial_k T(-k)\\
\partial_k\sum_{s_0\in \{+,-\}}A_{-,s_0}(k)&= \mathbbm{1}_+(k)\partial_k R_-(k)-\mathbbm{1}_-(k)\partial_kT(-k).
    \end{aligned}  
\end{align}
Hence, we rewrite  $I_+(s,k)$ (see \eqref{I*-Quadriform}) in the expanded form
\begin{align}
\label{rearange-I+}
\begin{aligned}
&\sum_{s_0,...,s_4 \in \{+,-\}}\overline{A_{+,s_0}(k)}\mathcal{Q}_{\widehat{\Lambda_+},+,s_1,s_2,s_3,s_4}(F_1,F_2,F_3,F_4)(k)=\\
&\sum_{s_1,...,s_4\in \{+,-\}}\Big[\overline{A_{+,+}(k)}+\overline{A_{+,-}(k)}\Big]\mathcal{Q}_{\widehat{\Lambda_+},+,s_1,s_2,s_3,s_4}(F_1,F_2,F_3,F_4)(k)\,+ \\
&\quad\sum_{s_1,...,s_4\in \{+,-\}}\overline{A_{+,-}(k)}\big[\mathcal{Q}_{\widehat{\Lambda_+},-,s_1,s_2,s_3,s_4}(F_1,F_2,F_3,F_4)(k)-\mathcal{Q}_{\widehat{\Lambda_+},+,s_1,s_2,s_3,s_4}(F_1,F_2,F_3,F_4)(k)\big],
\end{aligned}
\end{align}
and see that the term inside the bracket of the second sum in the right-hand side  vanishes at $k=0$; so that taking the derivative does not introduce any singular term (see \eqref{A+der}). Moreover, according to  \eqref{A-der}, it is clear that $I_-$ can be written in a similar manner. Thus, in what follows we focus solely on the term $Q_{m,+}$ and, for simplicity, set $\ell_j=+$, $j=1,...,4$ (the remaining cases are handled analogously). Adopting the notation $\mathcal{Q}^{+,+,+,+}_{\widehat{\Lambda_+},s_0,s_1,s_2,s_3,s_4}:=\mathcal{Q}_{\widehat{\Lambda_+},s_0,s_1,s_2,s_3,s_4}$, it turns out that $\partial_kI_+$ may be expressed as a linear combination of $\mathcal{Q}_{\widehat{\Lambda_+},-,...}$ and $\mathcal{Q}_{\widehat{\Lambda_+},+,...}$ as well as their derivatives with bounded coefficients. 
Therefore, using \eqref{uniform-estimate-T-R} we  estimate
\begin{align}
\| 
\partial_k Q_{m,+}\|_{L^2_k} &\lesssim Q_{m,1} +Q_{m,2},   
\end{align}
where 
\begin{align*}
Q_{m,1}&:=\sum_{s_0\in \{+,-\}}\bigg\|\int_0^\infty\mathcal{Q}_{\widehat{\Lambda_+},s_0,s_1,s_2,s_3,s_4}(F_1,F_2,F_3,F_4)\tau_m(s)ds\bigg\|_{L^2_k},
\\
Q_{m,2}&:=\sum_{s_0\in \{+,-\}}\bigg\| 
\int_0^\infty\partial_k\mathcal{Q}_{\widehat{\Lambda_+},s_0,s_1,s_2,s_3,s_4}(F_1,F_2,F_3,F_4)\tau_m(s)ds\bigg\|_{L^2_k}.    
\end{align*}
The goal is to bound $Q_{m,1}$ and $Q_{m,2}$ by the right-hand side of \eqref{reductionestdec}.
We will only consider the case $s_0=+$ since $s_0=-$ is identical.

Using Lemma \ref{lem:transference} and the linear decay from the bootstrap assumptions, we obtain
\begin{align}\label{Qm1}
\nonumber Q_{m,1} &\lesssim    \int_{2^{m-1}}^{2^{m+1}} \|\Lambda_+w_1(s,s_1\cdot)w_2(s,s_2\cdot)w_3(s,s_3\cdot)w_4(s,s_4\cdot)\|_{L^2_x}ds\\
\nonumber &\lesssim    \int_{2^{m-1}}^{2^{m+1}}  \|w_1(s,s_1\cdot)\|_{L^2_x}\|w_2(s,s_2\cdot)\|_{L^{\infty}_x}\|w_3(s,s_3\cdot)\|_{L^{\infty}_x}\|w_4(s,s_4\cdot)\|_{L^{\infty}_x}ds
\\
& \lesssim 2^m \cdot \varepsilon_0 \cdot\big( \varepsilon_0^\alpha 2^{-\frac{m}{2} + 2\delta m}\big)^3
 \lesssim \varepsilon_0^{1+3\al} 2^{-\frac{m}{2}+6\delta m}
\end{align}
which is more than sufficient for \eqref{reductionest}.

Next, we estimate $Q_{m,2}$ using the representation formula in Lemma \ref{partial-k-I}.
Indeed, we write 
\begin{align}
 Q_{m,2}\lesssim  \sum_{j=0}^5Q^{(j)}_{m,2} 
\end{align}
where 
\begin{align*}
Q^{(0)}_{m,2}&:=  \left\|\int^t_0\tau_m(s) s\langle k\rangle^{-1}\mathcal{Q}_{\zeta\widehat{\Lambda_+},+,s_1,s_2,s_3,s_4}(F_1,F_2,F_3,F_4)ds\right\|_{L^2_k}, \\
Q^{(1)}_{m,2}&:=  \left\|\int^t_0\tau_m(s) \langle k\rangle^{-1}\mathcal{Q}_{\widehat{\Lambda_+},+,s_1,s_2,s_3,s_4}(\langle \cdot \rangle \partial_{\xi_1}F_1,F_2,F_3,F_4)ds\right\|_{L^2_k},\\
Q^{(2)}_{m,2}&:=  \left\|\int^t_0\tau_m(s) \langle k\rangle^{-1}\mathcal{Q}_{\widehat{\Lambda_+},+,s_1,s_2,s_3,s_4}(F_1,\langle \cdot \rangle \partial_{\xi_2}F_2,F_3,F_4)ds\right\|_{L^2_k}, \\
Q^{(3)}_{m,2}&:= \left\| \int^t_0\tau_m(s) \langle k\rangle^{-1}\mathcal{Q}_{\widehat{\Lambda_+},+,s_1,s_2,s_3,s_4}(F_1,F_2,\langle \cdot \rangle \partial_{\xi_3}F_3,F_4)ds\right\|_{L^2_k}, \\
Q^{(4)}_{m,2}&:=  \left\|\int^t_0\tau_m(s)\langle k\rangle^{-1}\mathcal{Q}_{\widehat{\Lambda_+},+,s_1,s_2,s_3,s_4}(F_1,F_2,F_3,\langle \cdot \rangle \partial_{\xi_4}F_4)ds\right\|_{L^2_k},
\\
Q^{(5)}_{m,2}&:=  \bigg\|\langle k\rangle^{-1}\int^t_0\tau_m(s) \iiiint\frac{\Upsilon e^{is\Upsilon}}{\prod_{j=1}^4\langle \xi_j\rangle} F_1(\xi_1)F_2(\xi_2)F_3(\xi_3)F_4(\xi_4)\partial_\zeta \widehat{\Lambda_+}(\zeta)d\xi_{1234}ds\bigg\|_{L^2_k}.
\end{align*}

Since $\Theta_+$ is  a Schwartz function, 
it follows from Lemma \ref{lem:FT-varLambda} and the (global) decay estimate
for the $w_j$'s in \eqref{wjLinfty},
that
\begin{align}\label{Q0m2}
 \nonumber Q^{(0)}_{m,2}&\lesssim  \int^{2^{m+1}}_{2^{m-1}} s\|\Theta_+ \, w_1(s,s_1\cdot)w_2(s,s_2\cdot)w_3(s,s_3\cdot)w_4(s,s_4\cdot)\|_{L^2_x} ds
 \\
\nonumber 
  & \lesssim  \int^{2^{m+1}}_{2^{m-1}} s 
  \| w_1(s,\cdot)\|_{L^{\infty}_x}\| w_2(s,\cdot)\|_{L^{\infty}_x} \| w_3(s,\cdot)\|_{L^{\infty}_x}
  \|w_4(s,\cdot)\|_{L^{\infty}_x}
\\
&\lesssim 2^{2m} \cdot \big(\varepsilon_0^\alpha 2^{-\frac{m}{2} + 2 \delta m} \big)^4
\lesssim \varepsilon_0^{4\al} 2^{8\delta m} 
\end{align}
which is sufficient for \eqref{reductionest}.

The terms $Q^{(j)}_{m,2}$, $j=1,...,4$ are similar; thus, we focus only on the first term.
Note that, with 
\begin{equation}
W_j(s,x) 
  := \lr{D}^{-1}e^{is\lr{D}}\widehat{\mathcal{F}}^{-1}(\lr{\xi_j}
  \partial_{\xi_j} (A_{+,s_j}\wt{f}(s)))(x),
\end{equation}  
we have
\begin{align}\label{Wj-in-L2}
\|W_j(s,\cdot)\|_{L^2_x}\lesssim \|\wt{f}(s)\|_{H^1_k} \lesssim 
  \min \big(\varepsilon_0^2 s^{3/2} , s^{1/2}\big)
  + \varepsilon_0^\alpha \lr{s}^{\delta_1} + \varepsilon_0,
\end{align}
in view of $\wt{f} = \wt{h} - \wt{g}$ and \eqref{nablag-in-H1} in Lemma \ref{lem:bad-part-bds}.

Using Lemmas \ref{lem:FT-varLambda} and \ref{lem:transference},
the bound \eqref{Wj-in-L2} and the decay \eqref{wjLinfty}, we have 
\begin{align}\label{Q12m}
\nonumber 
Q^{(1)}_{m,2}
  &\lesssim    \int_{2^{m-1}}^{2^{m+1}} 
  \|\Lambda_+W_1(s,s_1\cdot)w_2(s,s_2\cdot)w_3(s,s_3\cdot)w_4(s,s_4\cdot)\|_{L^2_x}ds\\
\nonumber &\lesssim  \int_{2^{m-1}}^{2^{m+1}} \|W_1(s,s_1\cdot)\|_{L^2_x}\|w_2(s,s_2\cdot)\|_{L^{\infty}_x}\|w_3(s,s_3\cdot)\|_{L^{\infty}_x}\|w_4(s,s_4\cdot)\|_{L^{\infty}_x}ds
\\
&\lesssim 2^m \cdot 
  \big[ \min \big( \varepsilon_0^2 2^{3m/2} , 2^{m/2}\big) 
  + \varepsilon_0^\alpha 2^{\delta_1 m} + \varepsilon_0 \big]
  \cdot \big(\varepsilon_0^{\alpha} 2^{-m/2 + 2\delta m} \big)^3
\\
& \lesssim 
  \varepsilon_0^{3\alpha} 2^{6\delta m}.
\end{align}

To control $Q^{(5)}_{m,2}$, we use integration by parts in
time to rewrite the integral inside the norm as 
\begin{align*}
\int^{2^{m+1}}_{2^{m-1}} \tau_m(s)
&\iiiint\frac{\Upsilon e^{is\Upsilon}}{\prod_{j=1}^4\langle \xi_j\rangle}
  F_1(s,\xi_1)F_2(s,\xi_2)F_3(s,\xi_3)F_4(s,\xi_4)\partial_\zeta \widehat{\Lambda_+}(\zeta)\,d\xi_{1234}ds
\\
& 
= i \int^{2^{m+1}}_{2^{m-1}}
  \iiiint\frac{e^{is\Upsilon}}{\prod_{j=1}^4\lr{\xi_j}} F_1(s,\xi_1)F_2(s,\xi_2) F_3(s,\xi_3) F_4(s,\xi_4) \partial_{\zeta} \widehat{\Lambda_+}(\zeta)\,d\xi_{1234} 
  \, \tau_m'(s) ds
\\
&\quad + i\int^{2^{m+1}}_{2^{m-1}}\iiiint\frac{e^{is\Upsilon}}{\prod_{j=1}^4\lr{\xi_j}}
  \partial_s\big[ F_1(s,\xi_1)F_2(s,\xi_2) F_3(s,\xi_3) F_4(s,\xi_4) \big] 
  \partial_{\zeta} \widehat{\Lambda_+}(\zeta)\,d\xi_{1234}
  \tau_m(s) ds
  \\ &\quad + \quad \{ \mbox{boundary terms} \}
\\
& := N_1(2^m,k) + N_2(2^m,k) + \{ \mbox{boundary terms} \},
\end{align*}
where we disregard the ``boundary terms'' (that only appear for $m=0$ and $L+1$) since they 
can be treated similarly to $N_1$ terms.
It will suffice to control $\lr{\cdot}^{-1}N_j$, $j=1,2$ in $L^2_k$ to obtain the desired bound for $Q^{(5)}_{m,2}    
$. 
Using the relation 
\begin{equation}
\partial_{\zeta}\widehat{\Lambda_+}  
  (\zeta)=-s_j\partial_{\xi_j}\widehat{\Lambda_+}  (\zeta),\quad j=1,...,4, 
\end{equation}
and via integration by parts in one frequency variable, 
say, $\xi_1$, and using Lemma \ref{lem:quadriform} , we further decompose $N_1$ as follows:
\begin{align*}
N_1(2^m,k) & = - is_1 \int^{2^{m+1}}_{2^{m-1}} \tau_m'(s) \iiiint\frac{\xi_1e^{is\Upsilon}}{\lr{\xi_1}^2\prod_{j=1}^4\lr{\xi_j}}F_1(s,\xi_1)F_2(s,\xi_2) F_3(s,\xi_3) F_4(s,\xi_4) \widehat{\Lambda_+}(\zeta)\,d\xi_{1234} \, ds 
\\
&\quad - s_1 \int^{2^{m+1}}_{2^{m-1}} \tau_m'(s) s\iiiint\frac{\xi_1e^{is\Upsilon}}{\lr{\xi_1}\prod_{j=1}^4\lr{\xi_j}}F_1(s,\xi_1)F_2(s,\xi_2) F_3(s,\xi_3) F_4(s,\xi_4) \widehat{\Lambda_+}(\zeta)\,d\xi_{1234} \, ds
\\
&\quad + is_1
  \int^{2^{m+1}}_{2^{m-1}} \tau_m'(s) \iiiint\frac{e^{is\Upsilon}}{\prod_{j=1}^4\lr{\xi_j}}\partial_{\xi_1}F_1(s,\xi_1)F_2(s,\xi_2) F_3(s,\xi_3) F_4(s,\xi_4)  \widehat{\Lambda_+}(\zeta)\,d\xi_{1234} \, ds 
\\
& =: N_1^{(1)}(2^{m},k)+N_1^{(2)}(2^{m},k)+N_1^{(3)}(2^{m},k).
\end{align*}
Thus, arguing as before, using $\tau_m' \lesssim 2^{-m}$, we find
\begin{align}
\nonumber\|\langle \cdot\rangle^{-1}N_1^{(1)}(2^{m},\cdot)\|_{L^2_k}
 & \lesssim 
 \sup_{s\in[2^{m-1},2^{m+1}]}
 \bigg\|\langle k\rangle^{-1}\mathcal{Q}_{\widehat{\Lambda_+},+,s_1,s_2,s_3,s_4}\bigg(\frac{\cdot}{\lr{\cdot}^2}F_1(s),F_2(s),F_3(s),F_4(s)\bigg)\bigg\|_{L^2_k}
 \\
\nonumber 
& \lesssim \sup_{s\in[2^{m-1},2^{m+1}]} \|\wt{f}(s)\|_{L^2_k}
  \|w_2(s,\cdot)\|_{L^{\infty}_x} \|w_3(s,\cdot)\|_{L^{\infty}_x}
  \|w_4(s,\cdot)\|_{L^{\infty}_x}
  \\
 \label{N11} 
 & \lesssim \varepsilon_0^{1+3\alpha}2^{-\frac{3m}{2}+6\delta m} . 
\end{align}
It can immediately be deduced from \eqref{N11} that the bound for $N_1^{(2)}$ is also acceptable
as it has only one extra factor of $2^m$.
Regarding the last term, since $\wt{f}(s,0)=0$, and using  \eqref{wjLinfty} and \eqref{Wj-in-L2},
we have 
\begin{equation}\label{K13}
\begin{aligned}
\|\langle \cdot\rangle^{-1}N_1^{(3)}(2^{m},\cdot)\|_{L^2_k}
   &\lesssim 
   \sup_{s\in[2^{m-1},2^{m+1}]} \big\|\langle k\rangle^{-1}\mathcal{Q}_{\widehat{\Lambda_+},+,s_1,s_2,s_3,s_4}\big(\partial_{\xi_1}F_1(s),F_2(s),F_3(s),F_4(s)\big)\big\|_{L^2_k}
   \\
\nonumber 
& \lesssim \sup_{s\in[2^{m-1},2^{m+1}]} \|\wt{f}(s,\cdot)\|_{H^1_k}
  \|w_2(s)\|_{L^{\infty}_x}\|w_3(s)\|_{L^{\infty}_x}\|w_4(s)\|_{L^{\infty}_x}
  \\
&\lesssim \big[ \min \big(\varepsilon_0^2 2^{3m/2} , 2^{m/2}\big) 
  + \varepsilon_0^\alpha 2^{\delta_1 m} + \varepsilon_0 \big]
 \cdot \varepsilon_0^{3\al} 2^{-3m/2+6m\delta}\lesssim \varepsilon_0^{3\al}2^{6m\delta}.
\end{aligned}
\end{equation}
These bounds are all acceptable according to \eqref{reductionestdec}.

For the remaining term $N_2(2^{m},k)$, we first observe that the $\partial_s$-derivative inside the integral can be distributed over each of the functions $F_j$ using the product rule, resulting in the sum of four similar integrals.  Thus, it will suffice to analyze only one of those. Using integration by parts in $s$ and Lemma \ref{lem:FT-varLambda}, we write
\begin{align*}
\nonumber N_2^{(1)}(2^m,k)&=i\int_{2^{m-1}}^{2^{m+1}}\iiiint\frac{e^{is\Upsilon}}{\prod_{j=1}^4\lr{\xi_j}}\partial_sF_1(s,\xi_1)F_2(s,\xi_2) F_3(s,\xi_3) F_4(s,\xi_4)\partial_{\zeta} \widehat{\Lambda_+}(\zeta)\,d\xi_{1234} \tau_m(s) ds
\\
\nonumber &= - s_2\int_{2^{m-1}}^{2^{m+1}}s\iiiint \frac{\xi_2}{\lr{\xi_2}}\frac{e^{is\Upsilon}}{\prod_{j=1}^4\lr{\xi_j}}\partial_sF_1(s,\xi_1)F_2(s,\xi_2) F_3(s,\xi_3) F_4(s,\xi_4) \widehat{\Lambda_+}(\zeta)\,d\xi_{1234}\tau_m(s)ds
\\
\nonumber &\quad - is_2\int_{2^{m-1}}^{2^{m+1}}\iiiint \frac{\xi_2}{\lr{\xi_2}^2}\frac{e^{is\Upsilon}}{\prod_{j=1}^4\lr{\xi_j}}\partial_sF_1(s,\xi_1)F_2(s,\xi_2) F_3(s,\xi_3) F_4(s,\xi_4) \widehat{\Lambda_+}(\zeta)\,d\xi_{1234}\tau_m(s)ds
\\
\nonumber&\quad +is_2\int_{2^{m-1}}^{2^{m+1}}\iiiint \frac{e^{is\Upsilon}}{\prod_{j=1}^4\lr{\xi_j}}\partial_sF_1(s,\xi_1)\partial_{\xi_2}F_2(s,\xi_2) F_3(s,\xi_3) F_4(s,\xi_4) \widehat{\Lambda_+}(\zeta)\,d\xi_{1234}\tau_m(s)ds
\\
\nonumber &=-s_2\int_{2^{m-1}}^{2^{m+1}}s\mathcal{Q}_{\widehat{\Lambda_+},+,s_1,s_2,s_3,s_4}\big(\partial_{s}F_1(s,\xi_1),\cdot\lr{\cdot}^{-1}F_2(s,\xi_2),F_3(s,\xi_3),F_4(s,\xi_4)\big)\tau_m(s)ds
\\
&\quad +is_2\int_{2^{m-1}}^{2^{m+1}}\mathcal{Q}_{\widehat{\Lambda_+},+,s_1,s_2,s_3,s_4}\big(\partial_{s}F_1(s,\xi_1),\cdot\lr{\cdot}^{-2}F_2(s,\xi_2),F_3(s,\xi_3),F_4(s,\xi_4)\big)\tau_m(s)ds
\\
\nonumber &\quad -is_2\int_{2^{m-1}}^{2^{m+1}}\mathcal{Q}_{\widehat{\Lambda_+},+,s_1,s_2,s_3,s_4}\big(\partial_{s}F_1(s,\xi_1),\partial_{\xi_2}F_2(s,\xi_2),F_3(s,\xi_3),F_4(s,\xi_4)\big)\tau_m(s)ds
\\
&:= \mathrm{M}_1(2^m,k)+\mathrm{M}_2(2^m,k)+\mathrm{M}_3(2^m,k).
\end{align*}
In what follows, we use the bound 
\begin{align}\label{partial-sf}
\|\partial_s\wt{f}(s,\cdot)\|_{L^2_k} \lesssim \Rho(s) 
 +\varepsilon_0^{2\al}\lr{s}^{-2+2\delta_2}
 +\varepsilon_0^{1+3\al}\lr{s}^{-3/2+3\delta} 
  :=\varkappa(s)
\end{align}
which can be derived directly from \eqref{f-evolution} 
and our bootstrap assumptions \eqref{w-assumptions-bootstrap-pde}. 
Notice that $s\varkappa(s) \leq 1$.
Next, we define 
\begin{align*}
 \mathfrak{W}(s,x)&=\lr{D}^{-1}e^{is\lr{D}}\widehat{\mathcal{F}}^{-1}(A_{+,s_1}\partial_s\wt{f})(x)
 \\[1.5ex]
 \mathfrak{W}_j(s,x)&=\lr{D}^{-1}e^{is\lr{D}}\widehat{\mathcal{F}}^{-1}(A_{+,s_j}
 \cdot\lr{\cdot}^{-1}\wt{f}(s,\cdot))(x)
 \\[1.5ex]
 \mathfrak{N}_j(s,x)&=\lr{D}^{-1}e^{is\lr{D}}\widehat{\mathcal{F}}^{-1}\big[\partial_{\xi_{j}}(A_{+,s_j}\wt{f})\big](x),\quad j=1,2,3,4.
 \end{align*} 
 By Plancherel, we have
\begin{equation} \label{W-partials-f}
\|\mathfrak{W}(s,\cdot)\|_{L^2_x}\lesssim \|\partial_s\wt{f}(s)\|_{L^2_k}\lesssim \varkappa(s).
\end{equation}
Similarly to \eqref{wjLinfty} (see Lemma \ref{lem:BoundednessW}), we have
\begin{equation}\label{WjLinfty}
\big\|\mathfrak{W}_j(s,\cdot)\|_{L^\infty}=\|\lr{D}^{-1}e^{is\lr{D}}\widehat{\mathcal{F}}^{-1}[A_{+,s_j}\cdot\lr{\cdot}^{-1}\wt{f}(s,\cdot) ]\big\|_{L^\infty}\lesssim \varepsilon_0^\al \lr{s}^{-1/2+2\delta},
\end{equation}
and by Sobolev embedding and \eqref{Wj-in-L2}
\begin{align}\label{NjLinfty}
\begin{split}
\big\|\mathfrak{N}_j(s,\cdot)\|_{L^\infty}=\|\lr{D}^{-1}e^{is\lr{D}}\widehat{\mathcal{F}}^{-1}[\partial_{\xi_j}(A_{+,s_j}\wt{f}(s))]\big\|_{L^\infty}
\\ \lesssim  \|\wt{f}(s)\|_{H^1_k}\lesssim 
  \min\big(\varepsilon_0^2s^{3/2},s^{1/2}\big) + \varepsilon_0^\alpha \lr{s}^{\delta_1}.
\end{split}
\end{align}
Thus, using \eqref{Wj-in-L2} and Lemma \ref{lem:transference} we find 
\begin{align}\label{L1}
\begin{aligned}
 \|\mathrm{M}_1(2^m,\cdot)\|_{L^2_k}&\lesssim  2^m\int_{2^{m-1}}^{2^{m+1}}\left\|\Lambda_{+}\mathfrak{W}(s,s_1\cdot)\mathfrak{W}_2(s,s_2\cdot)w_3(s,s_3\cdot)w_4(s,s_4\cdot)\right\|_{L^2_x}ds\\
&\lesssim 2^m\int_{2^{m-1}}^{2^{m+1}}\|\mathfrak{W}(s,s_1\cdot)\|_{L^2_x}\|\mathfrak{W}_2(s,s_2\cdot)\|_{L^{\infty}_x}\|w_3(s,s_3\cdot)\|_{L^{\infty}_x}\|w_4(s,s_4\cdot)\|_{L^{\infty}_x}ds\\
 &\lesssim 2^m\int_{2^{m-1}}^{2^{m+1}}\|\partial_s\wt{f}(s,\cdot)\|_{L^2_k}\|\mathfrak{W}_2(s,\cdot)\|_{L^{\infty}_x}\|w_3(s,\cdot)\|_{L^{\infty}_x}\|w_4(s,\cdot)\|_{L^{\infty}_x}ds
 \\
&\lesssim \varepsilon_0^{3\al}2^m \varkappa(2^m) 2^{-\frac{m}{2}+4\delta m}.
\end{aligned}
\end{align}
The above bound is acceptable since $2^m\varkappa(2^m)$ is uniformly bounded.
It is easily seen that $\mathrm{M}_2(2^m,k)$ satisfies a stronger bound. 
Regarding $\mathrm{M}_3(2^m,k)$, we have 
\begin{align}\label{M3}
\nonumber \|\lr{\cdot}^{-1}\mathrm{M}_3(2^m,\cdot)\|_{L^2_k}&\lesssim \int_{2^{m-1}}^{2^{m+1}}\|\Lambda_{+}\mathfrak{W}(s,s_1\cdot)\mathfrak{N}_2(s,s_2\cdot)w_3(s,s_3\cdot)w_4(s,s_4\cdot)\|_{L^2_x}ds\\
\nonumber&\lesssim \int_{2^{m-1}}^{2^{m+1}}\|\mathfrak{W}(s,s_1\cdot)\|_{L^2_x}\|\mathfrak{N}_2(s,s_2\cdot)\|_{L^\infty_x}\|w_3(s,s_3\cdot)\|_{L^\infty_x}\|w_4(s,s_4\cdot)\|_{L^\infty_x}ds\\
\nonumber &\lesssim \int_{2^{m-1}}^{2^{m+1}}\|\partial_s\wt{f}(s,\cdot)\|_{L^2_k}\|\wt{f}(s,s_2\cdot)\|_{H^1_k}\|w_3(s,s_3\cdot)\|_{L^\infty_x}\|w_4(s,s_4\cdot)\|_{L^\infty_x}ds\\
&\lesssim  \varepsilon_0^{2\al}2^m\varkappa(2^m)2^{-m+4\delta m}
 \big[\min\big(\varepsilon_0^22^{3m/2},2^{m/2}\big) + \varepsilon_0^\alpha 2^{\delta_1 m} \big]
\\
& \lesssim \varepsilon_0^{3\al}2^{4\delta m}
\end{align}
where we have used \eqref{wjLinfty}, \eqref{Wj-in-L2} and \eqref{NjLinfty}.
Once again, these bounds are sufficient for our purposes, as seen in \eqref{reductionestdec}.

\medskip
\subsubsection{Control of $Q_{m,R,\ell}$, $\ell=1,2$}\label{QmRl} 
To conclude the proof, we estimate the regular parts in \eqref{Qm-decomposed}, 
which are expected to have better bounds.
Recall, 
\begin{align*}
I_{R,\ell}(s,k)=\sum_{\ell_1,\ell_2,\ell_3,\ell_4\in \{+,-\}} \iiiint\frac{e^{is\Upsilon}}{\prod_{j=1}^4\lr{\xi_j}} \wt{f}^{(\ell_1)}(s,\xi_1)\wt{f}^{(\ell_2)}(s,\xi_2)\wt{f}^{(\ell_3)}(s,\xi_3)\wt{f}^{(\ell_4)}(s,\xi_4)\mu^{R,\ell}_{\ell_1\ell_2\ell_3\ell_4}d\xi_{1234}
\end{align*}
where $\mu^{R,\ell}_{\ell_1\ell_2\ell_3\ell_4}$ is defined in \eqref{mur1} and \eqref{mur2}. As before, for simplicity's sake, we set $\ell_j=+$, $j=1,2,3,4$, and simply write $\mu^{R,2}$ without the signs. 
We start with $Q_{m,R,1}$. By the nature of the argument we will employ here, it  will suffice to consider only two cases: $E_0=R$ and $E_j\in \{S,R\}$ for $j\neq 0$, and $E_0=S$, $E_1=R$, and $E_j\in \{S,R\}$ for $j\neq 0,1$. Thus, we define
\begin{equation*}
I_{R,1,\ell}(s,k)=\iiiint\frac{e^{is\Upsilon}}{\prod_{j=1}^4\lr{\xi_j}} 
  \wt{f}
  (s,\xi_1) \wt{f}(s,\xi_2)\wt{f}(s,\xi_3)\wt{f}(s,\xi_4)\mu^{R,1}_\ell d\xi_{1234},\quad \ell=1,2,   
\end{equation*}
where 
\begin{equation*}
 \mu_{1}^{R,1}(k,\xi_1,\xi_2,\xi_3,\xi_4)=\int \overline{\psi_R(x,k)}\psi_{E_1}(x,\xi_1)\psi_{E_2}(x,\xi_2) \psi_{E_3}(x,\xi_3)\psi_{E_4}(x,\xi_4)dx  
\end{equation*}
and 
\begin{equation*}
 \mu_{2}^{R,1}(k,\xi_1,\xi_2,\xi_3,\xi_4)=\int \overline{\psi_S(x,k)}\psi_{R}(x,\xi_1)\psi_{E_2}(x,\xi_2) \psi_{E_3}(x,\xi_3)\psi_{E_4}(x,\xi_4)dx. 
\end{equation*}
To estimate $\partial_kI_{R,1,\ell}$ in $L^2_k$, we will control the following  terms:
\begin{align}
J_\ell(s,k)= is\frac{k}{\lr{k}}
  \iiiint \frac{e^{is\Upsilon}}{\prod_{j=1}^4\lr{\xi_j}} \wt{f}(s,\xi_1) 
    \wt{f}(s,\xi_2) \wt{f}(s,\xi_3) \wt{f}(s,\xi_4)\mu^{R,1}_\ell d\xi_{1234}
\end{align}
and 
\begin{align}
H_\ell(s,k)= \iiiint   \frac{e^{is\Upsilon}}{\prod_{j=1}^4\lr{\xi_j}}
  \wt{f}(s,\xi_1) \wt{f}(s,\xi_2) \wt{f}(s,\xi_3) \wt{f}(s,\xi_4)
  \partial_k\mu^{R,1}_\ell d\xi_{1234}.
\end{align}
First, we rewrite $J_1$ using \eqref{phi-decomp} as follows: 
\begin{align*}
J_1(s,k)=-(2\pi)^{-2} is\frac{k}{\lr{k}}e^{-is\lr{k}}
  \int\overline{\psi_R(x,k)}U_{E_1}(s,x)U_{E_2}(s,x)   U_{E_3}(s,x)U_{E_4}(s,x)dx
\end{align*}
 where 
 \begin{equation*}
 U_{E_j} 
 (s,x) = \big(L^{-1}e^{
 isL}f
 (s,\cdot)\big)_{E_j}(s,x),
 \end{equation*} the notation $\phi_E$ being defined in \eqref{phi-decomp}.
 By the local decay bootstrap assumptions \eqref{w-assumptions-bootstrap-pde}, we have 
\begin{align}\label{J1}
\begin{aligned}
\nonumber&\bigg\|\int_0^\infty J_1(s,\cdot)\tau_m(s)ds\bigg\|_{L^2_k}
\\
\nonumber&\lesssim \int_{2^{m-1}}^{2^{m+1}}s\|\lr{\cdot}^{-11}U_{E_1}(s,\cdot)U_{E_2}(s,\cdot)U_{E_3}(s,\cdot)U_{E_4}(s,\cdot)\|_{L^2_x}ds\\
\nonumber &\lesssim 2^m\int_{2^{m-1}}^{2^{m+1}}\|\lr{\cdot}^{-2}U_{E_1}(s,\cdot)\|_{L^{\infty}_x}\|\lr{\cdot}^{-2}U_{E_2}(s,\cdot)\|_{L^{\infty}_x}\|\lr{\cdot}^{-2}U_{E_3}(s,\cdot)\|_{L^{\infty}_x}\|\lr{\cdot}^{-2}U_{E_4}(s,\cdot)\|_{L^{\infty}_x}ds\\
&\lesssim \varepsilon_0^{4\al} 2^{-2m+4\delta_2m}.
\end{aligned}
\end{align}
This estimate is stronger than the desired bound \eqref{reductionestdec}. Similarly, we write 
\begin{align*}
H_1(s,k)=(2\pi)^{-2} e^{-is\lr{k}}\int\partial_k\overline{\psi_R(x,k)}U_{E_1}(s,x)U_{E_2}(s,x) 
  U_{E_3}(s,x)U_{E_4}(s,x)dx,
\end{align*}
and using the local decay in the bootstrap assumptions, we find that
\begin{equation}\label{H1}
\begin{aligned}
\nonumber\bigg\|\int_0^\infty &H_1(s,\cdot)\tau_m(s)ds\bigg\|_{L^2_k}\\
&\lesssim \int_{2^{m-1}}^{2^{m+1}}\|\lr{\cdot}^{-10}U_{E_1}(s,\cdot)U_{E_2}(s,\cdot)U_{E_3}(s,\cdot)U_{E_4}(s,\cdot)\|_{L^2_x}ds\\
\nonumber &\lesssim \int_{2^{m-1}}^{2^{m+1}}\|\lr{\cdot}^{-2}U_{E_1}(s,\cdot)\|_{L^{\infty}_x}\|\lr{\cdot}^{-2}U_{E_2}(s,\cdot)\|_{L^{\infty}_x}\|\lr{\cdot}^{-2}U_{E_3}(s,\cdot)\|_{L^{\infty}_x}\|\lr{\cdot}^{-2}U_{E_4}(s,\cdot)\|_{L^{\infty}_x}ds\\
&\lesssim \varepsilon_0^{4\al}2^{-3m+4\delta_2m}
\end{aligned}
\end{equation}
which is acceptable.

For the remaining terms involving  $H_2$ and $J_2$, we argue as before using the expressions
\begin{align*}
 J_2(s,k)=-(2\pi)^{-2} is\frac{k}{\lr{k}}
  e^{-is\lr{k}}\int\overline{\psi_S(x,k)}U_{R}(s,x)U_{E_2}(s,x)   U_{E_3}(s,x)U_{E_4}(s,x)dx,   
\end{align*}
and 
\begin{align*}
H_2(s,k)=(2\pi)^{-2} e^{-is\lr{k}}\int\partial_k\overline{\psi_S(x,k)}U_{R}(s,x)U_{E_2}(s,x)   U_{E_3}(s,x)U_{E_4}(s,x)dx.
\end{align*}
Since the distorted Fourier transform (defined with singular kernel $\psi_S$) is bounded in $L^2_k$, we use \eqref{local-l-infty-decay-reg} to obtain 
\begin{equation}\label{J2}
\begin{aligned}
\nonumber &\bigg\|\int_0^\infty J_2(s,\cdot)\tau_m(s)ds\bigg\|_{L^2_k}
\\
\nonumber&\lesssim 2^m\int_{2^{m-1}}^{2^{m+1}}\|U_{R}(s,\cdot)U_{E_2}(s,\cdot)U_{E_3}(s,\cdot)U_{E_4}(s,\cdot)\|_{L^2_x}ds\\
 \nonumber &\lesssim 2^m\int_{2^{m-1}}^{2^{m+1}}\|\lr{\cdot}^{10}U_{R}(s,\cdot)\|_{L^{\infty}_x}\|\lr{\cdot}^{-2}U_{E_2}(s,\cdot)\|_{L^{\infty}_x}\|\lr{\cdot}^{-2}U_{E_3}(s,\cdot)\|_{L^{\infty}_x}\|\lr{\cdot}^{-2}U_{E_4}(s,\cdot)\|_{L^{\infty}_x}ds\\
\nonumber &\lesssim \int_{2^{m-1}}^{2^{m+1}}\|\wt{f}\|_{H^1_k}\|\lr{\cdot}^{-2}U_{E_2}(s,\cdot)\|_{L^{\infty}_x}\|\lr{\cdot}^{-2}U_{E_3}(s,\cdot)\|_{L^{\infty}_x}\|\lr{\cdot}^{-2}U_{E_4}(s,\cdot)\|_{L^{\infty}_x}ds
\\
&\lesssim \varepsilon_0^{3\al}
 \big[ \min(\varepsilon_0^2 2^{3m/2},2^{m/2}) + \varepsilon_0^\alpha2^{\delta_1m}\big] 	
 2^{-2m+3\delta_2m},
\end{aligned}
\end{equation}
which suffices.
Proceeding similarly, we can use the estimate
\begin{align}\label{partphiSL2}
\left\|\int\partial_k\overline{\psi_S(x,k)}f(x)dx\right\|_{L^2_k}\lesssim \|\lr{x}f\|_{L^2_x}    
\end{align}
and deduce from \eqref{H1} that the bound for $H_2$ is acceptable as well. 

Regarding $Q_{m,R,2}$, we use once again the structure of the spectral distribution. 
We may rewrite \eqref{mur2} using \eqref{psi-regular-part} as follows
\begin{align*}
\mu^{R,2}&(k,\xi_1,\xi_2,\xi_3,\xi_4)
\\
&=\sum\int\sigma_{s_0}(x)\overline{\psi_{S,s_0}(x,k)}\sigma_{s_1}(x)\psi_{S,s_1}(x,\xi_1)\sigma_{s_2}(x)\psi_{S,s_2}(x,\xi_2)\sigma_{s_3}(x)\psi_{S,s_3}(x,\xi_3)\sigma_{s_4}(x)\psi_{S,s_4}(x,\xi_4)dx
\end{align*}
where the sum is taken over \[s_0,s_1,s_2,s_3,s_4\in \{+,-\}\quad \mbox{with} \quad (s_0,s_1,s_2,s_3,s_4)\neq (+,+,+,+,+),(-,-,-,-,-).\]
Since  $\sigma_+\sigma_-$ has compact support 
it will suffice to consider only one combination 
of the types of terms that appear in the decomposition,
all the other contributions being similar.
More precisely, we consider the distribution 
\begin{equation}\label{mu1R2}
\mu^{R,2}_1(k,\xi_1,\xi_2,\xi_3,\xi_4)
=\int\sigma_{-}(x)\overline{\psi_{S,-}(x,k)}\sigma_{+}(x)\psi_{S,+}(x,\xi_1)\psi_{S}(x,\xi_2)\psi_{S}(x,\xi_3)\psi_{S}(x,\xi_4)dx
\end{equation}
and its corresponding term in the decomposition of $I_{R,2}$, namely,
\begin{align*}
I_{R,2,1}(s,k)&=  \iiiint\frac{e^{is\Upsilon}}{\prod_{j=1}^4\langle \xi_j\rangle} \wt{f}(s,\xi_1)\wt{f}(s,\xi_2)\wt{f}(s,\xi_3)\wt{f}(s,\xi_4)\mu^{R,2}_1d\xi_{1234}\\
&=(2\pi)^{-2}e^{-is\lr{k}}\int \sigma_-(x)\sigma_+(x)\overline{\psi_{S,-}(x,k)}U_{S,+}(s,x)U^3_{S}(s,x)dx.
\end{align*}
Here, with the notation in \eqref{phi-decomp'},  the term $U_{S,+}$ in the integrand is defined by   
\[ U_{S,+}(s,x)=\big(L^{-1}e^{is L}f(s)\big)_{S,+}(x).\]

Taking the partial derivative in $k$, we find 
\begin{align*}
(2\pi)^2\partial_k I_{R,2,1}(s,k)&:=-is\frac{k}{\lr{k}}e^{-is\lr{k}}\int \sigma_-(x)\sigma_+(x)\overline{\psi_{S,-}(x,k)}U_{S,+}(s,x)U^3_{S}(s,x)dx 
\\
&\quad +e^{-is\lr{k}}\int \sigma_-(x)\sigma_+(x)\partial_k\overline{\psi_{S,-}(x,k)}U_{S,+}(s,x)U^3_{S}(s,x)dx\\
&:=I^{(1)}_{R,2,1}(s,k)+I^{(2)}_{R,2,1}(s,k).
\end{align*}
By Remark \ref{rmk:r-s-local-decay}, \eqref{Wj-in-L2} and   the a priori assumptions, we have
\begin{align*}
\nonumber &\bigg\|\int_0^\infty I^{(1)}_{R,2,1}(s,\cdot)\tau_m(s)ds\bigg\|_{L^2_k}\\
&\lesssim 2^m\int_{2^{m-1}}^{2^{m+1}}\|\sigma_-\sigma_+U_{S,+}(s,\cdot)U_{E_2}(s,\cdot)U_{E_3}(s,\cdot)U_{E_4}(s,\cdot)\|_{L^2_x}ds\\
 &\lesssim 2^m\int_{2^{m-1}}^{2^{m+1}}\|\lr{\cdot}^{-2}U_{S,+}(s,\cdot)\|_{L^{\infty}_x}\|\lr{\cdot}^{-2}U_{E_2}(s,\cdot)\|_{L^{\infty}_x}\|\lr{\cdot}^{-2}U_{E_3}(s,\cdot)\|_{L^{\infty}_x}\|\lr{\cdot}^{-2}U_{E_4}(s,\cdot)\|_{L^{\infty}_x}ds
 \\
 &\lesssim \varepsilon_0^{3\al}
 \big[\min(\varepsilon_0^22^{3m/2},2^{m/2}) + \varepsilon_0^\alpha 2^{\delta_1m} \big]2^{-2m+3\delta_2m}
\end{align*}
which is acceptable according to \eqref{reductionestdec}. Since   \eqref{partphiSL2} is still valid if $\psi_S$ is replaced by $\psi_{S,-}$, we obtain 
\begin{align*}
\nonumber &\bigg\|\int_0^\infty I^{(2)}_{R,2,1}(s,\cdot)\tau_m(s)ds\bigg\|_{L^2_k}
\\
&\lesssim \int_{2^{m-1}}^{2^{m+1}}\|\lr{x}\sigma_-\sigma_+U_{S,+}(s,\cdot)U_{E_2}(s,\cdot)U_{E_3}(s,\cdot)U_{E_4}(s,\cdot)\|_{L^2_x}ds\\
 &\lesssim \int_{2^{m-1}}^{2^{m+1}}\|\lr{\cdot}^{-2}U_{S,+}(s,\cdot)\|_{L^{\infty}_x}\|\lr{\cdot}^{-2}U_{E_2}(s,\cdot)\|_{L^{\infty}_x}\|\lr{\cdot}^{-2}U_{E_3}(s,\cdot)\|_{L^{\infty}_x}\|\lr{\cdot}^{-2}U_{E_4}(s,\cdot)\|_{L^{\infty}_x}ds
 \\
 &\lesssim \varepsilon_0^{3\al}
  \big[\min(\varepsilon_0^22^{3m/2},2^{m/2}) + \varepsilon_0^\alpha 2^{\delta_1m} \big] 2^{-3m+3\delta_2m}.
\end{align*}
This is more than sufficient for \eqref{reductionest}
and concludes the proof of Lemma \ref{lemreduction}. $\hfill \Box$

\section{Asymptotics}\label{app:asy}
In this section we prove scattering for the continuous component 
$v$, and give a more precise upper bound on its decay (Subsection \ref{AsySecv});
we then derive precise asymptotics for the internal mode (Subsection \ref{sec:Asymp-A})
and also give a lower bound on the weighted-type norm of the `bad' component (Subsection \ref{app:growth-f}) 
which shows the optimality of the assumed bootstrap estimates.

\medskip
\subsection{Asymptotics for the continuous component}\label{AsySecv}
Recall that  the continuous component $v$ is related to the radiation 
$w$ through $w=\partial_tv+iLv$, and we have (see \eqref{duhamel-for-w})
\begin{equation}\label{Duhamel-w}
w(t)=e^{itL}w(0)+\int_0^te^{i(t-s)L}P_c[\beta(a(s)\varphi+v(s))^2+(a(s)\varphi+v(s))^4]ds.    
\end{equation}
Our interest here is in refining the decay estimate
for $v$ (without polynomial losses in $t$; compare with the linear decay estimate from the bootstrap norms),
and proving a scattering result with a polynomial rate.

\begin{proposition}\label{prop:scattering}
The field component  $w(t)$ of the solution $u(t)$ scatters 
to a linear (perturbed) Klein-Gordon wave in $H^{N-1}_x$ in the sense that 
there exists $\Phi_+\in P_c H^{N-1}_x$ such that,
\begin{align}\label{scattering}
{\|w(t) - e^{itL}\Phi_+\|}_{H^{N-1}_x} \lesssim \varepsilon_0 (1+\varepsilon_0^2 t)^{-\gamma}
\end{align}
for some fixed $\gamma >0$.
\end{proposition}

A similar scattering result holds for $v$ with the natural modification of the asymptotic limit;
indeed, since $v = \Im(L^{-1} w)$ with $\partial_t v = \Re(w)$, from \eqref{scattering} we have
\begin{align}\label{scatteringv}
{\big\|v(t) - \Im(L^{-1} e^{itL}\Phi_+ \big) \big\|}_{H^{N}_x}
 + {\big\|\partial_t v(t) - \Re( e^{itL}\Phi_+ \big) \big\|}_{H^{N-1}_x} \lesssim \varepsilon_0 (1+\varepsilon_0^2 t)^{-\gamma}.
\end{align}

\begin{proof}[Proof of Proposition \ref{prop:scattering}]
Introduce the function
\begin{equation}
\Phi_+(t) := w(0) + \int_0^t e^{-isL} P_c[\beta(a(s)\varphi+v(s))^2+(a(s)\varphi+v(s))^4]ds,
\end{equation}
so that $e^{itL}\Phi_+(t) = w(t)$.
Our goal is to prove that $\Phi_+(t)$ is Cauchy in $H^{N-1}_x$ as $t\rightarrow \infty$,
and obtain convergence to its limit, which we denote $\Phi_+$, at the claimed polynomial rate.
We split
\begin{align}\label{Phi_+}
\begin{split}
\Phi_+(t) & = w(0) + \Phi_a(t) + \Phi_v(t) + \Phi_{a,v}(t)
\\
\Phi_a(t) & := \int_0^t e^{-isL}a^2(s)\theta ds + \int_0^t a^4(s) e^{-isL} P_c(\varphi^4)ds,
  \quad \theta:=P_c(\beta\varphi^2),
\\
\Phi_v(t) & :=\int_0^t e^{-isL}P_c(\beta v^2(s)) ds+\int_0^t e^{-isL}P_c[v^{4}(s)]ds,
\\
\Phi_{a,v}(t) & := 2\int_0^t e^{-isL}a(s)P_c(\beta\varphi v(s)) ds
  + \sum_{\ell=1}^3 \int_0^t e^{-isL} a^{4-\ell}(s) P_c[\varphi_\ell v^{\ell}(s)]ds,
\end{split}
\end{align}
where we are adopting the same convention of Section \ref{sec:PDE-part1}, namely  
$\varphi_\ell = \big(\begin{array}{c}
                          \!4 \!\\ \!\ell\!
                          \end{array}
\big)\varphi^{4-\ell}$ . 
%
%
Using Lemma \ref{lem:smoothing-effect} 
we estimate
\begin{align}
\|\Phi_a(t_2)-\Phi_a(t_1)\|_{H^{N-1}_x}
 & \lesssim \|\Rho\|_{L^2_t([t_1,t_2])} + \|\Rho^2\|_{L^2_t([t_1,t_2])} \lesssim \Rho^{1/2}(t_1),
\end{align}
which is consistent with \eqref{scattering}.
Using the local decay estimate (see \eqref{w-assumptions-bootstrap-pde}) and the $H^N$ bounds,
we have
\begin{align}
\|\Phi_v(t_2) - \Phi_v(t_1)\|_{H^{N-1}_x} 
  \lesssim \left\|L^{N-1}(\beta v^2(s))\right\|_{L^{1}_xL^2_t([t_1,t_2])}
  + \int_{t_1}^{t_2} \|P_c[v^{4}(s)]\|_{H^{N-1}}ds \\
   \lesssim \varepsilon_0 \cdot \min\big(\varepsilon_0^{\alpha} t_1^{-1/2+\delta_2}, \varepsilon_0 \big),
\end{align}
which again suffices for \eqref{scattering}.
%
%
For the last term in \eqref{Phi_+}, using also the Gagliardo-Nirenberg inequality \eqref{GN-endpoint}
with $\ell = N-1$, we have  
\begin{align}
\left\|\Phi_{a,v}(t_2) - \Phi_{a,v}(t_1) \right\|_{H^{N-1}_x} 
  & \lesssim \left\|a(s)L^{N-1}(\beta \varphi v(s))\right\|_{L^{1}_xL^2_t([t_1,t_2])}
  + \sum_{\ell=1}^3 \left\|a(s)^\ell L^{N-1}(\varphi_\ell v^{4-\ell}(s))\right\|_{L^{1}_xL^2_t}
  \\
  & \lesssim \left\| \Rho^{1/2}(s) \cdot \varepsilon_0^{1-1/(2N-1)} s^{-(1-\delta_2)/(2N-1)} 
  \right\|_{L^2_s([t_1,t_2])} + \, \{\mbox{better terms}\}
  \\
  & \lesssim \varepsilon_0^{1-1/(2N-1)} t_1^{-(1-\delta_2)/(2N-1)}. 
\end{align}
Since we also have a uniform $\varepsilon_0$ bound from the a priori estimates
(see also \eqref{cond0}-\eqref{cond2}) it follows that
\begin{align}
\left\|\Phi_{a,v}(t_2) - \Phi_{a,v}(t_1) \right\|_{H^{N-1}_x} 
  & \lesssim \min\big( \varepsilon_0, \varepsilon_0^{1-1/(2N-1)} t_1^{-(1-\delta_2)/(2N-1)} \big)
  \lesssim \varepsilon_0 (1+\varepsilon_0^2t_1)^{-(1-\delta_2)/(2N-1)},
\end{align}
for $\delta_2$ small enough.

The bounds above show that $\Phi_a,\Phi_v,\Phi_{a,v}$ are Cauchy in time,
so that
\begin{align}
\Phi_+(t) \stackrel{t\rightarrow \infty}{\longrightarrow}
  w(0) + \int_0^\infty e^{-isL} P_c[\beta(a(s)\varphi+v(s))^2+(a(s)\varphi+v(s))^4]ds := \Phi_+
\end{align}
with 
\begin{equation}
{\big\| \Phi_+(t) - \Phi_+ \big\|}_{H^{N-1}_x} 
  \lesssim \varepsilon_0 (1+\varepsilon_0^2t)^{-(1-\delta_2)/(2N-1)},
\end{equation}
for $t \gg 1$. Applying the evolution $e^{itL}$ we conclude the proof of the proposition.
\end{proof}

\smallskip
Next, we observe that the asymptotic decay rate for $v(t)$ 
can be improved to almost match the optimal $t^{-1/2}$ decay rate for the Klein-Gordon flow.
Note that we are not claiming that the estimate below is sharp.

\begin{proposition}\label{improved-asmp}
The continuous component $v(t)$ of the solution $u(t)$ satisfies 
\begin{equation}\label{asymp-v}
\|v(t)\|_{L^\infty_x}\lesssim    \lr{t}^{-1/2}\log(2+\varepsilon^2_0 t) ,\quad t\geq \varepsilon_0^{-2}.
\end{equation}
\end{proposition}

\begin{proof}
We look back at the proof of Proposition \ref{prop:decay_boostrap}.
Since the linear solution $e^{itL}w(0)$ decays like $\varepsilon_0 t^{-1/2}$, 
it suffices to verify \eqref{asymp-v} for the nonlinear terms in $I_a$, $I_v$ and $I_{a,v}$ 
as defined in \eqref{Ia}, \eqref{Iv} and \eqref{Iav},
and for $I_Q$ as defined in \eqref{I_Q}.

Proceeding exactly as in the proof of Proposition \ref{prop:decay_boostrap},
we see from \eqref{Iatoref} that
\begin{align}
\left\|I_{a}(t)\right\|_{L_x^\infty} & 
  \lesssim \int_0^t \lr{t-s}^{-1/2}\big(\Rho(s) +\Rho^2(s) \big) ds 
  \lesssim\lr{t}^{-1/2} \log (2+\rho_0t). 
\end{align}
For the term $I_v$, we look back at the bound \eqref{Ivtoref} and see that it is already strong enough
${\left\|I_{v}(t)\right\|}_{L_x^\infty}  \lesssim \varepsilon_0^{2\al} \lr{t}^{-1/2}.$
For the terms $I_{a,v}^{(0)}(t)$  and  $I_{a,v}^{(\ell)}(t)$  we can also directly invoke \eqref{Iav0} and \eqref{Iav0'}.
Finally, for the pure quartic term $I_Q$ we can directly rely on \eqref{IQtoref} 
whose right-hand side is smaller than the right-hand side of \eqref{asymp-v}.
\end{proof}

\smallskip
\subsection{Asymptotics for the internal mode}\label{sec:Asymp-A}

In this part, we investigate the asymptotic behavior of $A$ defined via \eqref{a-in-A-defn}-\eqref{evolution-A-initial}. 
The main result of this section is the following:

\begin{theorem}\label{thm:asympA}
Let $A=A(t)$ be the profile of $a=a(t)$ according to \eqref{slow-variable},
with $(a,v)$ solutions of \eqref{ode-pde-system-initial} with given data as in \eqref{mainthmdata}.
Then, there exists $\psi_\infty \in \R$ (see \eqref{psiinftyformula} for the exact expression) such that 
\begin{align}\label{Asymp-A}
\left|A(t) - \exp\Big[ i \Big(\psi_\infty 
  - (c_2 - \beta_1) \tfrac{\lambda}{\Gamma}\log \big(1+\tfrac{\Gamma}{\lambda} \rho_0 t\big) \Big) \Big]
  \sqrt{\frac{\lambda}{\Gamma \cdot t}} \, \right|
  = 
  O\big( t^{-1/2-\gamma} \big), \qquad   t \gg \varepsilon_0^{-2}
\end{align}
for some small $\gamma > 0$, with $c_2:= \big(\frac{\Lambda}{2\lambda}+\frac{5c^2}{3\lambda^3}\big)$
and $\beta_1$ as given in \eqref{extracorrection1}.
\end{theorem}

The proof of this result is obtained through several lemmas. 
Recall the amplitude equation
\begin{equation}\label{evolA'}
\dot A(t)=P(t)+\frac{1}{2i\lambda}(\Lambda-i\Gamma)|A|^2(t)A(t)  +R(t),  
\end{equation}
and the renormalized amplitude from \eqref{B}:
\begin{equation}\label{AsyB}
B := A + \frac{c}{\lambda^2} \left( \frac{e^{i\lambda t}}{2} A^2 - \frac{e^{-3i\lambda t}}{6} \overline{A}^2 - e^{-i\lambda t} |A|^2 \right), \quad  c: = \int \beta \varphi^3 dx.
\end{equation}
We have shown (see Lemma \ref{lem:decay-B'})  that
\begin{align}\label{derB}
\begin{split}
\frac{d}{dt}B(t)= 
  \frac{c}{\lambda^2}\bigg(-e^{-i\lambda t}\frac{d}{dt}|A|^2(t)
  -\frac{e^{-3i\lambda t}}{6}\frac{d}{dt}\overline{A}^2(t)
  +\frac{e^{i\lambda t}}{2}\frac{d}{dt}A^2(t)\bigg)
  \\ +\frac{1}{2i\lambda}(\Lambda-i\Gamma)|A|^2(t)A(t)+R(t),
\end{split}
\end{align}
and $R=R_D+R_E+R_Q$ (see Lemma \ref{lemevolA}).
We start by isolating the cubic interactions in the equation \eqref{derB}.

\begin{lemma}  
The dynamics of $B$ can be written as 
\begin{align}\label{eq-B-dot}
 \frac{d}{dt}B(t)
=B_C(t)+B_Q(t)+B_M(t)+R_E(t)+R_Q(t)
\end{align}
where 

\setlength{\leftmargini}{2em}
\begin{itemize}

\item $B_C$ contains the cubic interactions: 
\begin{align*} B_C(t):=\frac{c}{\lambda^2} \left(e^{i\lambda t}A(t)P(t)-\frac{e^{-3i\lambda t}}{3}\bar{A}(t)\overline{P}(t)-2e^{-i\lambda t} \mathfrak{Re}(\bar{A}(t)P(t))   \right) +\frac{1}{2i\lambda}(\Lambda-i\Gamma)|A|^2(t)A(t)+R_D(t).
\end{align*}

\item $B_Q$ contains the quartic interactions:
\begin{align*}
   B_Q(t)&:=\frac{c}{2\lambda^3}|A|^2(t)\Big(\frac{1}{i}(\Lambda-i\Gamma)A^2(t)e^{i\lambda t}+\frac{1}{3i}(\Lambda+i\Gamma)\bar{A}^2(t)e^{-3i\lambda t}+2\Gamma |A|^2(t)e^{-i\lambda t}\Big) \\
   &\quad +\frac{c}{\lambda^2}\Big(A(t)R_D(t)e^{i\lambda t}-\bar{A}(t)\overline{R_D}(t)\frac{e^{-3i\lambda t}}{3}-2\,\Re (\bar{A}(t)R_D(t))e^{-i\lambda t}\Big).
\end{align*}

\item $B_M$ collects mixed interactions: 
\begin{align*}
    B_M(t)&:=\frac{c}{\lambda^2}\Big(A(t)R_E(t)e^{i\lambda t}-\bar{A}(t)\overline{R_E}(t)\frac{e^{-3i\lambda t}}{3}-2\,\Re(\bar{A}(t)R_E(t))e^{-i\lambda t}\Big)\\
&\qquad \quad+\frac{c}{\lambda^2}\Big(A(t)R_Q(t)e^{i\lambda t}-\bar{A}(t)\overline{R_Q}(t)\frac{e^{-3i\lambda t}}{3}-2\,\Re(\bar{A}(t)R_Q(t))e^{-i\lambda t}\Big),
\end{align*}
and, for all $t \geq \rho_0^{-1/2}$, satisfies
\begin{align}\label{AsyBMbound}
| B_M(t)| \lesssim \Rho(t) \cdot \varepsilon_0^{3\alpha}\lr{t}^{-1+8\delta}.
\end{align}
\end{itemize} 
In the above, $P$ is the self-interactions from \eqref{P-defn} 
and $R_D$, $R_E$ and  $R_Q$ are defined in Lemma  \ref{lemevolA}.
\end{lemma}

\begin{proof}
The proof of this lemma follows by direct calculations 
according to the definition of $B$ in Lemma \ref{lem:decay-B'}, the equation for $A$ 
in Lemma \ref{lemevolA}, and grouping the terms appropriately.
The estimate \eqref{AsyBMbound} follows from the bounds on $R_E$ and $R_Q$ in Remark \ref{rmk:controlR}.
\end{proof}

\medskip
Using the notation  $A_{-1}:=\bar{A}$, $A_{+1}:=A$,  
we rewrite the cubic term $B_C$ isolating the non-oscillatory part:
\begin{align}
\label{Bc}
\begin{aligned}
  B_C(t)&=  \sum_{\iota _j\in \{-1,+1\}}   c_{\iota_1,\iota_2,\iota_3} A_{\iota_1}A_{\iota_2}A_{\iota_3}e^{i(-1+\iota_1+\iota_2+\iota_3)\lambda t}
  \\
 &=c_1|A|^2A+    \sum_{\iota_1+\iota_2+\iota_3\neq 1}   c_{\iota_1,\iota_2,\iota_3} A_{\iota_1}A_{\iota_2}A_{\iota_3}e^{i(-1+\iota_1+\iota_2+\iota_3) \lambda t}, 
\end{aligned}    
\end{align}
where $c_{\iota_1,\iota_2,\iota_3}\in \mathbb C$ and
\begin{align}
c_1=-\frac{\Gamma}{2\lambda}-ic_2=-\frac{\Gamma}{2\lambda}-i\big(\frac{\Lambda}{2\lambda}+\frac{5c^2}{3\lambda^3}\big).
\end{align}

\begin{lemma}\label{lem:renormalization_evol}
Define the cubic correction
\begin{align}\label{AsyW}
\begin{split}
& W := B-N(A),
\\
& N(A)(s):=  \sum_{\iota_1+\iota_2+\iota_3\neq 1}   
  \frac{c_{\iota_1,\iota_2,\iota_3}}{i\lambda (-1+\iota_1+\iota_2+\iota_3)}
  A_{\iota_1}(s)A_{\iota_2}(s)A_{\iota_3}(s)e^{i(-1+\iota_1+\iota_2+\iota_3)\lambda s} .
\end{split}
\end{align}
The dynamics of $W$ is given by the equation
\begin{equation}\label{W'}
\begin{aligned}
\frac{d }{dt}W(t)&=-\frac{\Gamma}{2\lambda}|W|^2(t)W(t)-ic_2|W|^2(t)W(t)+B_M(t)+R_E(t)+R_Q(t)+O(|A|^4).
\end{aligned}
\end{equation}
\end{lemma}

\begin{proof}
First note that from \eqref{eq-B-dot} and \eqref{Bc}, we have 
\begin{align}  
 \frac{d}{dt}W =c_1|A|^2A +B_Q+B_M +R_E+R_Q-  \sum_{\iota_1+\iota_2+\iota_3\neq 1}   \frac{c_{\iota_1,\iota_2,\iota_3}}{i\lambda(-1+\iota_1+\iota_2+\iota_3)} \frac{d}{dt}(A_{\iota_1}A_{\iota_2}A_{\iota_3})e^{i(-1+\iota_1+\iota_2+\iota_3)\lambda t}
\end{align}   
Since $A=W+O(|A|^2)$, $B_Q=O(|A|^4)$ and $\dot{A}=O(A^2)$, this implies \eqref{W'}.
\end{proof}

Choose $t_0 \approx \rho_0^{-1}$. We introduce the nonlinear gauge transformation 
\begin{equation}\label{X}
X(t)=W(t) e^{ic_2\int_{t_0}^t|W(s)|^2ds}   
\end{equation}
to remove the oscillating term in \eqref{W'}.
Then, the evolution of $X$ reads 
\begin{align}\label{X'}
\frac{d}{dt}X=-\frac{\Gamma}{2\lambda}|X|^2X + (
  R_E+R_Q)e^{ic_2\int_{t_0}^t|W|^2ds} + O(\Rho^{7/4}(t)),
\end{align}
having absorbed the contribution from $B_M$ into the last remainder by using
that the right-hand side of \eqref{AsyBMbound} is $O(\Rho^{7/4}(t))$ for $\delta$ small enough.
We then look for asymptotic for $X$:

\begin{lemma}\label{lem:evolY&theta}
Set $X(t)=Y(t)e^{i\theta(t)}$, with $Y=|X|$, and $\theta\in \R$ a continuous phase with prescribed value $\theta(t_0)$.
If we define the nonlinear phase function  
\begin{align}\label{AsyTheta}
\Theta(t) = c_2\int^t_{t_0}|W(s)|^2ds-\theta(t),                                                               
 \end{align}
the evolution of the phase and the amplitude are governed by the following equations:
\begin{align}\label{theta'}
& \frac{d}{dt}\theta = Y^{-1}\Im\big[(
  R_E+R_Q)e^{i\Theta(t)}\big]+O(\Rho^{5/4}(t)),
\\
& \theta(t_0):=\theta_0=\arg(B(t_0)-N(A(t_0)))
\end{align}
and
\begin{align}\label{Y'}
\frac{d}{dt}{Y}&=-\frac{\Gamma}{2\lambda}|Y|^2Y + \Re\big[(
  R_E+R_Q)e^{i\Theta(t)}\big]+O(\Rho^{7/4}(t)),
\\
Y(t_0)&:=Y_0= \big|B(t_0)-N(A(t_0))\big|.
\end{align}
Moreover, we have
\begin{equation}\label{AsyY-A}
|Y|=|X|=|W| \approx |A| \approx \Rho^{1/2},
  \qquad W - A = O(\Rho), \qquad Y - |A| = O(\Rho), 
\end{equation}

\end{lemma}

\begin{proof}
From the definitions we immediately have $Y=|X|=|W| = |B - N(A)| \in [(1/2)|B|,2|B|]$,
in view of \eqref{AsyB} and our a priori assumptions. 
Similarly, $|Y - |A|| = ||B-N(A)|-|A|| \lesssim O(|A|^2)=O(\Rho)$. 
The equations \eqref{theta'} and \eqref{Y'} are derived from \eqref{X'} 
after taking the imaginary and real parts, respectively and using the lower 
bound $|Y|\gtrsim |A|$ to handle the remainder in \eqref{theta'}
\end{proof}

\smallskip
Next, we derive asymptotics from the equations \eqref{theta'} and \eqref{Y'}
by re-using some of the bounds from Section \ref{sec:ODE} 
and sharpening some others.

\begin{lemma}\label{lem:asymptoticsY}
Let $z(t):=Y^2(t)$, $t\geq t_0$. Then
\begin{align}\label{ODE|Y|^2}
\begin{split}
& 
\dot{z}(t) = -\frac{\Gamma}{\lambda} z^2(t) + 2
  \,\Re \big( (R_E(t)+R_Q(t) ) \overline{A}(t) \big) + O(\Rho^{9/4}(t)),\quad
\\
& z(t_0)=Y^2(t_0) := \wt z_0
\end{split}
\end{align}
where $A$ is the original profile of the internal mode's amplitude
and the terms $R_E$ and $R_Q$ are as in Lemma \ref{lemevolA}
(see also the splitting of $R_E$ in \eqref{R_E0-defn}-\eqref{R_E4-defn})

If we define for $t\geq t_0$ the function
\begin{align}\label{wtz}
  \wt z_{\Gamma_0}(t)=\frac{\wt z_0}{1+\Gamma_0\wt z_0(t-t_0)},
\end{align}
then, for $\wt{z}_0$ sufficiently small, $z(t)$ satisfies the estimate 
\begin{align}\label{asymptoticsYconc}
\Big|\frac{z(t)}{\wt z_{\Gamma/\lambda}(t)}-1\Big| \lesssim \wt z_{\Gamma/\lambda}^\gamma(t) \lesssim t^{-\gamma},  \qquad t\geq t_0,
\end{align}
for some $\gamma >0$.
\end{lemma}

\begin{proof}
{\it Proof of \eqref{ODE|Y|^2}.} From the equation for $Y$ in \eqref{Y'} we see that
\begin{align}\label{prY^21}
\begin{split}
& \dot{z}(t) = -\frac{\Gamma}{\lambda} z^2(t) + 2 
  \,\Re\big((R_E(t)+R_Q(t)) \, Y(t) e^{i\Theta(t)}\big)+O(\Rho^{9/4}(t)),\quad 
\end{split}
\end{align}
According to the definitions $Y=Xe^{-i\theta}$, \eqref{X} and \eqref{AsyTheta},
we see that $Ye^{i\Theta} = \overline{W}$. 
Then, to see that the right-hand side of \eqref{prY^21} is of the desired form 
it suffices to check that 
\begin{align}
\Re\big( (R_E(t)+R_Q(t)) (\overline{W}-\overline{A})(t) \big) = O(\Rho^{9/4}(t)).
\end{align}
This follows from \eqref{rmk:controlRE} and \eqref{rmk:controlRQ} 
combined with \eqref{AsyB} and \eqref{AsyW} which give $W = B - N(A) = A + O(|A|^2)$.

\smallskip
{\it Proof of \eqref{asymptoticsYconc}.} 
Note that for $\rho_0$ sufficiently small, $\wt z_0 \approx \rho_0$ in view of \eqref{AsyY-A}.
Then, to prove the asymptotics for $z$ we invoke Remark \ref{rem:lemODE}
to obtain that \eqref{asymptoticsYconc} holds provided we prove the following estimate on the remainders
on the right-hand side of \eqref{ODE|Y|^2}:
\begin{align}
\label{prasyY1}
& \int_{t_1}^{t_2} \Re \big( R_E(s) \overline{A}(s) \big) \, ds \lesssim t_1^{-1-\gamma},
\\
\label{prasyY2}
& \int_{t_1}^{t_2} \Re \big( R_Q(s) \overline{A}(s) \big) \, ds \lesssim t_1^{-1-\gamma}.
\end{align}
The estimate \eqref{prasyY2} is already guaranteed by \eqref{ODEdampestRQ}
and the estimate \eqref{prasyY1} follows from the bounds 
\eqref{ODEdampestRE0} -\eqref{ODEdampestRE4} for the contributions
involving the terms $R_{E,0}, \cdots, R_{E,4}$ (see \eqref{Rsplit'} and \eqref{R_E0-defn}-\eqref{R_E4-defn}). 
\end{proof}

\smallskip
Next, we study the asymptotics of the phase.

\begin{lemma}\label{lem:RHStheta'}
Let $\theta = \mathrm{arg}(X)$ be the phase defined in Lemma \ref{lem:evolY&theta},
satisfying \eqref{theta'}. Then for any $t_0 \leq t_1 < t_2$ we have
\begin{equation}\label{RHStheta'conc1}
\Big| \int_{t_1}^{t_2} \dot{\theta}(s) \, ds 
  - \beta_1 \int_{t_1}^{t_2} |A(s)|^2 \, ds \Big| \lesssim t_1^{-\gamma}
\end{equation}
for some $\gamma>0$, where $\beta_1 \in \R$ is the constant defined in \eqref{extracorrection1}. 
It follows that we can define
\begin{align}\label{RHStheta'conc2}
 \theta_{\infty} := \int^\infty_{t_0} \big( \dot{\theta}(s) - \beta_1 |A(s)|^2 
 \big) \, ds
\end{align}
with 
\begin{equation}\label{RHStheta'conc3}
\Big| \theta(t) - \theta(t_0) - \theta_\infty 
- \beta_1 \int_{t_0}^{t} |A(s)|^2 \, ds \Big| \lesssim 	t^{-\gamma},  
\end{equation}
where $\theta(t_0) = \mathrm{arg}(X(t_0))= \mathrm{arg}(B-N(A))(t_0)$.
\end{lemma}

\begin{proof}
Arguing as in the proof of Lemma \ref{lem:asymptoticsY} above,
by using the equation \eqref{theta'} 
and the fact that $Ye^{i\Theta} = \overline{W} = \overline{A} + O(\Rho)$, we have
\begin{align}\label{theta''}
\frac{d}{dt}\theta = Y^{-2}(t) \Im \big[(R_E(t)+R_Q(t)) \overline{A}(t) \big] + O(\Rho^{5/4}(t)),
\end{align}
and, therefore, \eqref{RHStheta'conc1} would follow provided we show the estimates
\begin{align}
\label{RHStheta'pr1}
& \Big| \int_{t_1}^{t_2} Y^{-2}(s) 
  \, \Im\big( R_E(s) \overline{A}(s) \big) \, ds 
  - \beta_1 \int_{t_1}^{t_2} |A(s)|^2 \, ds \Big| \lesssim t_1^{-\gamma},
\\
\label{RHStheta'pr2}
& \Big| \int_{t_1}^{t_2} Y^{-2}(s) 
  \,  R_Q(s) \overline{A}(s) \, ds \Big| \lesssim t_1^{-\gamma}.
\end{align}

We first prove \eqref{RHStheta'pr2}. 
Let
\begin{align}
F(s) := R_Q(s) \overline{A}(s),  \qquad \Big| \int_{t_1}^{t_2} F(s) \, ds \Big| \lesssim \Rho^{1+\delta_0}(t_1)
\end{align}
(see \eqref{ODEdampestRQ}, where  we may let $\delta_0=1/20$, for example) and let
$q(s) := Y^{-2}(s) \approx \Rho^{-1}(s),$ 
with
$$
q' = -Y^{-4} \frac{d}{dt} z = O(\Rho^{-8\delta}),$$ 
having used the remainder bounds from Remark \ref{rmk:controlR} for the last equality.
Then we can integrate by parts to obtain
\begin{align}
\Big| \int_{t_1}^{t_2} Y^{-2}(s)  \,  R_Q(s) \overline{A}(s) \, ds \Big| 
	& \lesssim \Big| q(t_1)  \, \int_{t_1}^{t_2} F(\tau) \, d\tau  \Big|
  + \Big| \int_{t_1}^{t_2} q'(s)  \int_{s}^{t_2} F(\tau) \, d\tau \, ds \Big| 
  \lesssim t_1^{-\gamma}
\end{align}
provided $\gamma < \delta_0 - 8 \delta$.

For \eqref{RHStheta'pr1} we can proceed similarly: recalling that 
$R_E = R_{E,0} + \cdots + R_{E,4}$ and that we have the bounds 
\eqref{ODEdampestRE0}-\eqref{ODEdampestRE4} we see that all the terms, with the exception of $R_{E,2}$
can be handled as $R_Q$ above and satisfy
\begin{align}
\Big| \int_{t_1}^{t_2} Y^{-2} \Im\big( R_{E,j} \overline{A} \big) \, ds \Big| \lesssim t_1^{-\gamma},
\end{align}
for $j=0,1,3,4$, and $\gamma < 1/4 - 8\delta$.
For the remaining contribution involving $R_{E,2}$ we go back to the 
proof of the estimate \eqref{ODEdampestRE2} that follows \eqref{R_E2-defn'}
and recall that we write 
\begin{align}
R_{E,2}(t) = T[R_w](t) = T[R_{w,1}](t) + T[I](t) + T[II](t) + T[III](t)
\end{align}
where the terms are defined in \eqref{Rw1} and \eqref{Rw2split}.
Since we have already proven that 
\begin{align}
|T[R_{w,1}](t) \overline{A}(t)| + |T[I](t) \overline{A}(t)| + |T[II](t) \overline{A}(t)| 
  \lesssim t^{-9/4},
\end{align}
we only need to treat the term $T[III](t) \overline{A}(t)$ which involves $III$.
We then invoke \eqref{J0toref} 
and the same integration by parts argument used above for $R_Q$, to see that all the terms in $T[III](t) \overline{A}(t)$ 
give acceptable remainders with the exception of the term $J_0$;
this latter contributes, see \eqref{extracorrection1},
\begin{align}
\label{RHStheta'pr5}
\int_{t_1}^{t_2} |Y(s)|^{-2}
  \, \Im\big( J_0(s) \big) \, ds
  = \beta_1 \int_{t_1}^{t_2} |A(s)|^2 \, ds + O(\Rho^{1/2}(t_1))
\end{align}
having used \eqref{AsyY-A}.
Equation \eqref{RHStheta'pr5} gives the additional shift in the claim \eqref{RHStheta'conc1} up to acceptable remainders.
The conclusion \eqref{RHStheta'conc1} is proved, and \eqref{RHStheta'conc2} and \eqref{RHStheta'conc3} follow.
\end{proof}


\smallskip
We can now conclude  the proof of  Theorem \ref{thm:asympA}.

\begin{proof}[Proof of Theorem \ref{thm:asympA}]
We want to obtain the asymptotic behavior of $A$ as in \eqref{Asymp-A}.
Recall our previous definitions \eqref{derB}, \eqref{AsyW} and \eqref{X}:
\begin{align}\label{asympApr0}
\begin{split}
& B =  A +O(|A|^2), \qquad W = B- N(A), \qquad  W = X e^{-ic_2\int_{t_0}^t|W(s)|^2ds} ,
\\
& Y = |X| = |W|, \qquad \theta = \arg(X), 
\end{split}
\end{align}
where $\Theta$ is defined in \eqref{AsyTheta}.
Defining $y(t) > 0$ by
\begin{align}
y^2(t) := \frac{Y^2(t_0)}{1+ (\Gamma/\lambda) Y^2(t_0)(t-t_0)}, 
\end{align}
according to the notation \eqref{wtz},
and applying the conclusion \eqref{asymptoticsYconc} from Lemma \ref{lem:asymptoticsY} we have 
\begin{align}\label{Y-y}
\Big|Y^2(t) - y^2(t) \Big| \lesssim t^{-1-\gamma},
\qquad \Big|Y(t) - y(t) \Big| \lesssim t^{-1/2-\gamma}, \qquad t\geq t_0.
\end{align}
From \eqref{Y-y} and \eqref{asympApr0}, which in particular gives $Y = |A| + O(|A|^2)$, 
we deduce that 
\begin{align}
\int_{t_0}^{t} \big( |A(s)|^2 - y^2(s) \big) \, ds, \qquad 
\int_{t_0}^{t} \big( |W(s)|^2 - y^2(s) \big) \, ds
\end{align}
are Cauchy sequences as $t\rightarrow \infty$; we can then define their respective limits $\Phi_A$ and $\Phi_W$ in $\R$, and have
\begin{align}\label{asympApr2}
\Big| \int_{t_0}^{t} \big( |A(s)|^2 - y^2(s) \big) \, ds - \Phi_A\Big| + \quad 
\Big| \int_{t_0}^{t} \big( |W(s)|^2 - y^2(s) \big) \, ds - \Phi_W\Big| \lesssim t^{-\gamma}.
\end{align}
Combining this with Lemma \ref{lem:RHStheta'} we obtain that 
\begin{equation}\label{asympApr3}
\Big| \theta(t) - \theta(t_0) - \theta_\infty 
- \beta_1 \int_{t_0}^{t} y(s)^2 \, ds -  \beta_1\Phi_A\Big| \lesssim t^{-\gamma}, \qquad t\geq t_0.
\end{equation}
It also follows that 
\begin{align}
\Big|  W(t) - e^{i(\beta_1 - c_2)\int_{t_0}^t y^2(s) ds + i\theta(t_0) + i\theta_\infty +  i \beta_1\Phi_A -  i c_2\Phi_W} y(t)  
  \Big| \lesssim t^{-1/2-\gamma}, \qquad t\geq t_0 
\end{align} 
and the same holds with $A$ replacing $W$.
That is the desired conclusion in Theorem \ref{thm:asympA}  since
\begin{align}
& \int_{t_0}^t y^2(s) ds = \frac{\lambda}{\Gamma} \log\big( 1+ (\Gamma/\lambda) Y^2(t_0)(t-t_0) \big)
 = \frac{\lambda}{\Gamma} \log\big( 1+ (\Gamma/\lambda) \rho_0 t \big) + 
  \frac{\lambda}{\Gamma} \log\big( Y^2(t_0)/ \rho_0 \big) + O(t^{-1}),
\\
& y(t) = \sqrt{\frac{\lambda}{\Gamma}} t^{-1/2} + O(t^{-3/2}), \quad t \gg t_0,
\end{align}
and letting
\begin{align}\label{psiinftyformula}
\psi_\infty := \theta(t_0) + \theta_\infty + \beta_1\Phi_A -  c_2\Phi_W +  (\beta_1-c_2) \frac{\lambda}{\Gamma} \log\big( Y^2(t_0)/ \rho_0 \big)  .
\end{align}
\end{proof}

\medskip
\subsection{Growth of the weighted norm}

In this last subsection we show how the resonant component has a growing weighted norm.

\begin{lemma}[Growth of the resonant component]\label{app:growth-f}
Assume the hypotheses of Theorem \ref{thm:mainthmdetail}, and let
$\widetilde g$ be defined by \eqref{bad-g-defn}. 
Then, there exist constants $c_\ast,C_\ast>0$, independent of $\varepsilon_0$, such that
\begin{equation}\label{growth-bad}
{\|\partial_k\widetilde g(t)\|}_{L^2_k} \geq c_\ast t^{1/2}, \qquad t\geq C_\ast\rho_0^{-1}.
\end{equation}
\end{lemma}

\begin{proof}
The argument is the one-dimensional analogue of the resonant lower-bound
argument in \cite{LP21}; for completeness, we include here a short proof 
based on our specific asymptotics from Theorem \ref{thm:asympA}.
Set
\begin{align*}
\omega(k):=\langle k\rangle-2\lambda,\qquad \theta:=P_c(\beta\varphi^2),\qquad
  \chi(k):=\phi_{\leq-C}(\omega(k)).
\end{align*}
By the Fermi Golden Rule \eqref{Fermi-Gordon}, there exists
$k_\ast\in\{-\sqrt{4\lambda^2-1},\sqrt{4\lambda^2-1}\}$
such that $\widetilde\theta(k_\ast)\neq0$. 
Moreover, $\omega'(k_\ast)\neq 0$. Hence, there exists a small neighborhood $U$ of $k_\ast$,
such that
\begin{equation}\label{nondegenerate-coefficient}
\big| \omega'(k)\chi(k) \widetilde{\theta}(k) \big| \geq c_0>0,\qquad k\in U.
\end{equation}
Theorem \ref{thm:asympA}, together with $B-A=O(\Rho)$, gives, for $s\geq t_0:=C_0\rho_0^{-1}$,
\begin{equation}\label{B-asymptotic-for-growth}
B(s)=B_\infty(s)+r_B(s),\qquad |r_B(s)|\lesssim s^{-1/2-\gamma},
\end{equation}
with $\gamma \in (0,1/2)$, and where
\begin{align}\label{Binfty}
& B_\infty(s) :=
 \exp\Big[ i \Big(\psi_\infty 
  - \kappa \log \big(1 + \mathfrak{b} s\big) \Big) \Big]
  \sqrt{\frac{\lambda}{\Gamma \cdot s}}
  \qquad \mathfrak{b} :=\frac{\Gamma}{\lambda}\rho_0,  \quad 
  \kappa := (c_2-\beta_1)\frac{\lambda}{\Gamma}.
\end{align}

We then want to look at the weighted norm of the bad Fermi component $g$ in \eqref{bad-g-defn}.
Differentiating $\wt{g}$ in \eqref{bad-g-defn} gives the leading contribution 
\begin{align}
\mathcal{M} (t, k) := 
  i\omega'(k)\chi(k)\widetilde{\theta}(k)
  \int_0^t s\, e^{-is(\lr{k}-2\lambda)} B^2(s) \, ds, 
\end{align}
where we are disregarding the lower orders term where $\partial_k$ hits 
the smooth function $\wt{\theta}$ or the cutoff $\chi$.
At the resonant frequency $k=k_\ast$, using \eqref{Binfty} we have
\begin{align*}
\int_{t_0}^t s \, B_\infty^2(s) ds
& =
\frac{\lambda}{\Gamma}e^{2i\psi_\infty}
\int_{t_0}^t(1+\mathfrak bs)^{-2i\kappa} ds 
\\
& = \frac{\lambda}{\Gamma}e^{2i\psi_\infty}
  \frac{(1+\mathfrak bt)^{1-2i\kappa} -(1+\mathfrak bt_0)^{1-2i\kappa}} {\mathfrak b(1-2i\kappa)}.
\end{align*}
Since $|(1+\mathfrak bs)^{1-2i\kappa}|=1+\mathfrak bs$, it follows that
\begin{equation}\label{resonant-time-integral}
\left| \int_{t_0}^t sB_\infty^2(s) ds \right|
 \geq \frac{\lambda}{\Gamma\sqrt{1+4\kappa^2}} (t-t_0).
\end{equation}
Furthermore,
\begin{align}
\left| \int_{t_0}^t sB_\infty^2(s) \big(e^{-is\omega(k)}-1\big) ds \right| \lesssim |\omega(k)|t^2. 
\end{align}
Consequently, for a sufficiently small fixed $\delta>0$ and all
$t\geq2t_0$,
\begin{align}
\Big| \int_{t_0}^t sB_\infty^2(s)e^{-is\omega(k)} \,ds \Big| \gtrsim t
\end{align}
whenever $k\in U$ and $|\omega(k)|\leq \delta t^{-1}$. Since
$\omega'(k_\ast)\neq0$, the set $E_t:=\{k\in U \, :\, |\omega(k)|\leq\delta t^{-1}\}$
has measure $|E_t|\gtrsim t^{-1}$. Then, using \eqref{nondegenerate-coefficient}, 
we obtain
\begin{equation}\label{main-growth-term}
{\|\mathcal M(t)\|}_{L^2_k} \gtrsim t \, |E_t|^{1/2} \gtrsim t^{1/2}.
\end{equation}

To complete the proof of the lemma it remains to estimate the error
between $\widetilde{g}$ and the expression just lower bounded above. 
On the support of $\chi$, the map $k\mapsto\omega(k)$ is a smooth change of variables on each of the two
components. 
Thus, using Plancherel we get
\begin{align}
{\Big\| a(k)\int_I e^{-is\omega(k)}F(s) \, ds \Big\|}_{L^2_k}
\lesssim_a {\|F\|}_{L^2_s(I)}
\end{align}
for every smooth coefficient $a$. 
From \eqref{B-asymptotic-for-growth} we know that
$\big| B^2(s) - B_\infty^2(s) \big| \lesssim s^{-1-\gamma},$
and, therefore,
\begin{equation}\label{B-asymptotic-error-growth}
{\Big\| \omega'(k) \chi(k) \widetilde{\theta}(k)
 \int_{t_0}^t s (B^2-B_\infty^2)(s)e^{-is\omega} \, ds \Big\|}_{L^2_k}
  \lesssim t^{1/2-\gamma}.
\end{equation}
The contribution over $0\leq s\leq t_0$, with $t_0 \lesssim \rho_0^{-1}$, is bounded by
\begin{align}
{\| s \, B^2(s) \|}_{L^2_s([0,t_0])} \lesssim \rho_0 t_0^{3/2} \lesssim \rho_0^{-1/2},
 \end{align}
Combining these estimates with \eqref{main-growth-term}, we obtain
\begin{align}
{\| \partial_k \widetilde{g}(t) \|}_{L^2_k} \geq ct^{1/2} -C t^{1/2-\gamma} -C\rho_0^{-1/2}.
\end{align}
For $\rho_0$ sufficiently small, the last two terms are absorbed by the
first whenever $t \geq C_\ast \rho_0^{-1}$, after choosing $C_\ast$ sufficiently large. 
This proves \eqref{growth-bad}.
\end{proof}

\bigskip
\section{Applications in scalar field models}\label{sec:application}
In this section we apply the preceding analysis to obtain new results on the asymptotic stability
of kinks in scalar field models.

\subsection{Basic setting}
We consider scalar field theories on the line governed by the Lagrangian action functionals 
\begin{equation} \label{eq:lagrangian}
    \iint_{\mathbb{R}^{1+1}} \biggl( \frac12 (\partial_t \phi)^2 - \frac12 ( \partial_x\phi)^2 - W(\phi) \biggr) \, dx \, dt,
\end{equation}
where $\phi \colon \mathbb{R} \times \mathbb{R} \to \mathbb{R}$ is the scalar field and $W \colon \mathbb{R} \to [0,\infty)$ is a double-well self-interaction potential such that
\begin{itemize}
    \item (Regularity) $W \in \mathcal{C}^\infty(\mathbb{R})$,\\
    \item (Existence of minima) $W$ has (at least) two consecutive global 
    minima denoted by $\phi_- <\phi_+$ with $W(\phi_\pm) = W'(\phi_\pm) = 0$ and $m^2 := W''(\phi_\pm) > 0$.
\end{itemize}
The minima $\phi_-$ and $\phi_+$ are also referred to as the vacuum states of the model.

The Euler-Lagrange equations associated with \eqref{eq:lagrangian} are the wave-type equations
\begin{equation} \label{eq:waveeq}
 \begin{aligned}
  (\partial_t^2 - \partial_x^2) \phi + W'(\phi) = 0, \quad (t,x) \in \mathbb{R} \times \mathbb{R}.
 \end{aligned}
\end{equation}
 The equation \eqref{eq:waveeq}  admits a static solution $K(x)$ called kink  that connects the two distinct vacuum states
\begin{equation} \label{eq:kink_ode}
 \left\{ \begin{aligned}
  -\partial_x^2 K + W'(K(x)) &= 0, \quad x \in \mathbb{R}, 
  \\
  \lim_{x\to\pm\infty} K(x) &= \phi_\pm.
 \end{aligned} \right.
\end{equation}

Our analysis for the quartic equation can be applied to study the asymptotic stability of such kink solutions under the influence of internal modes satisfying the nonlinear Fermi Golden Rule. The models we considered below are constructive and obtained as  perturbations of the double-well potential for the classical $\phi^4$ model.

More precisely, we consider $W$ satisfying several conditions:

\begin{itemize}
\smallskip
\item  (Symmetry)  the potential $W(\phi)$ is a smooth even function, \\

\smallskip
\item      (Positivity)
$W(\phi) > 0 \quad \text{for } \phi \neq \pm \varphi$,\\

\smallskip
\item (Asymptotics) For some integer $M>4$ and $m>0$, $W$ satisfies
\begin{equation}\label{Wpm}
 W(\phi) = \frac{m^{2}}{2}(\phi \mp \varphi)^{2} + O \left(|\phi \mp \varphi|^{M}\right), 
\qquad \phi \to \pm \varphi.
\end{equation}
\end{itemize}
One can directly check 
$-\phi_-=\phi_+=\varphi>0$ and that the kink connecting $-\varphi$ and $\varphi$,
with $K(0)=0$, is an \emph{odd} function.

Notice that under the symmetry assumption of $W$, clearly, Equation \eqref{eq:waveeq} preserves the oddness.

\medskip
In what follows, we address the asymptotic stability of the kink $K(x)$ under odd perturbations. Inserting  the odd function $\phi(t,x)=K(x)+u(t,x)$ into \eqref{eq:waveeq} results in the equation
\begin{equation}
    \partial_t^2 u-\partial_x^2 u+m^2u+V(x)u =N(K,u)
\end{equation}where $N$ collects higher order terms and 
\begin{equation}\label{eq:potential}
V(x)=W''(K(x))-m^2.
\end{equation}

Consider the linear operator 
\begin{equation}\label{eq:klinear}
 H := -\partial_x^2 + V(x), \qquad L := \sqrt{H+ m^2}.
\end{equation}
The following properties can readily be verified using the standard spectral theory:

\smallskip
\begin{itemize}

\item The continuous spectrum of $H+m^2$ is given by $[m^2,\infty)$,

\smallskip
\item $\lambda_0=0$ is in the discrete spectrum of $H+m^2$ with eigenfunction given by $K'(x)$;
since $K'$ is even, this is not present when restricting to odd perturbations. 
Moreover,  all other discrete eigenvalues are located on $(0,m^2)$. 
\end{itemize}

\smallskip
We further assume that $H$  satisfies the following assumptions,
 when restricted to odd functions:

\begin{itemize}
\smallskip
\item The bottom of the continuous spectrum 
  is neither an eigenvalue (with associated odd eigenfunction)
  nor a resonance  (with associated odd threshold eigenfunction).

\smallskip
\item  The discrete spectrum of $H+m^2$ (restricted to odd functions) 
  contains only one point $\omega^2 \in(m^2/4,m^2)$,
  with corresponding eigenfunction $\varphi_{\omega}$
\begin{equation}\label{varphiomega}
(H+m^2)\varphi_{\omega} = \omega^2\varphi_\omega, \qquad  4\omega^2 > m^2.
\end{equation}

\smallskip
\item The nonlinear Fermi Golden Rule holds:
\begin{equation}\label{nFGRW}
  \int \varphi_{4\omega^2}(x) W^{'''}(K(x)) \varphi_{\omega}^2(x) \, dx \neq 0,
\end{equation}
where $\varphi_{4\omega^2}$ is the nonzero odd solution to 
  $(H+m^2)\varphi_{4\omega^2} = 4\omega^2 \, \varphi_{4\omega^2}$
  with $\varphi_{4\omega^2}'(0)=1$.
\end{itemize}

\smallskip
\begin{remark}
The double-well potential from the classical $\phi^4$ model, 
  $W_{\phi^4}(\phi)= (1/4) (\phi^{2} - 1^{2})^{2}$, 
  satisfies all the conditions above, except the flatness condition \eqref{Wpm}, which
  it satisfies with $M=3$;
  indeed, all spectral assumptions are satisfied by the linearized operator 
  $H_{\phi^4}+ 2$ arising from $W_{\phi^4}$
  when restricted to odd functions.
\end{remark}

The construction of double-well potentials 
satisfying all the properties above as perturbations of $W_{\phi^4}(\phi)$ can be found  in \cite{KK11,KKK}.

\subsection{An example}
For the sake of completeness, we provide here an explicit 
construction of double-well potential so that all desired properties hold. 
The idea is to construct a perturbation of $W_{\phi^4}$ that slightly flattens 
the graph of $W_{\phi^4}$ around the vacuum states $\pm \varphi$.

We start by recalling the basic information for the $\phi^4$ model. Without loss of generality, we set $\varphi=1$. In other words, the double-well potential is given by
\begin{equation}
    W_{\phi^4}(\phi)=\frac{(\phi^{2} - 1)^2}{4}.
\end{equation}
In this setting, the kink solution is given by
\begin{equation}
    K_{\phi^4}(x)=\tanh\big({\tfrac{x}{\sqrt{2}}}\big)
\end{equation}and
\begin{equation}
    m^2= W''_{\phi^4}(\pm 1)=2.
\end{equation}
Linearizing the equation around the solution $K$, the corresponding linearized operator is given by
\begin{equation}
H_{\phi^4} + 2 :=-\partial_x^2+2+V_{\phi^4}
\end{equation}
where
\begin{equation}
V_{\phi^4} = W''_{\phi^4}\big(K_{\phi^4}(x)\big)-2
= -3\,\cosh^{-2} \left(\frac{x}{\sqrt{2}}\right).
\end{equation}
It is a  well-known fact that
the continuous spectrum of the operator
$H_{\phi^4} + 2$ is  the interval $[2,\infty)$. 
Moreover, restricting to odd functions, all the above spectral conditions are met: 

\begin{itemize}
\smallskip
\item The discrete spectrum of $ H_{\phi^4}+2$ consists only of one point
$\lambda = 3/2$ with corresponding eigenfunction
\begin{align}
\varphi_1(x) = \sinh(x/\sqrt{2}) \cosh^{-2}(x/\sqrt{2}).
\end{align}

\smallskip
\item The bottom of the continuous spectrum is neither an eigenvalue nor a resonance.

\smallskip
\item For the nonlinear Fermi Golden Rule, one has (see, for example, \cite{KMM17,DM22})
\begin{align}\label{phi4FGR}
\int \varphi_{6}(x) \, W'''(K(x)) \, \varphi_1^2(x) \,dx = 
6  \int \varphi_{6}(x)\,\frac{\sinh^{3}(x/\sqrt{2})}{\cosh^{5}(x/\sqrt{2})}\,dx \neq 0   
\end{align}
where \(\varphi_{6}(x)\) is the nonzero odd solution to
\[
(H_{\phi^4}+2)\varphi_{6}(x)=6\varphi_{6}(x), \qquad \varphi_6'(0)=1.
\]
\end{itemize}

From the standard spectral theory, the three preceding conditions are stable under small 
perturbations (in a sufficiently strong topology). 
Therefore, it will suffice to construct  a small perturbation of $ W_{\phi^4}(\phi)$ 
such that \eqref{Wpm} is satisfied.  
We first look at the Taylor expansion of $W_{\phi^4}(\phi)$ near $1$,
\begin{align}
   W_{\phi^4}(\phi)=\frac14\,(1-\phi^{2})^{2}
= (\phi-1)^{2}+(\phi-1)^{3}+\frac14(\phi-1)^{4}. 
\end{align}
Note that the cubic and quartic terms produce non-localized quadratic 
and cubic terms in nonlinear equations for the error term in the stability analysis of the kink for the $\phi^4$ model. Thus, the modified potential must cancel those terms which we cannot 
handle with our analysis.

For any $c_5\in \mathbb{R}$, we define
\begin{align}\label{eq:Wdelta}
W_\eta (\phi) = W(\phi)
:= W_{\phi^4}(\phi)
- \Bigl[\bigl(|\phi|-1\bigr)^{3}+\tfrac{1}{4}\bigl(|\phi|-1\bigr)^{4}
  -\frac{c_5}{5!} (|\phi|-1)^5\Bigr]\chi_{\eta}\bigl(|\phi|-1\bigr).
\end{align}
where $\chi_{\eta}(z)=\chi(z/\eta)$, 
for $0<\eta\ll 1$. Here  $\chi (z) $ is a smooth bump function  such that 
$\chi(z)=1$ for $|z|<\tfrac12$ and $\chi(z)=0$ for $|z|>1$. 
In what follows we will often omit the dependence on $\eta$ of $W_\eta$, and 
related quantities (e.g. the associated effective potential $V_\eta=V$ and kink $K_\eta = K$).

\begin{lemma}
  The potential $W_\eta$ is smooth, even, satisfies \eqref{Wpm} with $M\geq 5$ 
  ($M=5$ when $c_5\neq 0$), 
  and can be written as 
  \begin{align}\label{Wm}
W_\eta(\phi)=\begin{cases}
    (|\phi|-1)^2+q_5(|\phi|-1)^5\quad \mathrm{if}\quad \big||\phi|-1\big| < \eta/2
    \\    
    W_{\phi^4}(\phi)\qquad \qquad \qquad \qquad\mathrm{if}\quad \big||\phi|-1\big| > \eta
\end{cases}      
\end{align}
with $q_5 := c_5/5!$.

Moreover, the equation 
  \begin{equation}\label{kink}
     -\partial_x^2K + W_\eta'(K(x))=0,
     \qquad \lim_{x\rightarrow \pm \infty}K(x)=\pm 1
  \end{equation}
  has a unique odd solution $K_\eta$,
  and the linearized operator $H_\eta+2 =-\partial_x^2+2+V_\eta$,
  with $V_\eta :=  W''_{\eta}(K_\eta(x)) - 2$
  restricted to the space of odd functions, satisfies the desired spectral conditions \eqref{eq:klinear}-\eqref{nFGRW} provided that $\eta$ is chosen sufficiently small.
\end{lemma}

\begin{proof}
The claims on $W$ including \eqref{Wm} follow by construction. 
In particular, one has $W^{(5)}(1)=-W^{(5)}(-1)=c_5.$
The existence of a unique odd kink solution $K(x)$ to \eqref{kink} is standard. Using
\begin{equation}
    \int_0^{K(x)} \frac{1}{\sqrt{2W(s)}}\,ds = x
\end{equation} and $W(s)\gtrsim  (s-1)^2$ for $0\leq s<1$, 
one can show that $K(x)\rightarrow 1$ exponentially as $x\rightarrow \infty$. Similarly, one can also obtain $K(x)\rightarrow -1$ as $x\rightarrow -\infty$ exponentially. By construction, one also verifies
\begin{equation}
    \sup_{x\in\mathbb{R}}\left|K_{\phi^4}(x)-K(x)\right|\lesssim \eta.
\end{equation}
It remains to show that when $\eta$ is small enough,  the linearized operator
satisfies the desired spectral conditions.  
By standard spectral theory, 
it will suffice to show that $V$ is a small perturbation of $V_{\phi^4}$ in a sufficiently strong topology.
First note that from \eqref{Wm}, we have 
\begin{equation}
    \sup_{\phi\in \mathbb{R}}\left(|W(\phi)-W_{\phi^4}(\phi)|
    +|W'(\phi)-W'_{\phi^4}(\phi)|+|W''(\phi)-W''_{\phi^4}(\phi)|\right)\lesssim \eta.
\end{equation}
Fix a $0< \eta \ll 1$ and define the set 
\begin{align}
\mathcal{U}_\eta := \left\{x\in\mathbb{R} \, : \,\, \big||K_{\phi^4}(x)|-1\big|<\eta\right\}
= \big\{ |x| > \sqrt{2} \,\mathrm{arctanh}(1-\eta)=:R_\eta \big\}.
\end{align}
From \eqref{Wm}, 
we know that $V_\eta \equiv V_{\phi^4}$ for $x\in \R\setminus \mathcal{U}_\eta$  
since in this region $K_\eta \equiv K_{\phi^4}$.  
Then,
\begin{align}
\sup_{x\in\mathcal{U}_\eta}\left| V(x)-V_{\phi^4}(x)\right|&=\sup_{x\in\mathcal{U}_\eta}
  \left| W''(K(x))-W''_{\phi^4}(K_{\phi^4}(x))\right|
  \\
&\lesssim  \sup_{x\in\mathcal{U}_\eta}\left| W''(K(x))-W''_{\phi^4}(K(x))\right|
  +\sup_{x\in\mathcal{U}_\eta}\left| W''_{\phi^4}(K(x))-W''_{\phi^4}(K_{\phi^4}(x))\right|
  \\
&\lesssim {\|W'''_{\phi^4}\|}_{L^\infty([-1,1])}
  \sup_{x\in\mathcal{U}_\eta}
  \left|K(x)-K_{\phi^4}(x)\right|+\sup_{x\in\mathcal{U}_\eta}
  \left| W''(K(x))-W''_{\phi^4}(K(x))\right|
  \\
&\lesssim \eta.
\end{align}
Similarly, if $x\in \mathcal{U}_\eta$, i.e. $|x| \gtrsim |\log\eta|$,
we can derive weighted estimates on the difference of the potentials.
Indeed, for $|x| \geq R_\eta \approx |\log\eta|$ we have
\begin{align}
|V(x) - V_{\phi^4}(x)| \lesssim \eta \cdot e^{-c(|x| - R_\eta)}
\end{align}
and integrating this gives
\begin{align}\label{weights}
{\big\| \lr{x}^q \big( V - V_{\phi^4} \big) \big\|}_{L^1} \lesssim \eta \cdot (1+|\log\eta|)^{q}.
\end{align}
Thanks to the above bounds 
one can see that the desired spectral properties for $H_\eta+2$ 
follow from the analogous properties for the spectrum of $H_{\phi^4}+2$.
In particular, the structure of the continuous and discrete spectrum and the absence of 
(odd) threshold resonance follow from resolvent identities.

For the validity of the FGR property \eqref{nFGRW} it suffices to show that
\begin{align}\label{FGRconvergence}
\int \varphi_{4\omega^2}(x) W_\eta^{'''}(K_\eta(x)) \varphi_{\omega}^2(x) \, dx 
  \, \stackrel{\eta \rightarrow 0}{\longrightarrow} \,
  \int \varphi_{6}(x) \, W'''_{\phi^4}(K_{\phi^4}(x)) \, \varphi_1^2(x) \,dx
\end{align}
and then use \eqref{phi4FGR}.
To verify that \eqref{FGRconvergence} holds we 
observe the following: (a) $W'''_{\phi^4}(K_{\phi^4})$ and $W'''_{\eta}(K_\eta)$ are bounded for $x\in\R$
uniformly in $\eta$,
and agree for $|x| \leq R_\eta$,
(b) the bound states $\varphi_1$ and $\varphi_{\omega}$ are exponentially decaying,
$\omega^2 \rightarrow 3/2$ and $\varphi_\omega \rightarrow \varphi_1$ in $L^2$
as $\eta\rightarrow 0$,
(c) the generalized eigenfunctions $\varphi_{4\omega^2}$ and $\varphi_6$ 
are uniformly bounded, and convergence locally uniformly as $\eta\rightarrow0$, 
and (d) we have
\begin{align}
\int_{|x| \gtrsim R_\eta} \varphi_1^2(x)\,dx \lesssim \eta.
\end{align}
Putting all these together gives us \eqref{FGRconvergence} and concludes the proof of the lemma.
\end{proof}


Finally, we comment that similar higher order perturbations of
the $\phi^4$ model can be constructed in a similar way as above.

\subsection{Asymptotic stability of kinks}
With the same definitions used above, let us consider odd perturbation 
of $K(x)$ in the form $$\phi(t,x)=K(x)+u(t,x),$$
so that
\begin{equation}
    \partial_t^2 u-\partial_x^2 u+m^2u+V(x)u =N(K,u)
\end{equation}
where the nonlinear term is
\begin{equation}
  N(K,u) = - \big[ W'(u+K)-W'(K) - W''(K) u \big] . 
\end{equation}
Note that \begin{equation}
    -W'(\phi)=-m^2(\phi \mp 1)+O(|\phi \mp 1|^{M-1}).
\end{equation}

 By using Taylor's expansion, we can rewrite the nonlinear term 
 to extract the localized terms. The equation of the perturbation $u$
reads
\begin{equation}\label{eq:Kinkerrormodel}
\partial_t^2 u-\partial_x^2 u + m^2u + V(x)u 
  = \beta_2(x) u^2+\beta_3(x) u^3+\beta_4(x) u^4+\mathfrak{N}(K,u)
\end{equation}
where $\beta_2(x) = (-1/2) W'''(K(x))$ is odd,  
$\beta_3(x) = (-1/6)W^{(4)}(K(x))$ is even and $\beta_4(x) = -(1/4!)W^{(5)}(K(x))$ is odd.  
Both $\beta_2(x)$ and $\beta_3(x)$ decay to $0$ exponentially as $x\rightarrow\pm \infty$, 
but 
\begin{align}
\beta_4(x) \stackrel{x\rightarrow\pm \infty}{\longrightarrow} b_{\pm} := \mp \frac{c_5}{4!} = \mp 5q_5
\end{align}
exponentially fast ($\beta_4$ is actually constant outside a sufficiently large interval, see \eqref{Wm}).
Finally, $\mathfrak{N}(K,u)$ collects quintic and higher powers  in $u$.

Similarly to the general model \eqref{nlkg-quartic}, 
in order to study \eqref{eq:Kinkerrormodel} and prove our main final
result (see Theorem \ref{thm:Kink}), we focus our attention on the 
contribution of the quartic term $\beta_4(x) u^4$ to Duhamel's formula, as in Section \ref{sec:PDE-part2}. 
 This is the most important contribution to the PDE,
and remains the main difficulty in the analysis of the internal mode's ODE. 
The analysis for all the other terms can be carried out
similarly to Sections \ref{sec:Linear}-\ref{sec:PDE-part2}.
In particular, one should observe the following:

\begin{itemize}

\item[(1)] The linear estimates established in Section \ref{sec:Linear}
also hold in this context, where we are restricting to the odd spectral subspace
of the linear operator; 
we refer the reader to \cite{GPZ23} for some details on how to obtain linear dispersive 
estimates and local decay (on the continuous spectral subspace) in our current framework.
As in Propositions 3.1 and 3.2 in \cite{GPZ23} one can use the oddness of the spectral kernels
to compensate for the assumption $T(0)=0$ which may not hold.
The other estimates needed in our argument can be obtained by small adaptations
of the proofs given in Sections \ref{sec:Linear} and \ref{sec:linear-decay} along these same lines.

\smallskip
\item[(2)] Localized cubic terms of the form $\beta_3(x)u^3$ are more perturbative than
the quadratic ones $\beta_2(x)u^2$ in the context of the PDE estimates from 
Sections \ref{sec:PDE-part0}-\ref{sec:PDE-part2}.
They only contribute remainders to the equation for the amplitude $|A|^2$.
However, they contribute to the overall logarithmic phase corrections in the asymptotic,
where one needs to replace the constant $c_2$ in Theorem \ref{thm:asympA} by
\begin{align}
c_3 = c_2  + \frac{3}{2\omega} \int \beta_3(x) \varphi_\omega(x)^4 \, dx.
\end{align}

\smallskip
\item[(3)] The nonlinear terms with homogeneity at least $5$ (which may be localized or not)
can be treated perturbatively and do not need the interpolation and commutation
identities from Section \ref{sec:PDE-part2}.

\end{itemize}

Recall the notation of the partition $\{\tau_m\}_{m=0,\dots,L+1}$ 
of the time interval $[0,t]$ introduced earlier in \eqref{timeloc}. We denote by
$\mathsf{Q}_{m}$ the quartic term 
\begin{equation}
\label{Qbeta}
\begin{aligned}
\mathsf{Q}_{m}(x)=\int_0^\infty e^{-isL}P_c(\beta_4 v^4)(s,x)\tau_m(s)ds,\qquad \mathsf{Q} (t):= \sum_{m=0}^{L+1} \mathsf{Q}_{m}(x).
   \end{aligned}
\end{equation}
where 
\begin{equation}
  u(t)=a(t)\varphi_\omega+v,\quad (v,\varphi_\omega) = 0  .
\end{equation}
We also slightly abuse notation by assuming without loss of generality that the mass $m^2 = 2$
and denoting $\lr{x} := \sqrt{x^2+2}$ in what follows. 
Using the relation \eqref{relation-fwv}, we write (omitting the time arguments in the profiles),
\begin{align}
2^4e^{-is \langle k\rangle}\wt{\mathcal{F}}(\beta_4 v^4)
& = \Suml{\ell_1}{\ell_2}{\ell_3}{\ell_4}\ell_1\ell_2\ell_3\ell_4
\iiiint \frac{e^{is\Phi_{\ell_1\ell_2\ell_3\ell_4}}}{\prod_{j=1}^4
\langle \xi_j\rangle} \\ & \times \threef{f}{\ell_1}{\xi_1}{f}{\ell_2}{\xi_2}{f}{\ell_3}{\xi_3}
\onef{f}{\ell_4}{\xi_4}\nu_{\ell_1\ell_2\ell_3\ell_4}d\xi_{1234}
\end{align}
where $d\xi_{1234}$ is a shorthand for $d\xi_1d\xi_2d\xi_3d\xi_4$ 
and 
\begin{equation}\nu_{\ell_1\ell_2\ell_3,\ell_4}(k,\xi_1,\xi_2,\xi_3,\xi_4)=\int \beta_4(x)\firstpsi{\psi}{x}{k}{\psi}{\ell_1}{x}{\xi_1} \scdpsi{\psi}{\ell_2}{x}{\xi_2}{\psi}{\ell_3}{x}{\xi_3}\psi^{(\ell_4)}(x,\xi_4)dx\end{equation}
and
\begin{equation}\Phi_{\ell_1\ell_2\ell_3\ell_4}(k,\xi_1,\xi_2,\xi_3,\xi_4)=-\langle k\rangle+\ell_1\langle \xi_1\rangle+\ell_2\langle \xi_2\rangle+\ell_3\langle \xi_3\rangle+\ell_4\langle \xi_4\rangle.\end{equation}


According to \eqref{decomposition-sing-part}, we decompose the nonlinear spectral distribution  into two parts
\begin{equation}\label{NDFbeta}(2\pi)^{\frac{5}{2}}\nu_{\ell_1\ell_2\ell_3\ell_4}(k,\xi_1,\xi_2,\xi_3,\xi_4):=\nu_{\ell_1\ell_2\ell_3\ell_4}^{S}(k,\xi_1,\xi_2,\xi_3,\xi_4)+\nu_{\ell_1\ell_2\ell_3\ell_4}^{R}(k,\xi_1,\xi_2,\xi_3,\xi_4).\end{equation}
The singular part collects all singular terms in the decomposition of the generalized eigenfunctions
\begin{equation}
\nu^{S}_{\ell_1\ell_2\ell_3\ell_4}(k,\xi_1,\xi_2,\xi_3,\xi_4) = \nu^{+}_{\ell_1\ell_2\ell_3\ell_4}(k,\xi_1,\xi_2,\xi_3,\xi_4)+\nu^{-}_{\ell_1\ell_2\ell_3\ell_4}  (k,\xi_1,\xi_2,\xi_3,\xi_4)
\end{equation}
with 
\begin{align}\label{sing-mubeta}
& \nu^{\ast}_{\ell_1\ell_2\ell_3\ell_4}(k,\xi_1,\xi_2,\xi_3,\xi_4)
 \\
 & = b_\ast \int \Lambda_{\ast}(x)\overline{\psi_{S,\ast}(x,k)}\psi^{(\ell_1)}_{S,\ast}(x,\xi_1)\psi^{(\ell_2)}_{S,\ast}(x,\xi_2)\psi^{(\ell_3)}_{S,\ast}(x,\xi_3)\psi^{(\ell_4)}_{S,\ast}(x,\xi_4)dx,\,\, \ast\in \setsign\end{align}
where $\Lambda_{\ast} =\sigma_{\ast}^5$. 

Notice that compared with the regular part of our model problem, additional terms appearing here are
\begin{align}\label{regularadd-mubeta}
\begin{split}
 & \nu^{R,\ast}_{\ell_1\ell_2\ell_3\ell_4}(k,\xi_1,\xi_2,\xi_3,\xi_4)
 \\
 & = \int \Lambda_{\ast}(x)(\beta_4(x)- b_\ast 
 )
 \overline{\psi_{S,\ast}(x,k)}\psi^{(\ell_1)}_{S,\ast}(x,\xi_1)
 \psi^{(\ell_2)}_{S,\ast}(x,\xi_2)\psi^{(\ell_3)}_{S,\ast}(x,\xi_3)\psi^{(\ell_4)}_{S,\ast}(x,\xi_4)dx.
\end{split}
\end{align}

Now, define the set
\[Y_R=\{\mathbf{E}=(E_0,...,E_4):\exists \,k\in \{0,...,4\}\,\,\mbox{with}\,\, E_k=R\}\subset \{S,R\}^5.\]
 The regular part can also be decomposed as 
\begin{align}
 \nu^{R}_{\ell_1\ell_2\ell_3\ell_4} 
 & =
 \nu^{R,1}_{\ell_1\ell_2\ell_3\ell_4} 
 +\nu^{R,2}_{\ell_1\ell_2\ell_3\ell_4} 
 +\nu^{R,+}_{\ell_1\ell_2\ell_3\ell_4} 
 +\nu^{R,-}_{\ell_1\ell_2\ell_3\ell_4} 
\end{align}
where 
\begin{align}\label{mur1beta}
\nu^{R,1}_{\ell_1\ell_2\ell_3\ell_4}(k,\xi_1,\xi_2,\xi_3,\xi_4)=\sum_{\mathbf{E}\in Y_R}\int \beta_4(x)\overline{\psi_{E_0}(x,k)}\psi_{E_1}^{(\ell_1)}(x,\xi_1)\psi_{E_2}^{(\ell_2)}(x,\xi_2)\psi_{E_3}^{(\ell_3)}(x,\xi_3)\psi_{E_4}^{(\ell_4)}(x,\xi_4)dx
\end{align}
and
\begin{align}\label{mur2beta}
\nu^{R,2}_{\ell_1\ell_2\ell_3\ell_4}(k,\xi_1,\xi_2,\xi_3,\xi_4)
=\int\beta_4(x)\overline{\psi_{S}(x,k)}
\psi_{S}^{(\ell_1)}(x,\xi_1)\psi_{S}^{(\ell_2)}(x,\xi_2)\psi_{S}^{(\ell_3)}(x,\xi_3)
\psi_{S}^{(\ell_4)}(x,\xi_4)dx
\\
-  \nu^{S}_{\ell_1\ell_2\ell_3\ell_4}(k,\xi_1,\xi_2,\xi_3,\xi_4)-  \nu^{ R,+}_{\ell_1\ell_2\ell_3\ell_4}(k,\xi_1,\xi_2,\xi_3,\xi_4)-   \nu^{ R,-}_{\ell_1\ell_2\ell_3\ell_4}(k,\xi_1,\xi_2,\xi_3,\xi_4).
\end{align}
Observe that the integrand in $\nu^{R,1}$ contains at least one $\psi_R$ which decays faster  to zero as $|x|\rightarrow \infty$, while the integrand in $\nu^{R,2}$
contains terms with compact support.  Thus, \eqref{NDFbeta} implies the following decomposition
\begin{align}\label{NDFbeta'}
\begin{split}
&(2\pi)^{5/2}\nu = 
     \nu^{-}_{\ell_1,\ell_2,\ell_3,\ell_4}+\nu^{+}_{\ell_1,\ell_2,\ell_3,\ell_4}+\nu^{R,1}_{\ell_1,\ell_2,\ell_3,\ell_4}+\nu^{R,2}_{\ell_1,\ell_2,\ell_3,\ell_4}+\nu^{R,+}_{\ell_1,\ell_2,\ell_3,\ell_4}+\nu^{R,-}_{\ell_1,\ell_2,\ell_3,\ell_4}.
\end{split}
\end{align}

From \eqref{NDFbeta'}, we also decompose 
\begin{align}\label{decomposed-Qbeta}
\begin{split}
 & 2^4e^{-is \langle k\rangle}\wt{\mathcal{F}}(\beta_4 v^4)(k)
 \\
  & =(2\pi)^{-5/2}[I_{-}(s,k)+I_{+}(s,k)+I_{R,1}(s,k)+I_{ R,2}(s,k)+I_{R,+}(s,k)+I_{R,-}(s,k)],   
\end{split}
\end{align}
where each term on the right-hand side is defined with $\ast\in \{-,+,(R,1),(R,2),(R,-),(R,+)\}$ by
\begin{align}
 I_{\ast}(s,k) :=\Suml{}{}{}{}  \ell_1\ell_2\ell_3\ell_4 \iiiint \frac{e^{is\Phi_{\ell_1\ell_2\ell_3\ell_4}}}{\prod_{j=1}^4\lr{\xi_j}}\wt{f}^{(\ell_1)}(\xi_1)\wt{f}^{(\ell_2)}(\xi_2)\wt{f}^{(\ell_3)}(\xi_3)\wt{f}^{(\ell_4)}(\xi_4)
 \\
 \times \nu^\ast_{\ell_1\ell_2\ell_3\ell_4}(k,\xi_1,\xi_2,\xi_3,\xi_4)d\xi_{1,2,3,4}.
\end{align}
Next, define 
\begin{align}
\mathsf{Q}_{m,\ast}(k) :=\Suml{}{}{}{}  \ell_1\ell_2\ell_3\ell_4 \int_0^\infty\iiiint \frac{e^{is\Phi_{\ell_1\ell_2\ell_3\ell_4}}}{\prod_{j=1}^4\lr{\xi_j}}\wt{f}^{(\ell_1)}(\xi_1)\wt{f}^{(\ell_2)}(\xi_2)\wt{f}^{(\ell_3)}(\xi_3)\wt{f}^{(\ell_4)}(\xi_4)\\
 \times \nu^\ast_{\ell_1\ell_2\ell_3\ell_4}(k,\xi_1,\xi_2,\xi_3,\xi_4)\tau_m(s)d\xi_{1,2,3,4}ds.
\end{align}
From \eqref{decomposed-Qbeta} and taking the distorted Fourier transform of $\mathsf{Q}_m$ in \eqref{Qbeta}, we can write
\begin{equation}\label{Qm-decomposedbeta}
C \, \wt{\mathsf{Q}_{m}}(k) 
 = \mathsf{Q}_{m,-}(k)+\mathsf{Q}_{m,+}(k)
  +\mathsf{Q}_{m,R,1}(k) +\mathsf{Q}_{m,R,2}(k)+\mathsf{Q}_{m,R,+}(k) +\mathsf{Q}_{m,R,-}(k).   
\end{equation}
for some unimportant constant $C>0$.
The analysis of the terms $\mathsf{Q}_{m,\ast}(k)$ 
can be carried out  in the same way as in Section \ref{sec:PDE-part2} resulting in the following:

\begin{lemma}
Adopting the same notation as in the previous sections, assume that 
the bootstrap conditions \eqref{rho-assumption-bootstrap-pde}
and \eqref{w-assumptions-bootstrap-pde} hold.
Then, for all $m\in \mathbb Z\cap [0,L+1]$, 
\begin{equation}
\|  \partial_k\wt{\mathsf{Q}_{m}}(\cdot)\|_{L^2_k}\lesssim \varepsilon_0^{3\al}2^{8\delta m}.  
\end{equation}

\end{lemma}

Using this lemma, we conclude this part with the main result  
on the asymptotic stability of the kink solution \eqref{eq:kink_ode} under odd perturbations.

\begin{theorem}\label{thm:Kink}
Consider the wave equation \eqref{eq:waveeq} with an odd kink solution $K(x)$ 
that solves \eqref{eq:kink_ode} with a double-well potential $W$ satisfying \eqref{Wpm},
and the associated linearized operator as in \eqref{eq:potential}-\eqref{eq:klinear} satisfying the spectral conditions stated there, until \eqref{nFGRW}. 


Consider odd perturbations of $K(x)$ in the form $$\phi(t,x)=K(x)+u(t,x)$$ 
so that $u$ satisfies the equation 
\begin{equation}\label{eq:asyu}
    \partial_t^2 u-\partial_x^2 u+m^2u+V(x)u =N(K,u)
\end{equation}
where $m$  and $V$ are given by \eqref{Wpm} and \eqref{eq:potential}. 

There exists a sufficiently small $\varepsilon\in (0,1)$ 
such that for all $0<\varepsilon_0\leq \varepsilon$, if the initial data $(u_0, u_1)$ satisfies 
\begin{align}
& |(\varphi_\omega,u_0)| + |(\varphi_\omega,u_1)| = \varepsilon_0,
\\
& \left\| (L 
  P_c u_0, P_c u_1) \right\|_{H^{N-1}_x}
+ \left\| (\lr{x} L P_c u_0, \lr{x} P_c u_1)\right\|_{H^2_x} \leq \varepsilon_0,
\end{align}
for $N$ sufficiently large
(see \eqref{choice-of-csts})
then the equation \eqref{eq:asyu} admits a unique 
global solution $u\in C(\R,H^N) \cap C^1(\R,H^{N-1})$ with the following properties:

\begin{enumerate}

\smallskip
\item $u$ can be decomposed
into discrete and radiation components
 \begin{equation}\label{eq:decompu-asy}
u(t,x) = a(t) \varphi_\omega (x) + v(t,x), \quad (v(t), \varphi_\omega)=0
\end{equation}
with
 \begin{equation}
\label{main-result-bound-a-asy}
\sqrt{\dot{a}^2(t) + \omega^2 a^2(t)} \approx \varepsilon_0 (1+\varepsilon_0^2 t)^{-1/2},\quad t>0
\end{equation}
 and 
 \begin{equation}\label{v-est-mainthm-asy}
 \|v(t)\|_{L^\infty_x}\lesssim \frac {\log(2+\varepsilon_0^2 t)}{\lr{t}^{1/2}}, 
 \quad t\geq\varepsilon_0^{-2}.   
 \end{equation}
 
\smallskip
\item  The field component $w=\partial_tv+iLv$ scatters to a linear 
 Klein-Gordon wave in $H^{N-1}_x$ in the sense that there exists an odd $w_{\infty}\in P_cH^{N-1}_x$ 
 such that
\begin{equation}\label{scattering-w-asy}
\|w(t) - e^{iLt}w_{\infty}\|_{H^{N-1}_x} \lesssim  
  \varepsilon_0 (1+\varepsilon_0^2 t)^{-\sigma}, \quad t>0
\end{equation}
for some $\sigma>0$.
\end{enumerate}

In particular, this kink solution $K(x)$ is asymptotically stable under odd perturbations 
on the whole line. 
More refined information on $u$ solving \eqref{eq:asyu} 
can also be derived analogously to Theorem \ref{thm:mainthmdetail} and Theorem \ref{thm:asympA}. 
\end{theorem}
 





\begin{thebibliography}{100}
\normalsize

\bibitem{A75} Shmuel Agmon, Spectral properties of Schr\"{o}dinger operators and scattering theory. Ann. Scuola Norm.
Sup. Pisa Cl. Sci. 4 (1975) no. 2, pp. 151-218.



\bibitem{GY}  Galtbayar Artbazar, and Kenji Yajima.  $L^p$-continuity of wave operators for one dimensional Schr\"odinger operators. Journal of Mathematical Sciences-University of Tokyo 7.2 (2000): 221-240.
 
\bibitem{BC11} Dario Bambusi and Scipio Cuccagna, On dispersion of small energy solutions to the nonlinear Klein-Gordon equation with a potential. Amer. J. Math. 133 (2011), no. 5, pp. 1421-1468.

\bibitem{BS03} Vladimir S. Buslaev and Catherine Sulem, On asymptotic stability of solitary waves for nonlinear Schr\"odinger equations. Annales de l’Institut Henri Poincar\'e C 20 (2003), no. 3, pp. 419–475.

\bibitem{C85} Thierry Cazenave, Uniform estimates for solutions of nonlinear Klein-Gordon equations. Journal of Functional Analysis 60 (1985) pp. 36-55.

\bibitem{CL24} Gong Chen and Jonas Luhrmann, Asymptotic stability of the sine-Gordon kink. arXiv preprint arXiv:2411.07004 (2024).

\bibitem{CP22} Gong Chen and Fabio Pusateri, The 1d nonlinear Schr\"odinger equation with a weighted $L^1$ potential. Analysis \& PDE 15 (2022), no. 4, pp. 937-982.

\bibitem{CP24} Gong Chen and Fabio Pusateri, On the 1d cubic NLS with a non-generic potential. Communications in Mathematical Physics 405  (2024), 35.

\bibitem{GC21} Gong Chen, Long-time dynamics of small solutions to 1$ d $ cubic nonlinear Schr\" odinger equations with a trapping potential. arXiv preprint arXiv:2106.10106 (2021).

\bibitem{C08} Scipio Cuccagna, On asymptotic stability in 3D of kinks for the $\phi^4$ model. Trans. Amer. Math. Soc. 360 (2008), no. 5, pp. 2581–2614.

\bibitem{CM20} Scipio Cuccagna and Masaya Maeda, A survey on asymptotic stability of ground states of nonlinear Schr\"odinger equations II. Discrete Contin. Dyn. Syst. Ser. S 14 (2021), no. 5, pp. 1693-1716.

\bibitem{CM25} Scipio Cuccagna and Masaya Maeda, The asymptotic stability on the line of ground states of the pure power NLS with $0 <|p-3| \ll 1$. J. Functional Analysis 288 (2025), no. 11, 110861.


\bibitem{DT} Percy Deift and  Eugene Trubowitz, Inverse scattering on the line.Comm. Pure Appl. Math.32(1979), no.2, 121–251.
\bibitem{DM22} Jean-Mart Delort and Nader Masmoudi, Long time dispersive estimates for perturbations of a kink solutions of one dimensional cubic wave equations. Mem. Eur. Math. Soc. 1, EMS Press, Berlin (2022), ix+280pp.

\bibitem{DS88} 
Nelson J. Dunford and Jacob T. Schwartz, 
Linear operators. Part II. Spectral theory. Selfadjoint operators in
Hilbert space. Reprint of the 1963 original. \textit{Wiley Classics Library}. 
A Wiley-Interscience Publication. John Wiley \& Sons, Inc., New York (1988).



\bibitem{G25} Pierre Germain, A review on asymptotic stability of solitary waves in nonlinear dispersive problems in dimension one. Journ\'ees \'equations aux d\'eriv\'ees partielles (2025), no. 6, 22 pages. 

\bibitem{GP22-mathpi} Pierre Germain and Fabio Pusateri, Quadratic Klein-Gordon equations with a potential in one dimension. Forum Math. Pi. 10 (2022), no. e17, 172 pages.

\bibitem{GPR18} Pierre Germain, Fabio Pusateri and Frédéric Roussett, The nonlinear Schr\"odinger equation with a potential in dimension 1. Ann. Inst. H. Poincar\'e - AN 35 (2018), no. 6, pp. 1477-1530.

\bibitem{GPZ23} Pierre Germain, Fabio Pusateri and Katherine Zhang, On 1d quadratic Klein-Gordon equation with a potential and symmetries. Arch. Ration. Mech. Anal. 247, art. 17 (2023), 39 pages.

\bibitem{KP23} Adilbek Kairzhan, Fabio Pusateri, Asymptotic stability near the soliton for quartic Klein-Gordon in 1D. Pure Appl. Anal. 5 (2023), no. 4, pp. 795-832.

\bibitem{KMM17} Michal Kowalczyk, Yvan Martel and Claudio Mu\~noz, Kink dynamics in the $\phi^4$ model: Asymptotic stability for odd perturbations in the energy space. J. Amer. Math. Soc. 30 (2017) no. 3, pp. 769--798.

\bibitem{KMM17-review} Michal Kowalczyk, Yvan Martel and Claudio Mu\~noz, On asymptotic stability of nonlinear waves. S\'eminaire Laurent Schwartz—EDP et applications (2017) pp. 1-27.

\bibitem{KKK} Alexander I. Komech, Elena A. Kopylova and Sergey A. Kopylov,  On nonlinear wave equations with parabolic potentials. Journal of Spectral Theory 3 (2013) no. 4, pp. 485--503.

\bibitem{KK11}
Alexander I. Komech, Elena A. Kopylova, On asymptotic stability of kink for relativistic Ginzburg–Landau equations. ARMA 202 (2011) no. 1, pp. 213--245.

\bibitem{KNS12}
Joachim Krieger, Kenji Nakanishi and Wilhelm Schlag, 
Global dynamics above the ground state energy for the one-dimensional NLKG equation. 
\textit{Math. Z.} 272 (2012) no. 1-2, pp. 297-316.

\bibitem{LPPSS1} J. L\"{u}hrmann, J. Palacios, F. Pusateri, W. Schlag, S. Shahshahani, 
 Asymptotic stability of the degree-one vortex in the abelian Yang-Mills-Higgs model. arXiv preprint arXiv:2608.10610 (2026).
 
\bibitem{LPPSS2} J. L\"{u}hrmann, J. Palacios, F. Pusateri, W. Schlag, S. Shahshahani,  Asymptotic stability of the degree-one vortex in the abelian Yang-Mills-Higgs model: Linear Theory. arXiv Preprint arXiv:2608.10609 (2026).

\bibitem{LPPSS3} J. L\"{u}hrmann, J. Palacios, F. Pusateri, W. Schlag, S. Shahshahani, Asymptotic stability of the degree-one vortex in the abelian Yang-Mills-Higgs model: Spectral Theory and Numerics. arXiv  Preprint arXiv:2608.10608 (2026).

\bibitem{M24} Yvan Martel, Asymptotic stability of small standing solitary waves of the one-dimensional cubic-quintic Schr\"odinger equation. Inventiones mathematicae 237 (2024), no. 3, pp. 1253-1328.

\bibitem{PKA98} Dmitry E. Pelinovsky, Yuri S. Kivshar and Vsevolod V. Afanasjev, Internal Modes of Envelope Solitons. Physica D 116 (1998) pp. 121-142.

\bibitem{LP21} Tristan L\'eger and Fabio Pusateri, Internal modes and radiation damping for quadratic Klein-Gordon in 3d. arXiv preprint arXiv:2112.13163 (2021).

\bibitem{LP22} Tristan L\'eger and Fabio Pusateri, Internal mode-induced growth in 3d nonlinear Klein–Gordon equations. Atti Accad. Naz. Lincei Cl. Sci. Fis. Mat. Natur. 33 (2022), no. 3, pp. 695-727.

\bibitem{LLY22} Z. Lei, J. Liu, and Z. Yang, Energy Transfer, Weak Resonance, and Fermi’s Golden Rule in Hamiltonian Nonlinear Klein-Gordon Equations. arXiv reprint arXiv:2201.06490 (2022).

\bibitem{LLY23} Z. Lei, J. Liu, and Z. Yang, Energy transfer and radiation in Hamiltonian nonlinear Klein-Gordon equations: general case. arXiv preprint arXiv:2307.16191. (2023).

\bibitem{R25} Guillaume Rialland, Asymptotic stability of solitons for near-cubic NLS equation with an internal mode. Discrete Contin. Dyn. Syst. 45 (2025), no. 5, pp. 1591-1652.

\bibitem{S93} Israel M. Sigal, Nonlinear wave and Schr\"odinger equations I: Instability of periodic and quasiperiodic solutions. Commun. Math. Phys. 153 (1993), no. 2, pp. 297-320.

\bibitem{SW99} Avy Soffer and Michael I. Weinstein, Resonances, radiation damping and instability in Hamiltonian nonlinear wave equations. Invent. Math. 136 (1999), no. 1, pp. 9-74.


\bibitem{TY02} T.-P. Tsai and H.-T. Yau, Asymptotic dynamics of nonlinear Schr\"odinger equations: resonance-dominated and dispersion-dominated solutions. Comm. Pure Appl. Math. 55 (2002), no. 2, pp. 153-216.

\bibitem{Y10} Dmitri Yafaev, Mathematical scattering theory. Analytic theory. Mathematical Surveys and Monographs
158, AMS, Providence, RI (2010), xiv+444 pp.

\bibitem{Wed} Ricardo Weder, The $W^{k,p}$-continuity of the Schr\"odinger
wave operators on the line. \emph{Comm. Math. Phys.} 208 (1999), no.
2, 507\textendash 520. 


\end{thebibliography}
\end{document}